\documentclass[leqno,tbtags]{amsart}
\usepackage{amssymb,amsmath,amsthm,latexsym} 
\usepackage{subfigure}
\usepackage[pdftex]{color,graphicx}
\usepackage[all]{xy}
\usepackage{rotating}
\usepackage{mathtools}

\usepackage{hyperref}

\xyoption{all}
\xyoption{line}
\xyoption{knot}
\xyoption{color}

\newcommand{\hfkhat}[0]{\widehat{HFK}}
\newcommand{\hfkmin}[0]{HFK^-}

\newcommand{\noi}[0]{\noindent}
\newcommand{\isom}[0]{\cong}

\newcommand{\cA}[0]{\mathcal{A}}
\newcommand{\cB}[0]{\mathcal{B}}

\newcommand{\cE}[0]{\mathcal{E}}
\newcommand{\cG}[0]{\mathcal{G}}

\newcommand{\cR}[0]{\mathcal{R}}

\newcommand{\itk}[0]{\mathbb{F}}

\newcommand{\kpoly}[0]{\itk[\underline{x}]}
\newcommand{\kpolylong}[0]{\itk[x_0,\ldots,x_n]}

\newcommand{\weight}[0]{\textbf{\textrm{w}}}
\newcommand{\Weight}[0]{\textbf{\textrm{W}}}
\newcommand{\out}[0]{\mathrm{out}}
\newcommand{\into}[0]{\mathrm{in}}

\newcommand{\ztau}[0]{z_{\tau}}
\newcommand{\zbeta}[0]{z_{\beta}}
\newcommand{\overf}{\overline{f}}
\newcommand{\overg}{\overline{g}}

\newcommand{\nutopk}[0]{\nu\ztau\ksup{k+1}-\ztau\ksup{k+1}}

\newcommand{\anyzeta}[0]{\zeta\ksup{k+1}}

\newcommand{\lt}[1]{\mathrm{LT}\!\left(#1\right)}
\newcommand{\trt}[1]{\mathrm{TT}\!\left(#1\right)}
\newcommand{\lm}[1]{\mathrm{LM}\!\left(#1\right)}
\newcommand{\lc}[1]{\mathrm{LC}\!\left(#1\right)}
\newcommand{\tc}[1]{\mathrm{TC}\!\left(#1\right)}
\newcommand{\trm}[1]{\mathrm{TM}\!\left(#1\right)}
\newcommand{\lcm}[1]{\mathrm{lcm}\!\left(#1\right)}
\newcommand{\tor}[0]{\mathrm{Tor}}

\newcommand{\xsub}[2]{x_{#1,#2}}
\newcommand{\zsub}[2]{z_{#1,#2}}
\newcommand{\ksup}[1]{^{(#1)}}
\newcommand{\gm}[0]{\Gamma}
\newcommand{\dl}[0]{\Delta}
\newcommand{\gd}[1]{\Gamma #1 \Delta}
\newcommand{\sm}[0]{\setminus}

\theoremstyle{definition}\newtheorem{dfn}{Definition}[section]
\newtheorem{eg}{Example}[section]
\newtheorem{algorithm}{Algorithm}[section]
\theoremstyle{plain}\newtheorem{thm}[]{Theorem}[section]
\newtheorem{cor}[thm]{Corollary}

\newtheorem{prop}{Proposition}[section]
\newtheorem{lemma}[prop]{Lemma}
\newtheorem{obs}[prop]{Observation}
\newtheorem*{thm*}{Theorem}
\newtheorem*{cor*}{Corollary}
\newtheorem*{prop*}{Proposition}
\newtheorem*{conj*}{Conjecture}
\numberwithin{equation}{section}

\begin{document}
\title[Framed graphs and the non-local ideal]{Framed graphs and the non-local ideal in the knot Floer cube of resolutions}
\author{Allison Gilmore}
\thanks{The author was partially supported by NSF grant DMS-1103801.}

\begin{abstract}
This article addresses the two significant aspects of Ozsv\'ath and Szab\'o's knot Floer cube of resolutions that differentiate it from Khovanov and Rozansky's HOMFLY-PT chain complex: (1) the use of twisted coefficients and (2) the appearance of a mysterious non-local ideal. Our goal is to facilitate progress on Rasmussen's conjecture that a spectral sequence relates the two knot homologies. We replace the language of twisted coefficients with the more quantum topological language of framings on trivalent graphs. We define a homology theory for framed trivalent graphs with boundary that -- for a particular non-blackboard framing -- specializes to the homology of singular knots underlying the knot Floer cube of resolutions. For blackboard framed graphs, our theory conjecturally recovers the graph homology underlying the HOMFLY-PT chain complex. We explain the appearance of the non-local ideal by expressing it as an ideal quotient of an ideal that appears in both the HOMFLY-PT and knot Floer cubes of resolutions. This result is a corollary of our main theorem, which is that closing a strand in a braid graph corresponds to taking an ideal quotient of its non-local ideal. The proof is a Gr\"obner basis argument that connects the combinatorics of the non-local ideal to those of Buchberger's Algorithm.
\end{abstract}

\maketitle

\section{Introduction}
\label{intro}


This article aims to elucidate the key differences between Ozsv\'ath and Szab\'o's cube of resolutions chain complex for knot Floer homology~\cite{ozsszcube} and the cube of resolutions chain complex underlying Khovanov and Rozansky's HOMFLY-PT homology~\cite{kr2,rasmussenonkr}. Comparing the constructions is especially interesting in light of the conjecture~\cite{dgr,rasmussenonkr} that there should be a spectral sequence from HOMFLY-PT homology to knot Floer homology. In both constructions, a knot in $S^3$ is studied by considering the collection of graphs $G_I$ for $I\in\left\{0,1\right\}^n$ obtained by replacing each crossing in an $n$-crossing braid diagram with its oriented resolution or with a thick edge, as in Figure~\ref{resolutions}. The graphs are planar and trivalent, with one thick and two thin edges incident to each vertex. They are equivalent to singular knots by exchanging thick edges for 4-valent vertices as in Figure~\ref{4valwideedgeexchange}.

\begin{figure}[t]
\begin{center}
$$\input{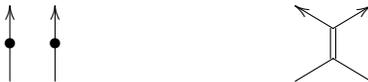}$$
\caption{A collection of graphs is obtained from a braid diagram for a knot by replacing each crossing with either its oriented resolution (left) or with a thick edge (right).}
\label{resolutions}
\end{center}
\end{figure}

\begin{figure}
\begin{center}
$$\input{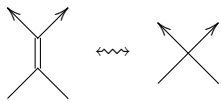}$$
\caption{Graphs, as defined in Section~\ref{framedgraphs}, correspond to singularized links via the exchange above.}
\label{4valwideedgeexchange}
\end{center}
\end{figure}

In the cube of resolutions for knot Floer homology, one associates a graded algebra $\cB_{\textrm{HFK}}(G_I)$ to each graph and assembles these into a bigraded chain complex whose homology is the knot Floer homology of the original knot. In HOMFLY-PT homology, one associates a bigraded chain complex to each $G_I$, then assembles these into a triply-graded chain complex. The triply-graded complex has one differential coming from the complexes associated to the $G_I$, but it is also given a new differential. Taking homology with respect to each of these in turn produces the HOMFLY-PT homology of the knot. Let $\cB_{\textrm{KR}}(G_I)$ denote the homology of the chain complex associated to $G_I$. 

The process of assembling a final chain complex from the $\cB_{\textrm{HFK}}(G_I)$ or the $\cB_{\textrm{KR}}(G_I)$ is quite similar; it is a standard cube of resolutions construction. We focus here on the differences between $\cB_{\textrm{HFK}}$ and $\cB_{\textrm{KR}}$, which we call the knot Floer graph homology and the HOMFLY-PT graph homology, respectively. 

Both the knot Floer and HOMFLY-PT graph homologies are built from certain ideals in polynomial rings. The polynomial rings are edge rings: they have indeterminates corresponding to thin edges of the graph. Generating sets for the ideals can be read off directly from the graph. Ozsv\'ath and Szab\'o observe that the ideals used in the two constructions are remarkably similar~\cite{ozsszcube}, but a precise relationship between the constructions has not been previously described. Our goal will be to make the comparison precise, with the intention of facilitating progress on the spectral sequence conjecture. 

We address two major differences between the knot Floer and HOMFLY-PT graph homologies.
\begin{enumerate}
\item \textbf{Twisted coefficients.} The knot Floer edge ring is defined over $\mathbb{Z}[t^{-1},t]]$, the ring of Laurent series in $t$, while the HOMFLY-PT edge ring is defined over $\mathbb{Q}$~\cite{kr2,rasmussenonkr} or $\mathbb{Z}$~\cite{krasner}. The variable $t$ appears in the definition of the knot Floer ideals as well because the knot Floer graph homology is in fact the singular knot Floer homology~\cite{ozsszstipsing} of the graph in $S^3$, computed with a particular choice of twisted coefficients.

\item \textbf{The non-local ideal.} The HOMFLY-PT graph homology is built from two ideals, $L(G)$ and $Q(G)$, both of which are specified entirely by local information (individual thick edges and their incident thin edges) in the graph. The knot Floer graph homology uses (twisted analogues of) these ideals, but also a non-local ideal $N(G)$, which cannot in general be specified by only local data from the graph.
\end{enumerate}

We address the issue of twisted coefficients by recasting it in terms of framed graphs. For a framed, planar, trivalent graph $G$, possibly with boundary, we define an edge ring $\cE(G)$, which is itself a quotient of a polynomial ring by an ideal $F(G)$ derived from the framing. We define mild generalizations of the ideals $L(G)$, $Q(G)$, and $N(G)$ mentioned above, and a graph homology
\[\cB\!\left(G\right)=\tor_\ast\!\left(\frac{\cE\!\left(G\right)}{L\!\left(G\right)},\frac{\cE\!\left(G\right)}{N\!\left(G\right)}\right)\otimes\Lambda^\ast V_G,\]
where $V_G$ is the free $\cE(G)$-module spanned by certain connected components of $G$.

If $G$ is a closed braid graph (i.e.,~obtained by replacing the crossings in a closed braid diagram with thick edges) with its outermost strand cut, then we recover the knot Floer and (conjecturally) HOMFLY-PT graph homologies by imposing certain framings. For a particular non-blackboard framing {\tt e}, we have $\cB_{\textrm{HFK}}(G)\isom\cB\!\left(G^{\tt e}\right)$. For a different non-blackboard framing, we recover the variant on the knot Floer graph homology considered in~\cite{reidmoves}.

Letting {\tt b} denote the blackboard framing, one may write the HOMFLY-PT graph homology as \[\cB_{\textrm{KR}}(G)\isom\tor_\ast\!\left(\frac{\cE(G^{\tt b})}{L(G^{\tt b})},\frac{\cE(G^{\tt b})}{Q(G^{\tt b})}\right)\otimes\Lambda^\ast V_G,\] but one may also re-state Conjecture 1.3 of \cite{manolescucube} as 
\[\tor_\ast\!\left(\frac{\cE(G^{\tt b})}{L(G^{\tt b})},\frac{\cE(G^{\tt b})}{Q(G^{\tt b})}\right)\isom\tor_\ast\!\left(\frac{\cE(G^{\tt b})}{L(G^{\tt b})},\frac{\cE(G^{\tt b})}{N(G^{\tt b})}\right).\]
If that conjecture holds, it would follow immediately that $\cB_{\textrm{KR}}(G)\isom\cB(G^{\tt b})$. That is, our $\cB$ would specialize to the HOMFLY-PT graph homology for the blackboard framing.

The approach via framed graphs is a modest generalization of existing graph homologies, but it situates these graph homologies in the context of quantum topology. In that setting, invariants of framed graphs are a natural extension of invariants of knots, and a typical stop along the way to invariants of 3-manifolds. It should be possible to extend $\cB$ to an invariant of knotted framed trivalent graphs via a cube of resolutions chain complex. It would be interesting to relate the resulting invariant to Viro's \cite{viroquantumrel} quantum relative of the Alexander polynomial, which draws on the representation theory of the quantum supergroup $\mathfrak{gl}(1\vert 1)$ to extend the multivariable Alexander polynomial to knotted framed trivalent graphs. Understanding such a relationship could help fill gaps in both the categorified and decategorified settings. On the categorified side, one might hope to extend knot Floer homology to tangles without appealing to  bordered sutured theory~\cite{zarev}. On the decategorified side, Heegaard Floer homology might suggest how to upgrade Viro's invariant of framed graphs to a $\mathfrak{gl}(1\vert 1)$ invariant of closed 3-manifolds. 

These advertisements for the framed graphs approach aside, our main result concerns the non-local ideal $N(G)$. It will be clear from the definitions that $Q(G)\subseteq N(G)$ for any graph $G$. Furthermore, the non-local ideal $N(G)$ coincides with the local ideal $Q(G)$ when $G$ is a braid graph with none of its strands closed; that is, when $G$ can be obtained from a braid with none of its strands closed by replacing crossings with thick edges as in Figure~\ref{resolutions} (see~\cite[Proposition~3.1.1]{gilmorethesis} and~\cite[Proposition~5.4]{manolescucube}, or implicitly \cite[Lemma~3.12]{ozsszcube} and \cite[Proposition~3.1]{reidmoves}). It is only as we close strands of the braid graph that we begin to see examples in which $Q(G)\subsetneq N(G)$. Therefore, we study the partially closed braids $G, G\ksup{1},\ldots,G\ksup{b-1}$ obtained by closing one strand at a time, as in Figure~\ref{intermedgraphs}. We allow any framing on $G$, and assume that the framing on $G\ksup{i}$ is inherited from that on $G$. Throughout, we consider $G\ksup{b-1}$ to be the closure of $G$, even though its outermost strand is still open. This is consistent with the quantum topology approach to the Alexander polynomial (e.g.~in~\cite{viroquantumrel}) and with the use of basepoints in~\cite{ozsszcube,reidmoves,manolescucube}. 

We prove that closing a braid strand corresponds to taking an ideal quotient of the non-local ideal by the edge variable associated to the strand being closed. See Section~\ref{mainresult} for full details of the notation.
\begin{thm}
\label{idealqthmintro}
Let $G$ be a braid graph with no strands closed and $G\ksup{k}$ denote the diagram obtained by closing the right-most $k$ strands of $G$. Let $\pi_k:\cE\!\left(G\ksup{k}\right)\to\cE\!\left(G\ksup{k+1}\right)$ denote the projection of edge rings corresponding to closing the $(k+1)^{\text{st}}$ strand of $G\ksup{k}$. Let $\ztau\ksup{k+1}$ denote the edge ring variable corresponding to the top boundary edge of the $(k+1)^{\text{st}}$ strand of $G\ksup{k}$. Then for $0\leq k\leq b-2$, the equality
 $$\pi_k\!\left(N\!\left(G\ksup{k}\right)\right) : \left(\ztau\ksup{k+1}\right) = N\!\left(G\ksup{k+1}\right),$$  holds in $\cE\!\left(G\ksup{k+1}\right)$.
\end{thm}
As a corollary, we may express the non-local ideal of a braid graph's closure in terms of a local ideal of the underlying braid graph.
\begin{cor}
\label{idealqcorintro}
With notation as in Theorem~\ref{idealqthmintro}, the equality
$$\pi_{b-2}\circ\cdots\circ\pi_0\!\left(Q\!\left(G\right)\right) : \left(\ztau\ksup{1}\cdots\ztau\ksup{b-1}\right)=N\!\left(G\ksup{b-1}\right)$$ holds in $\cE\!\left(G\ksup{b-1}\right)$.
\end{cor}

The reason for the appearance of the ideal quotient in relation to braid closures remains mysterious. There is a tempting analogy to Hochschild homology, which is the closure operation in Khovanov's construction of HOMFLY-PT homology via Soergel bimodules~\cite{khovHH}. Soergel bimodules categorify the Hecke algebra and Hochschild homology categorifies Ocneanu's trace on the Hecke algebra, so Khovanov's whole construction has a clean decategorification. One might hope for a similar story involving the ideal quotient and the Alexander polynomial.

\begin{figure}[tb]
\begin{center}
\input{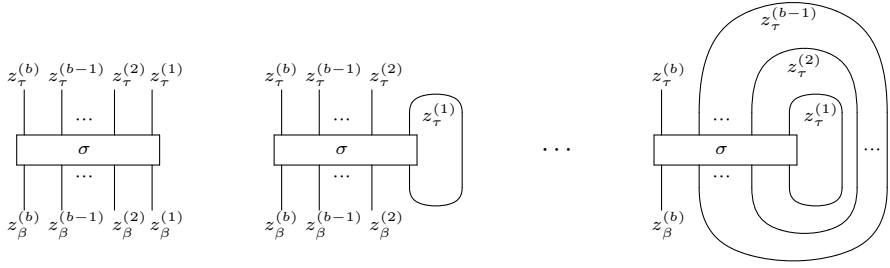}
\caption{From left to right: $G_\sigma\ksup{0}=G_\sigma$, the partial closure $G_\sigma\ksup{1}$, and the full closure $G_\sigma\ksup{b-1}=G_{\widehat{\sigma}}$.}
\label{intermedgraphs}
\end{center}
\end{figure}

With the ideal quotient result and interpretation via framed graphs in hand, we may describe the status of the HOMFLY-PT to knot Floer spectral sequence conjecture~\cite{dgr,rasmussenonkr} as follows.\footnote{All of the following results have parallels involving $\hfkhat$ and the reduced HOMFLY-PT homology.} Let $K$ be a knot in $S^3$. Let $\mathcal{K}$ be an $n$-crossing braid diagram for $K$ with outermost strand cut. Let $G_I$ for $I\in\left\{0,1\right\}^n$ be the collection of planar trivalent graphs that can be obtained by replacing each crossing of $\mathcal{K}$ with either a thick edge or with the oriented smoothing, as in Figure~\ref{resolutions}. 

Recall that {\tt e} denotes the framing for which $\cB$ specializes to the knot Floer graph homology. Ozsv\'ath and Szab\'o constructed the original cube of resolutions for knot Floer homology from the collection of $\cB\!\left(G_I^{\tt e}\right)$ for $I\in\left\{0,1\right\}^n$, which arose for them as singular knot Floer homology \cite{ozsszstipsing} with twisted coefficients. They showed that their cube of resolutions complex is the $E_1$ page of a spectral sequence to $\hfkmin(K)$ that collapses at the $E_2$ page~\cite[Theorem~1.1, Section~5]{ozsszcube}. Manolescu~\cite{manolescucube} studied the untwisted version of their construction. He described appropriate differentials and gradings from which to assemble a cube of resolutions chain complex from the $\cB\!\left(G_I^{\tt b}\right)$. He also identified that complex as the $E_1$ page of a spectral sequence to $\hfkmin(K)$~\cite[Theorem~1.1]{manolescucube}. Using yet another framing, for which $\cB$ also specializes to the knot Floer graph homology, the author has described another cube of resolutions chain complex~\cite{reidmoves}.  Like the Ozsv\'ath-Szab\'o complex, it is the $E_1$ page of a spectral sequence to $\hfkmin(K)$ that collapses at the $E_2$ page (see the proof of \cite[Proposition~9.1]{reidmoves}).

While both of the non-blackboard-framed spectral sequences mentioned above collapse at the $E_2$ page, Manolescu's blackboard-framed spectral sequence does not. In fact, he conjectures that it is exactly the desired spectral sequence from HOMFLY-PT homology to knot Floer homology. More precisely, and translating to our language, Manolescu conjectures that $\cB\!\left(G_I^{\tt b}\right)\isom\cB_{\textrm{KR}}\!\left(G_I\right)$, which would imply that the $E_2$ page of the blackboard-framed spectral sequence was the middle HOMFLY-PT homology of $K$ \cite[Conjecture~1.3]{manolescucube}. 

Corollary~\ref{idealqcorintro} allows us to rephrase Manolescu's conjecture as 
\[\tor_\ast\!\left(\frac{\cE\!\left(G_I^{\tt b}\right)}{L\!\left(G_I^{\tt b}\right)},\frac{\cE\!\left(G_I^{\tt b}\right)}{Q\!\left(G_I^{\tt b}\right):\left(\ztau\ksup{1}\cdots \ztau\ksup{b-1}\right)}\right)\isom \tor_\ast\!\left(\frac{\cE\!\left(G_I^{\tt b}\right)}{L\!\left(G_I^{\tt b}\right)},\frac{\cE\!\left(G_I^{\tt b}\right)}{Q\!\left(G_I^{\tt b}\right)}\right).\]
Theorem~\ref{idealqthmintro} suggests an inductive approach to the proof: close one strand of a braid diagram at a time and study how the corresponding ideal quotient changes the result of applying $\tor_\ast\!\left(\frac{\cE\!\left(G_I^{\tt b}\right)}{L\!\left(G_I^{\tt b}\right)},-\right)$. Theorem~\ref{idealqthmintro} and Corollary~\ref{idealqcorintro} also provide a new map to employ: the multiplication map $\frac{\cE\!\left(G_I\right)}{N\!\left(G_I\right)}\xrightarrow{\ztau\ksup{1}\cdots\ztau\ksup{b-1}}\frac{\cE\!\left(G_I\right)}{Q\!\left(G_I\right)}$. The multiplication map fits into a short exact sequence \[0\to\frac{\cE\!\left(G_I\right)}{N\!\left(G_I\right)}\xrightarrow{\ztau\ksup{1}\cdots\ztau\ksup{b-1}}\frac{\cE\!\left(G_I\right)}{Q\!\left(G_I\right)}\to\frac{\cE\!\left(G_I\right)}{Q\!\left(G_I\right)+\left(\ztau\ksup{1}\cdots\ztau\ksup{b-1}\right)}\to 0,\] which induces a long exact sequence when $\tor_\ast\!\left(\frac{\cE\!\left(G_I\right)}{L\!\left(G_I\right)},-\right)$ is applied. The multiplication map does not have the correct grading to induce Manolescu's conjectured isomorphism for all $\tor_i$, but it does have the appropriate grading to induce the isomorphism in the top degree, i.e.~when $i=b-1$ \cite[Conjecture 5.2]{manolescucube}. 

We expect that there are cube of resolutions chain complexes and spectral sequences analogous to those described above for many other compatible choices of framings on the set of $G_I$ obtained from a given knot diagram. Any choice of framings that corresponds to an admissible twisting of the singular knot Floer homology (see~\cite[Lemma~2.1]{manolescucube},\cite[Section~2.1]{ozsszcube}) should do, and that encompasses any non-negative framing. We expect all such spectral sequences to converge to knot Floer homology, and to collapse if sufficiently far from blackboard-framed. In particular, we expect that such a spectral sequence would collapse if the compatible choice of framings had the property that every closed component of every $G_I$ had non-zero total framing. It would be interesting to know under what conditions the $E_1$ and/or $E_2$ pages of such spectral sequences are knot invariants, and whether there is any relationship to HOMFLY-PT homology outside the blackboard-framed case.

Aside from the conjectured spectral sequence, it would be interesting to study $\cB(G)$ (or a suitable generalization to knotted framed graphs) as an invariant in its own right. For example, it would be interesting to know for what framings $\cB$ satisfies (perhaps modified) categorified MOY relations, as it does for {\tt b} and {\tt e}~\cite{kr2,reidmoves}.   There has also been little work done on applications of knot Floer homology to the study of singular knots or spatial graphs. As a starting point, one might look for a relationship between $\cB$ and the sutured Floer homology of a graph's complement in $S^3$. 

The proof of Theorem~\ref{idealqthmintro} is a computational commutative algebra argument. We use a Gr\"obner basis technique (Buchberger's Algorithm) to construct a generating set for the appropriate ideal quotients from the defining generating set of $N\!\left(G\ksup{k}\right)$. The result is miraculously the same as the defining generating set for $N\!\left(G\ksup{k+1}\right)$. The computational approach makes for some rather involved arguments, but ultimately succeeds because of a match between the combinatorics of Buchberger's Algorithm and those of the ideals we associate to framed graphs. We are optimistic that Gr\"obner basis techniques may prove useful for the spectral sequence conjecture or in other efforts to study $\cB$.

The paper is organized as follows. Section~\ref{dfns} makes precise the concepts and notation referenced so far: framed graphs, the framing ideal, the local and non-local ideals, edge rings, and the graph homology $\cB$. It also discusses the relation of $\cB$ to the HOMFLY-PT and knot Floer graph homologies, and computes $\cB$ in two simple cases. Finally, it establishes the notation used to state Theorem~\ref{idealqthmintro}. Section~\ref{gbbackground} is a primer on Gr\"obner basis techniques and Buchberger's Algorithm. Since the proof of Theorem~\ref{idealqthmintro} is rather technical, Section~\ref{outlinesec} gives an overview and Section~\ref{examplesection} an example illustrating the arguments to come. Sections~\ref{outlinesec} and~\ref{examplesection} also highlight the reasons that Gr\"obner basis techniques are well suited to the combinatorics of our problem. Sections~\ref{bbrd1} and~\ref{bbrd2} carry out the proof in detail for the blackboard framed case, where it is at least somewhat less notationally intensive. Section~\ref{nonbbframings} describes the modifications necessary to extend from blackboard to arbitrary framings.

\subsection*{Acknowledgements} The author thanks Ciprian Manolescu, who was the first to mention ideal quotients to her in this context, and who provided useful input on drafts of this paper. She is also grateful for several useful conversations with Mikhail Khovanov, Robert Lipshitz, Peter Ozsv\'ath, and Zolt\'an Szab\'o. Finally, the author appreciates the hospitality of the Simons Center for Geometry and Physics, where she proved a limited version of this result (see~\cite{gilmorethesis}) while a visiting student.

\section{Definitions: Framed graphs and associated algebraic objects}
\label{dfns}
\subsection{Framed graphs}
\label{framedgraphs}
In this paper, \emph{graph} will mean an oriented graph properly embedded in the disk $D^2$ with the following properties:
\begin{enumerate}
\item vertices have degree at most three;
\item every connected component has at least one vertex of degree greater than one;
\item edges have an assigned weight of one (thin) or two (thick); 
\item edges incident to univalent and bivalent vertices are thin; and
\item for trivalent vertices, the sum of weights of incoming edges equals the sum of weights of outgoing edges.
\end{enumerate}
Univalent vertices will also be called \emph{boundary vertices} of the graph and their incident edges will be called \emph{boundary edges}. All other thin edges will be called \emph{interior edges}. Graphs with these properties are equivalent to singularized projections of tangles: exchange thick edges in the graph for 4-valent vertices as in Figure~\ref{4valwideedgeexchange}. 

A \emph{framing} of a graph $G$ will mean an extension of the embedding of $G\hookrightarrow D^2$ to an embedding $F\hookrightarrow D^2\times[0,1]$, where $F$ is a compact surface with boundary and
\begin{enumerate}
\item $G=F\cap\left(D^2\times\left\{\frac{1}{2}\right\}\right)$ is a deformation retract of $F$;
\item $F\cap\partial G=\partial F\cap \partial G$; and
\item $F\cap\left(\partial D^2\times[0,1]\right)=F\cap\left(\partial D^2\times\left\{\frac{1}{2}\right\}\right)$ with each component thereof an arc containing exactly one univalent vertex of $G$.
\end{enumerate}
The \emph{blackboard framing} of $G$ is the surface $F_0$ obtained by taking a closed neighborhood of $G$ in $D^2\times\left\{\frac{1}{2}\right\}$. We require the thick edges in our graphs to be blackboard framed. On each edge, one may compare a framing $F$ to the blackboard framing to obtain an integer, which is the number of positive or negative half-twists that must be inserted in $F_0$ to match $F$. We will represent a framed graph diagrammatically by marking each thin edge and labeling the marking with an integer. We will omit the framing from the text notation for the graph unless discussing a property that holds only for particular framings. 

We will consider framed graphs up to planar graph
 isotopies: isotopies of the graph in $D^2\times\left\{\frac{1}{2}\right\}$ that extend to isotopies of the framing surface $F$ in $D^2\times[0,1]$ and fix its intersection with $\partial D^2\times [0,1]$. It will be clear from the definitions that $\cB$ is invariant under such isotopies. We expect that $\cB$ could be extended to an invariant of knotted framed graphs (i.e.~allow the graph to be embedded in $D^2\times[0,1]$ and require merely that isotopies fix the intersection of the graph with $\partial D^2$) using a cube of resolutions construction, but we do not pursue the point here. 

\subsection{Ideals associated to framed graphs}
\label{idealdfnsec}

We work over a ground ring $\cR=\itk[t^{-1},t]]$ of Laurent series in $t$, with $\itk$ a field.\footnote{Much of the background material on Gr\"obner bases that we use generalizes to the case where $\itk$ is a Noetherian commutative ring, but certain computability properties are required of the ring for the full theory of Gr\"obner bases to generalize.} Let $G$ be a framed graph, so each thin edge of $G$ has an orientation and a marking. Assign to each thin edge a pair of indeterminates $x_i$ and $y_i$ labeling the first and second (with respect to the orientation) segments of the thin edge. Let $\underline{x}(G)$ and $\underline{y}(G)$  denote the sets of $x_i$ and $y_i$. We consider four ideals in $\cR[\underline{x}(G),\underline{y}(G)]$ associated to the graph $G$. 
\begin{dfn}
\label{idealdfn}
\begin{enumerate}
\item[(a)] the \emph{framing ideal}, $F(G)$, is generated by linear polynomials associated to markings on thin edges
\begin{align*}
t^\ell y_i-x_i\quad\text{to}\quad
{\xy
(0,-5)*{}="b";
(0,5)*{}="t";
(0,0)*{-}="m";
{\ar "b";"t"};
(2.5,3.25)*{y_i};
(2.5,-3.5)*{x_i};
(-2,0)*{\ell};
\endxy}

\end{align*}
\item[(b)] the \emph{linear ideal}, $L(G)$, is generated by linear polynomials associated to thick edges
\begin{align*}
(x_a+x_b)&-(y_c+y_d)\quad\text{to}\quad
{\xy
(-4,-6)*{}="bl";
(4,-6)*{}="br";
(-4,6)*{}="tl";
(4,6)*{}="tr";
(0,2)*{}="tm";
(0,-2)*{}="bm";
{\ar@{-} "bl";"bm"};
{\ar@{-} "br";"bm"};
{\ar "tm";"tl"};
{\ar "tm";"tr"};
{\ar@{=} "bm";"tm"};
(-5,4)*{x_a};
(5,3.5)*{x_b};
(-5,-3.5)*{y_c};
(5,-4)*{y_d};
\endxy}

\end{align*}
\item[(c)] the \emph{quadratic ideal}, $Q(G)$, is generated by quadratic polynomials associated to thick edges
\begin{align*}
x_ax_b-y_cy_d\quad\text{to}\quad
{\xy
(-4,-6)*{}="bl";
(4,-6)*{}="br";
(-4,6)*{}="tl";
(4,6)*{}="tr";
(0,2)*{}="tm";
(0,-2)*{}="bm";
{\ar@{-} "bl";"bm"};
{\ar@{-} "br";"bm"};
{\ar "tm";"tl"};
{\ar "tm";"tr"};
{\ar@{=} "bm";"tm"};
(-5,4)*{x_a};
(5,3.5)*{x_b};
(-5,-3.5)*{y_c};
(5,-4)*{y_d};
\endxy}

\end{align*}
and linear polynomials associated to bivalent vertices
\begin{align*}
x_a-y_c\quad\text{to}\quad
{\xy
(0,-6)*{}="b";
(0,6)*{}="t";
(0,0)*{\bullet}="m";
{\ar "b";"t"};
(2.5,2.75)*{x_a};
(2.5,-3)*{y_c};
\endxy}

\end{align*}
\item[(d)] the \emph{non-local ideal}, $N(G)$, is generated by polynomials of varying degrees associated to sets of thick edges and bivalent vertices in $G$. Let $\gm$ be such a set. The \emph{weight} $\weight(\gm)$ of $\gm$ is the sum of the framings on thin edges that are internal to $\gm$ (i.e.~have both endpoints incident to a thick edge or bivalent vertex in $\gm$). Let $x_{\gm,\text{out}}$ be the product of $x_i$ associated to thin edges from $\gm$ to its complement and $y_{\gm, \text{in}}$ be the product of $y_i$ associated to thin edges into $\gm$ from its complement. Then the generator of $N(G)$ associated to $\gm$ is $$t^{\weight(\gm)}x_{\gm,\text{out}}-y_{\gm,\text{in}},$$  We consider a thin edge with a marking denoting its framing to be a single edge. 
\end{enumerate}
\end{dfn}

The definition of a subset's weight given above differs from that in~\cite{ozsszcube}, but becomes equivalent in the edge ring when the graph has the appropriate framing.  See Proposition~\ref{hfkspec}. If $G$ is a closed braid graph (obtained from a braid diagram by replacing crossings with thick edges), then the ideal $N(G)$ has other generating sets. Rather than associating a generator to each subset in $G$, we may instead associate a generator to each closed path in $G$ or to certain regions in $D^2\setminus G$. The equivalence of all of these definitions is proved in \cite[Proposition 3.1]{reidmoves}. These alternative generating sets are smaller in general, but less well adapted to the combinatorics of Buchberger's Algorithm.

\begin{dfn}
\label{edgering}
The \emph{edge ring} of $G$ is $$\cE(G)=\cR[\underline{x}(G)]\isom\frac{\cR[\underline{x}(G),\underline{y}(G)]}{F(G)},$$ fixing the isomorphism that retains the variable $x_i$ from each generator $t^\ell y_i-x_i$ of $F(G)$.\end{dfn} 
\noi Although we have defined $L(G)$, $Q(G)$, and $N(G)$ in $\cR[\underline{x}(G),\underline{y}(G)]$, we will actually work with their images in $\cE(G)$. 

In the edge ring, we have $Q(G)\subseteq N(G)$ for any $G$. The generator associated to a bivalent vertex in $Q(G)$ is the same as it is in $N(G)$, unless the incoming and outgoing edges of the bivalent vertex coincide. If they do, and that edge has framing $\ell$, then the generator of $N(G)$ is $t^\ell-1$. The generator of $Q(G)$ is $(t^\ell-1)x$ for the appropriate edge variable $x$. Let $v$ be a thick edge in $G$ and $g_{\left\{v\right\}}$ its corresponding generator in $N(G)$. If none of its incident thin edges coincide, then $g_{\left\{v\right\}}$ is identical to the generator of $Q(G)$ associated to $v$. If some of its incident thin edges do coincide, then the generator of $Q(G)$ associated to $v$ is a multiple of $g_{\left\{v\right\}}$. For example, consider the diagram from the definition of $Q(G)$ above. Suppose $x_b$ and $y_d$ label the same thin edge, and that thin edge has framing $\ell$. Let $\ell_c$ be the framing on the edge labeled by $y_c$. Then $g_{\left\{v\right\}}=t^\ell x_a-y_c$, which becomes $t^\ell x_a-t^{-\ell_c}x_c$ in the edge ring. The generator of $Q(G)$ corresponding to $v$ is $x_ax_b-y_cy_d$, which becomes $x_ax_b-t^{-\ell_c-\ell}x_cx_b=t^{-\ell}x_bg_{\left\{v\right\}}$ in the edge ring.

We allow some of the connected components of $G$ to be designated ``special.'' Let $V_G$ be the free $\cE(G)$-module spanned by the non-special connected components of $G$. Then we define the graph homology promised in the introduction to be
\begin{equation}
\cB(G)=\tor_\ast\!\left(\frac{\cE(G)}{L(G)},\frac{\cE(G)}{N(G)}\right)\otimes\Lambda^\ast V_G.
\end{equation}

\subsection{Relation of $\cB$ to other graph homologies}
\label{relationtootherssec}

As mentioned in the introduction, $\cB(G)$ specializes to the knot Floer and (conjecturally) HOMFLY-PT graph homologies for particular framings on $G$. In this section, we indicate more precisely how those specializations hold. Although the definition of $\cB(G)$ makes sense for framed graphs in the generality described in Section~\ref{idealdfnsec}, we restrict to closed braid graphs here because of the parallel restriction on the graph homologies in~\cite{ozsszcube,reidmoves,manolescucube}. Recall from the introduction that we use ``closed'' to describe a braid graph with its outermost strand cut.

\begin{prop}
\label{homflyspec}
Let $G$ be a closed braid graph with its outermost connected component (containing the cut edge) designated special. Let {\tt b} denote the blackboard framing. Then it follows from \cite[Conjecture~1.4]{manolescucube} that
\[\cB\!\left(G^{\tt b}\right)\isom\cB_{\textrm{KR}}(G),\]
where the right-hand side is the HOMFLY-PT graph homology defined in~\cite{kr2}.
\end{prop}
\begin{proof}
It is immediate from the definitions that $L$, $Q$, and $N$ (as ideals in $\cE$) are identical to those defined in \cite{manolescucube}. When $G$ is blackboard framed, every subset has weight zero and the framing ideal merely identifies each $x_i$ with its corresponding $y_i$. Our $\cB(G^{\tt b})$ and $\cB_{\textrm{KR}}(G)$ are then identical to those in \cite{manolescucube}. Theorem~1.2 of~\cite{manolescucube} is that $\cB_{\textrm{KR}}$ is the HOMFLY-PT graph homology from~\cite{kr2} and Conjecture~1.4 of~\cite{manolescucube} is that $\cB\isom\cB_{\textrm{KR}}$.
\end{proof}

\begin{prop}
\label{hfkspec}
Let $G$ be a closed braid graph with its outermost connected component (containing the cut edge) designated special. Let {\tt e} denote the framing in which each thin edge is +1 framed. Then 
\[\cB\!\left(G^{\tt e}\right)\isom\cA(G),\]
where $\cA$ is the knot Floer graph homology defined in~\cite{ozsszcube}.
\end{prop}
\begin{proof}
Treat the graph $G$ as a singular knot by replacing thick edges with 4-valent vertices as in Figure~\ref{4valwideedgeexchange}. The images in $\cE\!\left(G^{\tt e}\right)$ of the generating sets of $L\!\left(G^{\tt e}\right)$, $Q\!\left(G^{\tt e}\right)$, and $N\!\left(G^{\tt e}\right)$ are exactly the ideal generated by the relations given in the definition of $\cA(G)$ in~\cite[Section~1]{ozsszcube}. For example, the generator $x_a+x_b-y_c-y_d$ of $L\!\left(G^{\tt e}\right)$ becomes $x_a+x_b-t^{-1}x_c-t^{-1}x_d$ in $\cE\!\left(G^{\tt e}\right)$ via the elements $ty_c-x_c$ and $ty_d-x_d$ of $F\!\left(G^{\tt e}\right)$. The generator $t^{\weight(\gm)}x_{\gm,\out}-y_{\gm,\into}$ becomes $t^{\weight(\gm)}x_{\gm,\out}-t^{-\weight_{\into}(\gm)}x_{\gm,\into}$, where $\weight_{\into}(\gm)$ is the sum of the framings on the thin edges into $\gm$ from its complement and $x_{\gm,\into}$ is the product of the $x_i$ associated to those thin edges. Clearing denominators, we have $\weight(\gm)+\weight_{\into}(\gm)$ as the exponent of $t$, which is the sum of the framings on all thin edges incoming to $\gm$. When $G$ has the framing {\tt e}, that sum is twice the number of thick edges plus the number of bivalent vertices in $\gm$, which is the weight that Ozsv\'ath and Szab\'o assign to $\gm$ in~\cite{ozsszcube}. 

It follows directly from these observations that the knot Floer graph homology $\cA(G)$ defined in~\cite[Section~1]{ozsszcube} is isomorphic to $\cE\!\left(G^{\tt e}\right)/\left(L\!\left(G^{\tt e}\right)+N\!\left(G^{\tt e}\right)\right)$, which is the degree zero part of $\cB\!\left(G^{\tt e}\right)$. If $G$ is connected, then one may adapt the argument for~\cite[Theorem~3.1]{ozsszcube} to see that $\cB\!\left(G^{\tt e}\right)$ is concentrated in degree zero, so $\cB\!\left(G^{\tt e}\right)\isom\cA(G)$. If $G$ is not connected, then $\cA(G)$ vanishes. The same is true of $\cB\!\left(G^{\tt e}\right)$. The complete set of thick edges and bivalent vertices in a closed component of $G^{\tt e}$ not containing the cut strand will yield a generator of the form $t^\ell-1$ in $N\!\left(G^{\tt e}\right)$, where $\ell>0$. Since $t^\ell-1$ is a unit in the ground ring, $\cE\!\left(G^{\tt e}\right)/N\!\left(G^{\tt e}\right)$ will vanish.
\end{proof}

Finally, we define a framing for which $\cB$ specializes to the graph homology studied in~\cite{reidmoves}, which is an alternative knot Floer graph homology that also satisfies certain categorified Murakami-Ohtsuki-Yamada relations. Let $G$ be obtained (by replacing crossings with thick edges) from a braid diagram in which each crossing appears in a distinct horizontal layer. If there are $k$ thick edges between the incident thick edges of a given thin edge, assign the framing $k+1$ to that thin edge. The total framing on each strand will be equal to the number of thick edges in $G$. Call this framing {\tt s}.
\begin{prop}
\label{gilmorespec}
Let $G$ be a closed braid graph with its outermost connected component (containing the cut edge) designated special. Let {\tt s} be the framing defined above. Then 
\[\cB\!\left(G^{\tt s}\right)\isom\cA(G),\]
where $\cA$ is the alternative knot Floer graph homology defined in~\cite{reidmoves}. 
\end{prop}
\begin{proof}
The argument is almost the same as for Proposition~\ref{hfkspec}. If $G$ is disconnected, both $\cB\!\left(G^{\tt s}\right)$ and $\cA(G)$ vanish, again because a generator of the form $t^\ell-1$ appears in $N\!\left(G^{\tt s}\right)$. Otherwise, with a bit more care, the proof of \cite[Theorem~3.1]{ozsszcube} may be adapted to prove that $\cB\!\left(G^{\tt s}\right)$ is concentrated in degree zero. 

The degree zero part of $\cB\!\left(G^{\tt s}\right)$ is $\cE\!\left(G^{\tt s}\right)/\left(L\!\left(G^{\tt s}\right)+N\!\left(G^{\tt s}\right)\right)$. Compare the layered diagram used to define the framing {\tt s} with the layered diagrams studied in~\cite{reidmoves} by replace a marking denoting a framing of $k$ with $k-1$ bivalent vertices. Tracing through the definitions, one may confirm that the images of $L\!\left(G^{\tt s}\right)$, $Q\!\left(G^{\tt s}\right)$, and $N\!\left(G^{\tt s}\right)$ in $\cE\!\left(G^{\tt s}\right)$ are exactly the $L$, $Q$, and $N$ defined in~\cite{reidmoves} (after clearing denominators in some generators). Therefore, the degree zero part of $\cB\!\left(G^{\tt s}\right)$ is isomorphic to $\cA(G)$.
\end{proof}

\subsection{Examples}
\label{smallexamplesec}

\begin{eg}
\label{onecrossingunknot}
Consider the diagram with one thick edge, two boundary edges, and framing $\ell$ on the remaining thin edge. Assume that the edges labeled $x$ and $z$ were blackboard-framed and that we have already used the corresponding generators of the framing ideal to eliminate one variable associated to each.\smallskip

\hspace{.01cm}
{\xy
(-4,-6)*{}="bl";
(4,-9)*{}="br";
(-4,6)*{}="tl";
(4,9)*{}="tr";
(0,2)*{}="tm";
(0,-2)*{}="bm";
{\ar@{-} "bl";"bm"};
{\ar "tm";"tl"};
{\ar@{=} "bm";"tm"};
{"bm";"tm" **\crv{"br" & (7,0) & "tr"}};
{\ar (3.5,5.7);(3.7,5.8)};
(5.5,0)*{-};
(-5,4)*{x};
(7,5)*{w};
(-5,-4)*{z};
(7,-5)*{y};
(7.5,0)*{\ell};
\endxy}
\hspace{-.2cm}
\begin{tabular}{c p{3.2cm} | p{6.2cm}}
&
{$F=(t^\ell y-w)$

$\cE\isom\cR[w,x,z]$

$L=(x+w-z-t^{-\ell}w)$

$Q=(xw-t^{-\ell}wz)$

$N=(t^\ell x-z)$
}
& 
{$\cE/N\xrightarrow{x+w-z-t^{-\ell}w}\cE/N$

\vspace{.1cm}

$\cR[w,z]\xrightarrow{(t^{-\ell}-1)(z-w)}\cR[w,z]$

\vspace{.2cm}

$\cB\isom\begin{cases}
\cR[w,z]_{(1)}\oplus\cR[w,z]_{(0)} & \text{ if } \ell=0\\
\cR[z]_{(0)} & \text{otherwise.}\end{cases}$
}
\end{tabular}
\smallskip

\noindent On the left, we present the edge ring under the convention for retaining edge labels specified in Definition~\ref{edgering}, along with generating sets for (the images of) $L$, $Q$, and $N$ in that edge ring. On the right, we use the sole generator of $L$ to resolve $\cE/L$, then tensor with $\cE/N$ to obtain (in the first line) a chain complex with homology $\tor_\ast(\cE/L,\cE/N)$. In the second line, we simplify that chain complex using the generator of $N$ to eliminate $x$. The result in the third line follows: the map $(t^{-\ell}-1)(z-w)$ has no kernel unless $\ell=0$.
\end{eg}

\begin{eg}
\label{twocrossingexample}
Consider the diagram below. Assume that unmarked edges were blackboard-framed and that we have already eliminated one variable associated to each of them using the appropriate generators of the framing ideal. \smallskip

\hspace{.001cm}
{\xy
(-4,-6)*{
{\xy
(-4,-6)*{}="bl";
(4,-6)*{}="br";
(-4,6)*{}="tl";
(4,6)*{}="tr";
(0,2)*{}="tm";
(0,-2)*{}="bm";
{\ar@{-} "bl";"bm"};
{\ar@{-} "br";"bm"};
{\ar@{-} "tm";"tl"};
{\ar@{-} "tm";"tr"};
{\ar@{=} "bm";"tm"};
\endxy}
};
(4,6)*{
{\xy
(-4,-6)*{}="bl";
(4,-6)*{}="br";
(-4,6)*{}="tl";
(4,6)*{}="tr";
(0,2)*{}="tm";
(0,-2)*{}="bm";
{\ar@{-} "bl";"bm"};
{\ar@{-} "br";"bm"};
{\ar@{-} "tm";"tl"};
{\ar@{-} "tm";"tr"};
{\ar@{=} "bm";"tm"};
\endxy}
};
{(8,12);(8,0) **\crv{(10,14) & (14,12) & (14,0) & (14,-12) & (10,-14) & (8,-12)}};
{(0,12);(0,-12) **\crv{(1,17) & (10,19) & (16,17) & (18,0) & (16,-17) & (10,-19) & (1,-16)}};
{\ar (-8,0);(-8,12)};
(-10,4)*{x};
(-10,-11)*{v};
(3.5,1.25)*{w^\prime};
(-2.5,.5)*{\scriptstyle{\ell}};
(0,0)*{-};
(0,-2.5)*{w};
(9,14.5)*{z^\prime};
(10,11)*{\scriptstyle{m}};
(8,12)*{-};
(5,11.5)*{z};
(2.5,18.5)*{y^\prime};
(3.5,15)*{\scriptstyle{k}};
(1.25,15)*{-};
(-2,12.5)*{y};
\endxy}
\begin{tabular}{c p{7cm}}
&
{$F=(t^\ell w^\prime-w, t^k y^\prime-y, t^mz^\prime-z)$\medskip

$\cE\isom\cR[v,w,x,y,z]$\medskip

$L=(y+z-t^{-\ell}w-t^{-m}z,x+w-v-t^{-k}y)$\medskip

$Q=((y-t^{-\ell-m}w)z,xw-t^{-k}vy)$\medskip

$N=(t^{m}y-t^{-\ell}w,t^{k+\ell+m}x-v)$\medskip
}
\end{tabular}\smallskip

\noindent The generators of $L$ are a regular sequence in $\cE$. We use the Koszul complex to resolve $\cE/L$, then tensor with $\cE/N$ to obtain a complex whose homology is $\tor_\ast(\cE/L,\cE/N)$. In the second line below, we simplify by using the generators of $N$ to eliminate $w$ and $v$.
\begin{align*}
\left(\cE/N\xrightarrow{y-t^{-\ell}w+(1-t^{-m})z}\cE/N\right)&\otimes\left(\cE/N\xrightarrow{x+w-v-t^{-k}y}\cE/N\right)\\
\left(\cR[x,y,z]\xrightarrow{(t^m-1)(t^{-m}z-y)}\cR[x,y,z]\right)&\otimes\left(\cR[x,y,z]\xrightarrow{(t^{k+\ell+m}-1)(t^{-k}y-x)}\cR[x,y,z]\right)
\end{align*}

\noindent The homology of the complex above depends on whether $m$ and/or $k+\ell+m$ is zero. If one or both values is zero, then the maps in one or both tensor factors of the complex are zero. Otherwise, $t^m-1$ and/or $t^{k+\ell+m}-1$ is a unit in $\cR$ and the map involving that factor has no kernel. When both maps are non-zero, we have a Koszul complex on a regular sequence in $\cR[x,y,z]$. Therefore, the possible outcomes are as follows.

\[\cB\isom\begin{cases}
\cR[x]_{(0)} & m\neq 0, k+\ell+m\neq 0\\
\cR[x,z]_{(0)}\oplus\cR[x,z]_{(1)} & m=0, k+\ell+m\neq 0\\
\cR[x,y]_{(0)}\oplus\cR[x,y]_{(1)} & m\neq0, k+\ell+m=0\\
\cR[x,y,z]_{(0)}\oplus\cR[x,y,z]^{\oplus 2}_{(1)}\oplus\cR[x,y,z]_{(2)} & m=0, k+\ell+m=0\\
\end{cases}\]
\end{eg}

\subsection{Statement of main result}
\label{mainresult}

Our main result concerns the process of closing strands in braid graphs and the corresponding algebraic operation on the non-local ideal defined in Section~\ref{idealdfnsec}. Let $\sigma$ be a braid with $b$ strands, none of which are closed. Let $G_\sigma=G_\sigma\ksup{0}$ be the framed braid graph obtained by replacing the crossings in $\sigma$ with thick edges. Let $G_\sigma\ksup{1},\ldots,G_\sigma\ksup{b-1}$ be the intermediary graphs obtained by closing one strand of $G_\sigma$ at a time from right to left, as in Figure~\ref{intermedgraphs}. 

The final graph $G_\sigma\ksup{b-1}$ has two boundary vertices. We refer to it as the closure of $G$ nonetheless and consider it to be the singularization of $\widehat{\sigma}$, writing $G_{\widehat{\sigma}}$ for $G_\sigma\ksup{b-1}$.
The choice to work with diagrams in which the outermost strand is cut is typical in the quantum topology literature because the representation theory underlying the Alexander polynomial demands it (see e.g.~\cite{viroquantumrel}). It is also in line with the use of basepoints in related work~\cite{ozsszcube,reidmoves,manolescucube}.

We work in the edge rings $\cE\!\left(G_\sigma\ksup{k}\right)$, where we have retained the edge labels in $\underline{x}\!\left(G_\sigma\ksup{k}\right)$ and discarded those in $\underline{y}\!\left(G_\sigma\ksup{k}\right)$ under the isomorphism specified in Definition~\ref{edgering}. We assume that generators of the ideals $L\!\left(G_\sigma\ksup{k}\right)$, $Q\!\left(G_\sigma\ksup{k}\right)$, and $N\!\left(G_\sigma\ksup{k}\right)$ have been rewritten accordingly. In diagrams, we assume that edges bearing a single label have retained the appropriate one under our conventions. Abbreviate the edge ring $\cE\!\left(G_\sigma\ksup{k}\right)$ as $\cE_k$.

Figure~\ref{intermedgraphs} illustrates the notation used in the rest of this section. Let $\ztau\ksup{k}$ and $\zbeta\ksup{k}$ denote the top- and bottom-most remaining edges on the $k^\text{th}$ strand of $G_\sigma$ after the variables in $\underline{y}\!\left(G_\sigma\right)$ have been discarded and the markings denoting framing removed from thin edges. The $\ztau\ksup{k}$ and $\zbeta\ksup{k}$ are not new variables, but simply alternate names for certain variables in $\underline{x}\!\left(G_\sigma\right)$. Let $a_k$ be the framing on the top boundary edge of the $k^\text{th}$ strand of $G_\sigma$. Let $Z_{k+1}\subset\cE_k$ be the ideal generated by $\ztau\ksup{k+1}-t^{a_{k+1}}\zbeta\ksup{k+1}$. Closing the $(k+1)^\text{st}$ strand of $G_\sigma\ksup{k}$ corresponds to taking the quotient of $\cE_k$ by $Z_{k+1}$. Let $\pi_k:\cE_k\to \cE_{k+1}$ be the quotient map with kernel $Z_{k+1}$ that retains $\ztau\ksup{k+1}$ and discards $\zbeta\ksup{k+1}$. In $G\ksup{k}$, we call the joined edges $\ztau\ksup{i}=t^{a_i}\zbeta\ksup{i}$ for $i\leq k$ \emph{closure edges} and continue to call the remaining $\ztau\ksup{i}$ and $\zbeta\ksup{i}$ for $i>k$ \emph{boundary edges}.

Closing $G_\sigma$ completely to $G_{\widehat{\sigma}}$ corresponds to taking successive quotients of $\cE(G_\sigma)$ by each of the $Z_k$. Equivalently, let $Z\subset\cE(G_\sigma)$ be the ideal generated by $\ztau\ksup{k}-t^{a_k}\zbeta\ksup{k}$ for $1\leq k \leq b-1$. Then the edge rings of $G_{\sigma}$ and its closure are related by $\cE(G_\sigma)/Z\isom\cE(G_{\widehat{\sigma}})$.
Let $\pi:\cE(G_\sigma)\to \cE(G_{\widehat{\sigma}})$ be the composition $\pi_{b-2}\circ\cdots\circ\pi_0$, so $\pi$ retains all of the $\ztau\ksup{k}$ and discards all of the $\zbeta\ksup{k}$ except $\zbeta\ksup{b}$.

For each intermediary graph, we have ideals $L\!\left(G_\sigma\ksup{k}\right)$, $Q\!\left(G_\sigma\ksup{k}\right)$, and $N\!\left(G_\sigma\ksup{k}\right)$ in $\cE_k$.  We abbreviate these $L_k$, $Q_k$, and $N_k$.  The generators of the linear and quadratic ideals depend only on local information at each thick edge, which will not change when strands are closed. Therefore, $\pi_k(L_k)=L_{k+1}$ and 
$\pi_k(Q_{k})=Q_{k+1}$. The analogous statement is not at all true for the non-local ideal. The set of edges that are internal to a subset $\gm$ may change as we proceed from $G_\sigma\ksup{k}$ to $G_{\sigma}\ksup{k+1}$. Specifically, a subset $\gm$ in $G_\sigma\ksup{k}$ may have $\ztau\ksup{k+1}$ as an outgoing edge and $\zbeta\ksup{k+1}$ as an incoming edge. The generator of $N_k\subset\cE_k$ associated to $\gm$ will then have $\ztau\ksup{k+1}$ dividing one of its terms and $\zbeta\ksup{k+1}$ dividing the other. Under $\pi_k$, the $\zbeta\ksup{k+1}$ will be replaced by $t^{a_{k+1}}\ztau\ksup{k+1}$, so both of its terms will be divisible by $\ztau\ksup{k+1}$. In $G_\sigma\ksup{k+1}$, however, the closure edge on the $(k+1)^\text{st}$ strand will be internal to $\gm$. Therefore, neither term of the generator of $N_{k+1}$ associated to $\gm$ will be divisible by $\ztau\ksup{k+1}$. Refer to Section~\ref{examplesection} and the set $\gm\cup\dl$ in Figure~\ref{smallexample} for an example of how this can occur. 

We have seen so far that $\pi_k(N_{k})\subsetneq N_{k+1}$. The content of our main result is that the situation described above fully explains the behavior of the non-local ideal under braid closure. That is, closing a braid strand corresponds to taking a quotient (Definition~\ref{iqdfn}) of the non-local ideal. Now that all of the necessary notation is in place, we remind the reader of our main result and its corollary.

\begin{thm*}[Restatement of Theorem~\ref{idealqthmintro}]
With notation as above, the following ideals are equal in $\cE\!\left(G_\sigma\ksup{k+1}\right)$
 $$\pi_k\!\left(N\!\left(G_\sigma\ksup{k}\right)\right) : \left(\ztau\ksup{k+1}\right) = N\!\left(G_\sigma\ksup{k+1}\right)$$ for $0\leq k\leq b-2$.
\end{thm*}
 
\begin{cor*}[Restatement of Corollary~\ref{idealqcorintro}]
Let $\ztau$ be the product $\ztau\ksup{1}\cdots\ztau\ksup{b-1}$.  The following ideals are equal in $\cE(G_{\widehat{\sigma}})$
 $$\pi(Q(G_\sigma)) : (\ztau) = N(G_{\widehat{\sigma}}).$$
\end{cor*}
 \begin{proof}
The corollary follows immediately from the facts that $Q(G_\sigma)=N(G_\sigma)$ and that $I:(xy)=(I:(x)):(y)$ for any ideal $I$, ring $R$, and ring elements $x,y\in R$.
 \end{proof}

We will prove Theorem~\ref{idealqthmintro} in detail when $G_\sigma$ is blackboard framed, then describe in Section~\ref{nonbbframings} how to modify the proof to handle non-blackboard framings. No substantive changes to the proof are required; the restriction to blackboard framings merely simplifies the notation.
 
\section{Background: Gr\"obner bases and Buchberger's Algorithm}
\label{gbbackground}

We approach Theorem~\ref{idealqthmintro} as a commutative algebra calculation: given generating sets for two ideals in a polynomial ring, create a generating set for their ideal quotient. In fact, we would like to re-create a previously specified generating set. Gr\"obner bases are a convenient tool for this sort of calculation.  They make it possible to generalize sensibly the division algorithm for single-variable polynomials to a division algorithm for multivariable polynomials, thereby reducing certain difficult questions in commutative algebra and algebraic geometry to computational problems. Gr\"obner bases are the foundation of computer algebra programs that do commutative algebra in polynomial rings, such as Macaulay 2~\cite{macaulay2}.  In this section, we define Gr\"obner bases, describe an algorithm for converting an arbitrary generating set for an ideal into a Gr\"obner basis, and explain how Gr\"obner bases can be used to calculate generating sets for ideal intersections and quotients. The exposition here is an adaptation of that in~\cite{gbbook}.

\subsection{Monomial orders}

Let $\itk$ be a field and $\itk[x_0,\ldots,x_n]=\kpoly$ a polynomial ring over it. 
\begin{dfn}
A \emph{monomial order} is a total ordering of the monomials $x_0^{\alpha_0}\cdots x_n^{\alpha_n}$ in $\kpoly$ that satisfies 
\begin{enumerate}
\item $1<x_0^{\alpha_0}\cdots x_n^{\alpha_n}$ for all monomials with $\alpha_i$ not all zero, and
\item $y<y^\prime$ implies $yz<y^\prime z$ for any monomials $y, y^\prime, z$ in $\kpoly$.
\end{enumerate}
\end{dfn}
We will use the lexicographic ordering on $\kpoly$ in which $x_0>x_1>\cdots>x_n>1$. This means that $x_0^{\alpha_0}\cdots x_n^{\alpha_n}>x_0^{\beta_0}\cdots x_n^{\beta_n}$ when $\alpha_i>\beta_i$ for the first $i$ at which the exponents differ. The largest monomial is written first. For example, the following polynomials are written correctly with respect to the lexicographic term order.
\begin{center}
\begin{tabular}{ccc}
$f_1=x_1^2x_2-x_1x_2^2$ & $f_2=2x_1-x_2x_3x_4$ & $f_3=x_5+4x_6^3-1$
\end{tabular}
\end{center}
Throughout the remaining sections, we will write polynomials with respect to the lexicographic order unless we specify otherwise. 

Given a monomial order,  denote the leading term and the leading monomial of a polynomial $f\in\kpoly$ by $\lt{f}$ and $\lm{f}$ respectively. For example, $\lt{f_2}=2x_1$ and $\lm{f_2}=x_1$. For blackboard framed graphs, there will be no difference between leading terms and leading monomials because our coefficients are always $\pm 1$. When we deal with polynomials that have only two terms, we will denote the trailing term (i.e.~the non-leading term) by $\trt{f}$ and the trailing monomial by $\trm{f}$.

\subsection{Gr\"obner bases and the division algorithm}

A Gr\"obner basis is a generating set for an ideal that accounts for all possible leading monomials of polynomials in that ideal.
\begin{dfn}
A Gr\"obner basis for an ideal $I\subset \kpoly$ is a set of polynomials $g_1,\ldots,g_k$ in $I$ such that for any $f\in I$, there is some $i$ for which $\lm{g_i}$ divides $\lm{f}$.
\end{dfn}
\noi It follows from the Hilbert Basis Theorem and a few basic observations that every nonzero ideal in $\kpoly$ has a Gr\"obner basis~\cite[Corollary 1.6.5]{gbbook}. Gr\"obner bases are not unique and are typically highly redundant; an ideal typically has a smaller generating set that is not a Gr\"obner basis.

The key advantage of Gr\"obner bases over other generating sets is that they make it possible to generalize the division algorithm to multivariable polynomials in a useful way. Generalizing the algorithm is straightforward enough: To divide $f$ by $g$ in $\kpoly$, we see whether $\lm{g}$ divides $\lm{f}$. If it does, we record $\lt{f} / \lt{g}$ as a term of the quotient and replace $f$ by $f-\frac{\lt{f}}{\lt{g}}g$. If not, we record $\lt{f}$ as a term in the remainder and replace $f$ with $f-\lt{f}$. Continuing this process as long as possible, we eventually obtain a decomposition of $f$ as $f=qg+r$ for some $q, r\in\kpoly$. We may also divide $f$ by a collection of polynomials $g_1,\ldots,g_k$ to obtain a decomposition $f=q_1g_1+\cdots+q_kg_k+r$.  At each step, we look for the first $i$ such that $\lm{g_i}$ divides $\lm{f}$, then record $\lt{f}/\lt{g_i}$ as a term in the quotient $q_i$ and replace $f$ by $f-\frac{\lt{f}}{\lt{g_i}}g_i$. If no $\lm{g_i}$ divides $\lm{f}$, then we record $\lt{f}$ as a term in the remainder $r$ and replace $f$ with $f-\lt{f}$. We will write $$f\xrightarrow{g_1,\ldots,g_k}r$$ and say ``$f$ reduces to $r$ via $g_1,\ldots,g_k$'' if $r$ is obtained as a remainder when using this algorithm to divide $f$ by $g_1,\ldots,g_k$.

In general, the result of this generalized division algorithm depends on the monomial order chosen on $\kpoly$ and the order in which the polynomials $g_1,\ldots,g_k$ are listed. Neither the quotients $q_1,\ldots,q_k$ nor the remainder $r$ are unique. Consequently, this generalized division algorithm on its own is of little use. It is not true, for example, that the remainder $r$ is zero if and only if $f$ is in the ideal generated by $g_1,\ldots,g_k$. 

However, if $g_1,\ldots,g_k$ are a Gr\"obner basis for the ideal they generate, then the remainder $r$ is unique: it does not depend on the monomial order or on the order in which the $g_i$ are listed. The quotients are still not unique, but the uniqueness of the remainder is sufficient to make the generalized division algorithm useful for commutative algebra computations. For instance, if $g_1,\ldots,g_k$ are a Gr\"obner basis for the ideal they generate, then $f\in (g_1,\ldots,g_k)$ if and only if $f$ reduces to zero via $g_1,\ldots,g_k$.

\subsection{Buchberger's Algorithm and ideal quotients}

Buchberger~\cite{buch1} developed an algorithm for converting any generating set of an ideal into a Gr\"obner basis. Such an algorithm must produce new generators that account for leading monomials of polynomials in the ideal that did not appear as leading monomials among the original generators. New leading monomials arise when a linear combination of existing generators causes their leading terms to cancel. Buchberger's Algorithm systematically produces these cancellations using S-polynomials.
\begin{dfn}
\label{spolydfn}
The \emph{S-polynomial} of two non-zero polynomials $f,g\in\kpoly$ is 
\begin{align*}
S(f,g)&=\frac{\lcm{\lm{f},\lm{g}}}{\lt{f}}f-\frac{\lcm{\lm{f},\lm{g}}}{\lt{g}}g\\
&=\frac{\lm{g}}{\lc{f}\gcd(\lm{f},\lm{g})}\overf-\frac{\lm{f}}{\lc{g}\gcd(\lm{f},\lm{g})}\overg,
\end{align*}
where $\overf=f-\lt{f}$ and $\overg=g-\lt{g}$.
\end{dfn}
\noi  If $\lt{f}=\lm{f}$ and $\lt{g}=\lm{g}$, we have the simplification
\begin{align*}
S(f,g)=\frac{\lm{g}}{\gcd(\lm{f},\lm{g})}\overf-\frac{\lm{f}}{\gcd(\lm{f},\lm{g})}\overg.
\end{align*}

Buchberger's theorem~\cite[Theorem 1.7.4]{gbbook} is that a generating set $g_1,\ldots,g_k$ for an ideal $I\subset\kpoly$ is a Gr\"obner basis for $I$ if and only if $S(g_i,g_j)\xrightarrow{g_1,\ldots,g_k}0$ for all $i\neq j$. Buchberger's Algorithm, then, is as follows. 

\begin{algorithm}[Buchberger]
\label{buchalg}
\noindent Let $g_1,\ldots,g_k$ be a generating set for an ideal $I\subset\kpoly$.
\begin{enumerate}
\item Compute $S(g_i,g_j)$ for some $i\neq j$ and attempt to reduce it via $g_1,\ldots,g_k$ using the generalized division algorithm.
\item If $S(g_i,g_j)\xrightarrow{g_1,\ldots,g_k}0$, go back to the previous step and compute a different S-polynomial. If $S(g_i,g_j)\xrightarrow{g_1,\ldots,g_k}r$ and $r\neq 0$, then add $r$ to a working basis.
\item Repeat the previous two steps until a basis $g_1,\ldots,g_{k+s}$ is obtained for which $S(g_i,g_j)\xrightarrow{g_1,\ldots,g_{k+s}}0$ for all $i\neq j$.
\end{enumerate}
\end{algorithm}
\noindent Buchberger~\cite{buch1} proved that this algorithm terminates and produces a Gr\"obner basis for $I$.

Theorem~\ref{idealqthmintro} claims that the process of closing a braid strand corresponds to taking an ideal quotient of the non-local ideal.
\begin{dfn}
\label{iqdfn}
Let $I,J$ be ideals in a ring $R$. The \emph{ideal quotient of $I$ by $J$} is 
$$I:J=\{r\in R\,\vert\,rJ\subset I\}.$$ Note that $I$ is always contained in $I\!:\!J$.
\end{dfn}
We will use Buchberger's Algorithm to produce an explicit generating set for the ideal quotient $\pi_k\!\left(N\!\left(G_\sigma\ksup{k}\right)\right):\left(\ztau\ksup{k+1}\right)$. It will be readily recognizable as the generating set by which $N\!\left(G_{\sigma}\ksup{k+1}\right)$ was defined. First, we produce a generating set for the intersection $\pi_k\!\left(N\!\left(G_\sigma\ksup{k}\right)\right)\cap\left(\ztau\ksup{k+1}\right)$ in $\cE\!\left(G_\sigma\ksup{k+1}\right)$. The following straightforward proposition explains how a generating set for an intersection yields a generating set for a quotient. It is a rephrasing of \cite[Lemma 2.3.11]{gbbook}, for example.

\begin{prop}
Let $I\subset R$ be an ideal in a polynomial ring and $x\in R$. If $h_1,\ldots,h_k$ is a generating set for $I\cap (x)$, then $\frac{h_1}{x},\ldots,\frac{h_k}{x}$ is a generating set for $I:(x)$.
\end{prop}

To produce a Gr\"obner basis for an intersection, we follow the method prescribed in \cite[Proposition 2.3.5]{gbbook}. Suppose that $I,J\subset\kpolylong$ are ideals with generating sets $p_1,\ldots,p_k$ and $q_1,\ldots,q_\ell$ respectively. Enlarge the polynomial ring to include a dummy variable $\nu$. Define the monomial order on $\itk[x_0,\ldots,x_n,\nu]$ to be lexicographic with $\nu>x_0>\cdots>x_n>1$. (The lexicographic ordering is a special case of an ``elimination ordering,'' which is what is actually required for this procedure to work.)  Then 
$$I\cap J=(\nu I+(\nu-1)J)\cap\kpolylong$$ and a Gr\"obner basis for $I\cap J$ can be obtained from a Gr\"obner basis for $\nu I + (\nu-1)J$ by intersecting the basis with $\kpolylong$~\cite[Theorem 2.3.4]{gbbook}. Therefore, to obtain a basis for $I\cap J$, we apply Buchberger's Algorithm to the basis $$\nu p_1,\ldots,\nu p_k,(\nu-1)q_1,\ldots,(\nu-1)q_\ell,$$ then discard any generator in which $\nu$ appears.

In sum, we have the following algorithm for producing a Gr\"obner basis for the ideal quotient $I:(x)$ (where $x$ is a monomial) starting from a generating set $p_1,\ldots,p_k$ for $I$.
\begin{algorithm}
\label{iqalg}
\begin{enumerate}
\item[]{}
\item Apply Buchberger's Algorithm (Algorithm~\ref{buchalg}) to\\ $\displaystyle{\left\{\nu p_1,\ldots,\nu p_k,\nu x-x\right\}}$ in $\itk[x_0,\ldots,x_n,\nu]$ with an ordering in which $\nu>x_i$ for all $x_i$. Let $\left\{p_1,\ldots,p_{k+s}\right\}$ ($m\geq k$) be the output of Buchberger's Algorithm.
\item Intersect $\left\{p_1,\ldots,p_{k+s}\right\}$ with $\itk[x_0,\ldots,x_n]$. Let $\left\{p_1^\prime,\ldots,p_m^\prime\right\}$ be the resulting subset of generators.
\item Divide each of the $p_i^\prime$ by $x$. The set $\displaystyle{\left\{\frac{p_1^\prime}{x},\ldots,\frac{p_m^\prime}{x}\right\}}$ is a Gr\"obner basis for $I:(x)$.
\end{enumerate}
\end{algorithm}

\subsection{Simplifying Gr\"obner basis computations}
\label{gbsimplify}
We record here a collection of propositions that will simplify computations encountered when applying Buchberger's Algorithm. 

\begin{prop}
\label{gbprinciplegcd}
Let $f,g\in\kpoly$. If $\gcd(\lm{f},\lm{g})=1$, then $S(f,g)\xrightarrow{f,g}0$.
\end{prop}
\begin{proof}
Let $f=\lt{f}+\overf$ and $g=\lt{g}+\overg$. Then we can compute and reduce $S(f,g)$ as follows.  The two possible term orders are considered in two columns.
\begin{equation*}
\begin{aligned}
S(f,g)&=\frac{\lm{g}}{\lc{f}}\overf-\frac{\lm{f}}{\lc{g}}\overg\\
\textrm{reduce}\;&-\frac{\overf}{\lc{f}\lc{g}}(\lt{g}+\overg)\\
&=-\frac{\overg}{\lc{g}}\left(\lm{f}+\frac{\overf}{\lc{f}}\right)\\
\textrm{reduce}\;&+\frac{\overg}{\lc{f}\lc{g}} f\\
&=0
\end{aligned}
\quad\vline\quad
\begin{aligned}
S(f,g)&=-\frac{\lm{f}}{\lc{g}}\overg+\frac{\lm{g}}{\lc{f}}\overf\\
\textrm{reduce}\;&+\frac{\overg}{\lc{f}\lc{g}}(\lt{f}+\overf)\\
&=\frac{\overf}{\lc{f}}\left(\lm{g}+\frac{\overg}{\lc{g}}\right)\\
\textrm{reduce}\;&-\frac{\overf}{\lc{f}\lc{g}} g\\
&=0
\end{aligned}
\end{equation*}
\end{proof}

\begin{prop}
\label{gbprinciplecoeff}
Let $f=\lt{f}+\overf$ and $g=\lt{g}+\overg$ be polynomials in $\kpoly$. Let $a,b$ be monomials in $\kpoly$. Then $$S(af,ag)=aS(f,g).$$
If $\gcd(a,b)=\gcd(a,\lm{g})=\gcd(b,\lm{f})=1$, then $$S(af,bg)=abS(f,g).$$
If $\gcd(a,\lm{f})=1$, then $$S(af, a\lt{g}+\overg)=S(f,a\lt{g}+\overg).$$
\end{prop}
\begin{proof}
Since $\gcd(a\lm{f}, a\lm{g})=a\gcd(\lm{f}, \lm{g}),$ we compute as follows for the first claim.
\begin{eqnarray*}
S(af,ag)
&=&\frac{a\lm{g}}{\lc{f}a\gcd(\lm{f},\lm{g})}a\overf-\frac{a\lm{f}}{\lc{g}a\gcd(\lm{f},\lm{g})}a\overg\\
&=&a\left(\frac{\lm{g}}{\lc{f}\gcd(\lm{f},\lm{g})}\overf-\frac{\lm{f}}{\lc{g}\gcd(\lm{f},\lm{g})}\overg\right)\\&=&aS(f,g)
\end{eqnarray*}
The term order in the first two lines above may not be as written, but it is not changed by cancelling or factoring out $a$ from the expression.

For the second claim, our assumptions imply that $\gcd(af,bg)=\gcd(f,g)$. The computation is as follows.
\begin{align*}
S(af,bg)&=\frac{b\lm{g}}{\lc{f}\gcd(\lm{f},\lm{g})}a\overf-\frac{a\lm{f}}{\lc{g}\gcd(\lm{f},\lm{g})}b\overg\\
&=ab\left(\frac{\lm{g}}{\lc{f}\gcd(\lm{f},\lm{g})}\overf-\frac{\lm{f}}{\lc{g}\gcd(\lm{f},\lm{g})}\overg\right)\\
&=abS(f,g)
\end{align*}
Again, the term order may not be as written, but it does not change when we factor out $ab$.

For the third claim, the key observations are that \begin{align*}\gcd(a\lm{f},a\lm{g})&=a\gcd(\lm{f},\lm{g})\quad\text{and}\\\gcd(\lm{f},a\lm{g})&=\gcd(\lm{f},\lm{g})\end{align*} when $\gcd(a,\lm{f})=1$. Therefore, 
\begin{align*}
S(af, a\lt{g}+\overg)=&\frac{a\lm{g}}{\lc{f}a\gcd(\lm{f},\lm{g})}a\overf\\
&\,-\frac{a\lm{f}}{\lc{g}a\gcd(\lm{f},\lm{g})}\overg\\
=&\frac{a\lm{g}}{\lc{f}\gcd(\lm{f},a\lm{g})}\overf\\
&\,-\frac{\lm{f}}{\lc{g}\gcd(\lm{f},a\lm{g})}\overg\\
=&S(f, a\lt{g}+\overg).
\end{align*}
\end{proof}

We will sometimes encounter expressions with unknown term order after computing an S-polynomial. The following proposition allows us to reduce some such expressions without explicitly determining their term order.

\begin{prop}
\label{unorderedreduction}
Let $p, q, r, s\in\kpoly$ be monomials whose relationships to each other under the monomial order are unknown. Then whichever of $ps-rq$ or $rq-ps$ is correctly ordered is reducible to zero by the correctly ordered versions of $p-q$ and $r-s$.
\end{prop}
\begin{proof}
Suppose that $ps-rq$ is correctly ordered, so $ps>rq$. Then either $p>q$ or $s>r$ or both. Assume without loss of generality that $p>q$. Then term orders are correct in the following computation.
\begin{align*}
&ps-rq\\
\text{reduce:}\quad-\,&s(p-q)\\
=\,&q(s-r) \text{ or } q(r-s)
\end{align*}
The term order in the last line depends on whether $r<s$ or $s<r$. Either way, the last expression reduces by the version of $r-s$ with the correct term order.

If instead $rq-ps$ is correctly ordered, then either $q>p$ or $r>s$ or both. Without loss of generality, assume $q>p$. Reduce by $q-p$ to get $p(r-s)$ or $p(s-r)$ depending on which term order is correct for $r-s$. Either way, the result reduces by the correctly ordered version of $r-s$. 
\end{proof}

\section{Outline of Proof of Theorem~\ref{idealqthmintro}}
\label{outlinesec}

Gr\"obner bases and Buchberger's Algorithm offer a concrete, constructive approach to our claim that the non-local ideal arises as an ideal quotient. With a carefully chosen monomial order, the computations and reductions of S-polynomials prescribed by Buchberger's Algorithm actually produce exactly the defining generators of the non-local ideal for the braid graph with one additional strand closed. Moreover, it is possible to interpret all S-polynomial computations required by Buchberger's Algorithm with reference to the graph, and thereby ensure that the algorithm produces no extraneous generators for the ideal quotient.

This section outlines the proof of Theorem~\ref{idealqthmintro} via Algorithm~\ref{iqalg}. Algorithm~\ref{iqalg} calls for an application of Buchberger's Algorithm (Algorithm~\ref{buchalg}) to a generating set for $\pi_k\!\left(N\!\left(G_\sigma\ksup{k}\right)\right)\subset\cE\!\left(G_\sigma\ksup{k+1}\right)\!\left[\nu\right]$, so we begin in Sections~\ref{mosec} and~\ref{startingbasissec} by setting up notation to describe such a set and defining a monomial order on $\cE\!\left(G_\sigma\ksup{k+1}\right)\!\left[\nu\right]$. We then describe how Buchberger's Algorithm progresses (Section~\ref{bbrd1ovw}) and the output it produces (Lemma~\ref{buchend}). We go on to prove Theorem~\ref{idealqthmintro} from Lemma~\ref{buchend} in Section~\ref{lemmatothmsec}. In Section~\ref{examplesection}, we work out a detailed example illustrating the algorithm for a particular small graph.

\subsection{Notation and a Monomial Order}
\label{mosec}

We will use the notation first introduced in Section~\ref{mainresult} and shown in Figure~\ref{intermedgraphs}, but drop all reference to $\sigma$. Refer also to Section~\ref{examplesection} and Figure~\ref{smallexample} for a specific example. Let $G=G\ksup{0}$ be the graph obtained by from a non-closed braid diagram with $b$ strands by replacing crossings with thick edges. Let $G\ksup{1},\ldots,G\ksup{b-1}$ be the intermediary graphs obtained by closing strands of $G$ one at a time from right to left. Assume that $G$ carries some framing, and that the $G\ksup{i}$ inherit this framing with newly closed strands bearing the sum of the framings on the edges that formed them.

We have already fixed the isomorphisms $\cE\!\left(G\ksup{i}\right)\isom\cR\!\left[\underline{x}\!\left(G\ksup{i}\right)\right]$ that retain the edge labels in $\underline{x}\!\left(G\ksup{i}\right)$ and discard the edge labels in $\underline{y}\!\left(G\ksup{i}\right)$. Recall that we abbreviate $\cE\!\left(G\ksup{i}\right)$ by $\cE_i$. Let $\underline{x}(G)=\left(x_0,\ldots,x_n\right)$ be the edges in $G$, labeled from top to bottom, right to left. Recall that $\ztau\ksup{i}$ and $\zbeta\ksup{i}$ are additional labels for the top- and bottom-most edges on the $i^\text{th}$ strand of $G$ after the variables in $\underline{y}(G)$ have been discarded. Recall also that $a_{i}$ is the framing on the top boundary edge of the $i^\text{th}$ strand of $G$, and $\pi_i$ is the quotient map with kernel $Z_{i+1}=\left(\ztau\ksup{i+1}-t^{a_{i+1}}\zbeta\ksup{i+1}\right)$ that retains the edge label $\ztau\ksup{i+1}$ and discards $\zbeta\ksup{i+1}$.

Consider the closure of the $(k+1)^\text{st}$ strand of $G\ksup{k}$ as in Theorem~\ref{idealqthmintro}. The diagrams $G\ksup{k}$ and $G\ksup{k+1}$ inherit their edge labels from $G$ in accordance with our conventions for the isomorphisms $\pi_i$ that relate $\cE_{i}$ and $\cE_{i+1}$.  Under these conventions, $\cE_{k+1}$ is the polynomial ring over $\cR$ with indeterminates $\ztau\ksup{1},\ldots,\ztau\ksup{b}$, $\zbeta\ksup{k+2},\ldots,\zbeta\ksup{b}$, and an appropriate proper subset of the original $x_0,\ldots,x_n$.

To implement Algorithm~\ref{iqalg}, we require a monomial order on $\cE_{k+1}[\nu]$. The ordering we employ relies crucially on the edge labeling conventions specified above.

\begin{dfn}
\label{edgeringmo}
Let $x_0,\ldots,x_n$ label the edges of $G$ from top to bottom, right to left. Let $\ztau\ksup{i}$ and $\zbeta\ksup{i}$ be additional labels for the top- and bottom-most remaining edges on the $i^\text{th}$ strand for $0\leq i\leq b-1$. For each $i\leq k+1$, label the closure edges of $G\ksup{k+1}$ with $\ztau\ksup{i}$ and discard the label $\zbeta\ksup{i}$. The monomial order on $\cE_{k+1}[\nu]$ is the lexicographic ordering with $\nu>\ztau\ksup{k+1}>x_i>1$ for all $i$ and $x_i>x_j$ when $i<j$.
\end{dfn}

The property that $\ztau\ksup{k+1}$ precedes all other edge variables in the monomial order on $\cE_{k+1}[\nu]$ allows us to relate divisibility by $\ztau\ksup{k+1}$ to determination of a polynomial's leading term. In diagrammatic terms, divisibility by $\ztau\ksup{k+1}$ encodes the relationship of a subset to the braid strand being closed. This connection between the monomial order and the braid diagram is what makes it possible to keep the size and composition of our Gr\"obner basis under control, which is ultimately what allows us to describe the process and outcome of Buchberger's Algorithm.

\begin{obs}
\label{MOanddiv}
Let $f\in\cE_{k+1}[\nu]$ be a polynomial with no term divisible by $\nu$. If $\ztau\ksup{k+1}$ divides exactly one term of $f$, then the term divisible by $\ztau\ksup{k+1}$ must be the leading term of $f$. 
\end{obs}

The observation follows immediately from the requirement that the monomial order satisfy $\ztau\ksup{k+1}>x_i$ for all $x_i$. With this requirement, only a term divisible by $\nu$ could precede a term divisible by $\ztau\ksup{k+1}$. In the absence of $\nu$, we look to $\ztau\ksup{k+1}$ to determine the leading term. If it occurs in only one term, then that term must lead. Any monomial order for which Observation~\ref{MOanddiv} holds could be used to prove Theorem~\ref{idealqthmintro} in the same way. We have specified the lexicographic ordering on the remaining variables simply for concreteness. 

\subsection{Input and output bases and the monomial order}
\label{startingbasissec}

Let $N_i\subset\cE_i$ be the non-local ideal defined in Section~\ref{idealdfnsec}. In our abbreviated notation, Theorem~\ref{idealqthmintro} claims that $$\pi_k\!\left(N_k\right):\left(\ztau\ksup{k+1}\right) = N_{k+1}$$ in $\cE_{k+1}$ for $0\leq k\leq b-2$.

To implement Algorithm~\ref{iqalg}, we first require a generating set for $\pi_k(N_k)\subset\cE_{k+1}$.  We will obtain it from the generating set of $N_k\subset\cE_k$ specified in Section~\ref{idealdfnsec}. Let $h_\gm\ksup{k}$ denote the generator of $N_k$ associated to a subset $\gm$ of the thick edges and bivalent vertices in $G$. Let $g_\gm\ksup{k}=\pi_k\!\left(h_\gm\ksup{k}\right)$ denote its image in $N_{k+1}$. Let $g_\gm\ksup{k+1}$ denote the generator of $N_{k+1}$ associated to $\gm$. Then our input basis for Algorithm~\ref{iqalg} is the set of  $g_\gm\ksup{k}$ for all subsets $\gm$, which means our starting basis for Buchberger's Algorithm is 
\begin{equation}
\label{startingbasis}
\cG_0=\left\{\nu g_\gm\ksup{k}\,\vert\,\gm\subset G\right\}\cup\left\{\nu\ztau\ksup{k+1}-\ztau\ksup{k+1}\right\},
\end{equation}
as a set of polynomials in $\cE_{k+1}[\nu]$. We will prove Theorem~\ref{idealqthmintro} by showing that the output of Algorithm~\ref{iqalg} is 
\begin{equation}
\label{endingbasis}
\cG_{\text{end}}=\left\{g_\gm\ksup{k+1}\,\vert\,\gm\subset G\right\},
\end{equation}
which is the defining basis for $N_{k+1}$. (We abuse notation slightly throughout by letting the words ``in $G$'' and the symbols ``$\subset G$'' mean ``is a subset of the thick edges and bivalent vertices in $G$.'')

We now analyze the polynomials $g_\gm\ksup{k}$ and $g_\gm\ksup{k+1}$ in greater detail, particularly with respect to the monomial order on $\cE_{k+1}[\nu]$. For simplicity, we will assume from now until Section~\ref{nonbbframings} that the graphs $G\ksup{i}$ are blackboard-framed. 

Given subsets $\gm$ and $\dl$ in $G$, let $\xsub{\gm}{\dl}$ be the product of interior edges in $G=G\ksup{0}$ from $\gm$ to $\dl$. Let $\zsub{\gm}{\dl}\ksup{i}$ denote the product of closure edges $\ztau\ksup{j}$ that go from $\gm$ to $\dl$ in $G\ksup{i}$, $\zsub{\gm}{\tau}\ksup{i}$ denote the product of edges $\ztau\ksup{j}$ that go from $\gm$ to the top boundary of $G\ksup{i}$, and $\zsub{\beta}{\gm}\ksup{i}$ denote the product of $\zbeta\ksup{j}$ in $G\ksup{i}$ that go from the bottom boundary of $G\ksup{i}$ to $\gm$. Notice that the factors of $\zsub{\gm}{\dl}\ksup{i}$ come only from $\ztau\ksup{j}$ with indices $j\leq i$, while the factors of $\zsub{\gm}{\tau}\ksup{i}$ and $\zsub{\beta}{\gm}\ksup{i}$ come only from $\ztau\ksup{j}$ or $\zbeta\ksup{j}$ with indices $j>i$. 


Ignoring term orders for the moment, we may express generators of $N_k$ as
\begin{equation}
\label{ggmkdfn}
h_\gm\ksup{k}=\xsub{\gm}{G\sm\gm}\zsub{\gm}{G\sm\gm}\ksup{k}\zsub{\gm}{\tau}\ksup{k}-\xsub{G\sm\gm}{\gm}\zsub{G\sm\gm}{\gm}\ksup{k}\zsub{\beta}{\gm}\ksup{k}.
\end{equation}
The generator of $N_{k+1}$ associated to $\gm\subset G$ is 
\begin{equation}
\label{ggmkplusonedfn}
g_\gm\ksup{k+1}=\xsub{\gm}{G\sm\gm}\zsub{\gm}{G\sm\gm}\ksup{k+1}\zsub{\gm}{\tau}\ksup{k+1}-\xsub{G\sm\gm}{\gm}\zsub{G\sm\gm}{\gm}\ksup{k+1}\zsub{\beta}{\gm}\ksup{k+1}.
\end{equation}

\noindent The map $\pi_k$ replaces all instances of $\zbeta\ksup{k+1}$ with $\ztau\ksup{k+1}$. Let 
\begin{equation}
\label{anyzetadfn}
\anyzeta_\gm = \frac{\zsub{\gm}{\gm}\ksup{k+1}}{\zsub{\gm}{\gm}\ksup{k}}=\begin{cases} \ztau\ksup{k+1} & \text{ if } \ztau\ksup{k+1} \text{ is internal to } \gm \text{ in } G\ksup{k+1} \\ 1 & \text{ otherwise.}\end{cases}
\end{equation}
Then the generators of $\pi_k(N_k)$ and $N_{k+1}$ are related by
\begin{equation}
\label{ggmkvkplusone}
g_\gm\ksup{k}=\pi_k\!\left(h_\gm\ksup{k}\right)=\anyzeta_\gm g_\gm\ksup{k+1}.
\end{equation}

We write $g_\gm^{(i), \out}$ for the monomial of $g_\gm\ksup{i}$ that is a product of edges outgoing from $\gm$ and $g_\gm^{(i), \into}$ for the monomial of $g_\gm\ksup{i}$ that is a product of edges incoming to $\gm$. It follows from Equation~\ref{ggmkvkplusone} that $g_\gm\ksup{k}$ and $g_\gm\ksup{k+1}$ have the same term order with respect to the monomial order on $\cE_{k+1}[\nu]$. If this term order is determined by an outgoing edge variable, then we call $\gm$ \emph{out-led}. If it is determined by an incoming edge variable, we call $\gm$ \emph{in-led}. The labeling of $\gm$ as out-led or in-led depends on $k$ because the monomial ordering depends on $k$. Since we work in $\cE_{k+1}$ throughout the proof of Theorem~\ref{idealqthmintro}, we will always mean in-led and out-led with respect to the term order on $\cE_{k+1}[\nu]$. See Section~\ref{examplesection} for examples of in-led and out-led subsets. 

\subsection{Algorithm Overview}
\label{bbrd1ovw}

We now outline the computations that occur as we run Buchberger's Algorithm on $\cG_0$. The flowchart in Figure~\ref{flowchart} summarizes the first round of the algorithm, which produces the basis $\cG^\prime$. In the second round, all S-polynomials reduce to zero within $\cG^\prime$ and the algorithm terminates. Lemma~\ref{buchend} describes the outcome of the algorithm.

Buchberger's Algorithm instructs us to compute S-polynomials among all of the generators in the initial basis $\cG_0$. Initially, this means we have two types of computations: S-polynomials between $\nutopk$ and the $\nu g_\gm\ksup{k}$ and S-polynomials among the $\nu g_\gm\ksup{k}$. These are handled in Sections~\ref{nutoponespolys} and~\ref{subsetspolys}, respectively. 

\begin{figure}
\begin{center}
\input{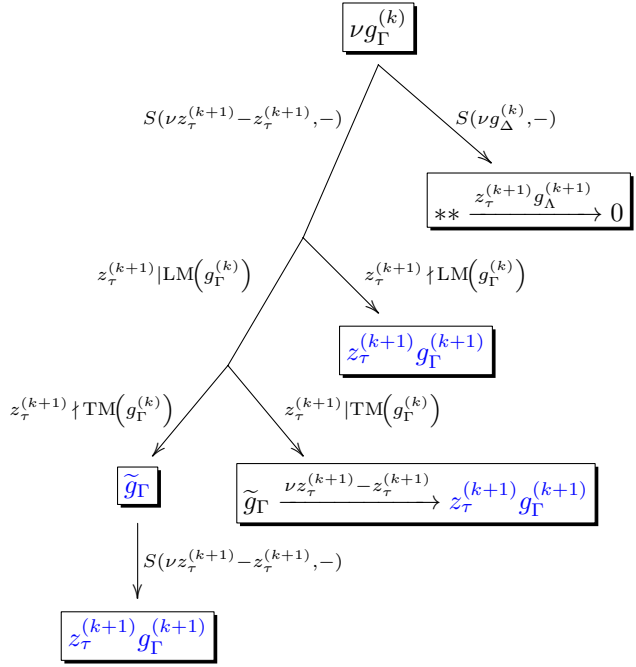}
\caption{Round 1 of Buchberger's Algorithm applied to $\cG_0$. Generators added to the working basis $\cG^\prime$ are shown in blue. The ** indicates that various outcomes are produced at that step, but all reduce to zero as indicated, where $\Lambda\in\{\gm\sm(\gm\cap\dl), \dl\sm(\gm\cap\dl), \gm\cap\dl, \gm\cup\dl\}$.}
\label{flowchart}
\end{center}
\end{figure}

\subsubsection{S-polynomials with $\nutopk$}

As the only element of $\cG_0$ that is not divisible by $\nu$, the generator $\nutopk$ plays a special role. Proposition~\ref{gbprinciplecoeff} implies that S-polynomials among the elements of $\cG_0$ divisible by $\nu$ are equal to $\nu$ times an S-polynomial of the underlying generators of $\pi_k(N_k)$. Therefore, the steps of Buchberger's Algorithm on $\cG_0$ that do not involve $\nutopk$ are parallel to the steps of Buchberger's Algorithm applied to the basis for $\pi_k(N_k)$. That is, in the process of running Buchberger's Algorithm on $\cG_0$ we incidentally produce a Gr\"obner basis for $\pi_k(N_k)$ itself, except that every basis element is multiplied by $\nu$. By contrast, the S-polynomials involving $\nutopk$ have no parallel in Buchberger's Algorithm applied to a basis for $\pi_k(N_k)$. They are the only steps of the algorithm that can possibly produce generators that do not involve $\nu$ in any of their terms. The plan, of course, is to discard any generator in which $\nu$ appears (Step 2 of Algorithm~\ref{iqalg}). Therefore, precursors to the $g_\gm\ksup{k+1}$, which  we are hoping to find in the ideal quotient, will have to be produced by S-polynomials with $\nutopk$.

The result of $S(\nutopk,\nu g_\gm\ksup{k})$ depends on whether $\ztau\ksup{k+1}$ divides the leading monomial of $g_\gm\ksup{k}$. If not, then the S-polynomial reduces to $\ztau\ksup{k+1}g_\gm\ksup{k+1}$, which is a precursor to $g_\gm\ksup{k+1}$. If it does, then the S-polynomial with $\nutopk$ reverses the term order of $\nu g_\gm\ksup{k}$ by removing $\nu$ from its leading term. We call the resulting polynomial a tilde generator
\begin{equation}
\label{tildegendfn}
\widetilde{g}_\gm=\begin{cases} \nu g_\gm^{(k),\into}-g_\gm^{(k),\out}& \text{ if } \gm \text{ is out-led}\\ 
\nu g_\gm^{(k),\out}-g_\gm^{(k),\into}& \text{ if } \gm \text{ is in-led.}
\end{cases}
\end{equation}

If $g_\gm\ksup{k}$ had also a trailing monomial divisible by $\ztau\ksup{k+1}$ (meaning that the edge labeled $\ztau\ksup{k+1}$ goes both into and out of $\gm$ in $G\ksup{k+1}$), then the tilde generator will reduce to $\ztau\ksup{k+1} g_\gm\ksup{k+1}$, which belongs in $\cG^\prime$. Otherwise, we add the tilde generator itself to the working basis. We then immediately compute a second S-polynomial $S(\nutopk,\widetilde{g}_\gm)$, which produces $\ztau\ksup{k+1} g_\gm\ksup{k+1}$, so we add that to $\cG^\prime$ as well. At this point, we will have produced $\ztau\ksup{k+1}g_\gm\ksup{k+1}$ for all subsets $\gm$, so we will have confirmed that $\pi_k(N_k):\ztau\ksup{k+1}\supseteq N_{k+1}$. The working basis will be
\begin{align}
\label{workingbasis}
\cG^\prime=\,\cG_0&\cup\left\{\widetilde{g}_\gm\,\vert\,\gm\subset G, \ztau\ksup{k+1}\vert\lt{g_\gm\ksup{k}}, \ztau\ksup{k+1}\nmid\trt{g_\gm\ksup{k}}\right\}\nonumber\\&\cup\left\{\ztau\ksup{k+1}g_\gm\ksup{k+1}\,\vert\,\gm\subset G\right\}.
\end{align}
All remaining computations will be aimed at proving that no further additions to our working basis are required, which will establish that $\pi_k(N_k):\ztau\ksup{k+1}\subseteq N_{k+1}$.
\begin{lemma}
\label{buchend}
The outcome of Algorithm~\ref{buchalg} applied to $\cG_0$ in $\cE_{k+1}[\nu]$ with the monomial order of Definition~\ref{edgeringmo} is 
$\cG^\prime.$
In particular, $\cG^\prime$ is a Gr\"obner basis for $\pi_k(N_k)$.
\end{lemma}

\subsubsection{S-polynomials among the $\nu g_\gm\ksup{k}$}

Since S-polynomials with $\nutopk$ produced the precursors to all of the generators of $N_{k+1}$,  the hope now is that all remaining S-polynomials reduce to zero via the working basis $\cG^\prime$. Section~\ref{subsetspolys} establishes that this is the case for $S(g_\gm\ksup{k},g_\dl\ksup{k})$ for any pair of subsets $\gm$ and $\dl$. 

The linchpin to the argument is Lemma~\ref{subsetoverlap}, which expresses $S(g_\gm\ksup{k},g_\dl\ksup{k})$ in terms of intersections, unions, and complements of $\gm$ and $\dl$.  The least common multiples and greatest common divisors of monomials that appear in the S-polynomial formula translate into unions and intersections of sets in $G$. For example, $\gcd(\xsub{\gm}{G\sm\gm},\xsub{\dl}{G\sm\dl})=\xsub{\gm\cap\dl}{G\sm(\gm\cup\dl)}$. This correspondence between operations on monomials and operations on sets in $G$ is what makes it possible to describe the progress of Buchberger's Algorithm in terms of the graph.

Once these convenient expressions for the $S(g_\gm\ksup{k},g_\dl\ksup{k})$ have been obtained, it remains to argue that they may be reduced to zero in $\cG^\prime$. For that, we make liberal use of Observation~\ref{MOanddiv} to analyze and compare term orders of the S-polynomials and the elements of $\cG^\prime.$

\subsubsection{S-polynomials involving generators in $\cG^\prime\setminus\cG_0$}

Out of the initial S-polynomial calculations we produced only two types of generators to include in $\cG^\prime$ that were not already in $\cG_0$: tilde generators for a limited class of subsets and $\ztau\ksup{k+1}g_\gm\ksup{k+1}$ for all subsets. Section~\ref{bbrd2} carries out a final round of computations to check that S-polynomials involving these new generators reduce to zero within $\cG^\prime$. Much of the computational work follows from Section~\ref{bbrd1} combined with the shortcuts of Section~\ref{gbsimplify}. The arguments for the reductions to zero essentially follow from Observation~\ref{MOanddiv}, but require careful case by case analyses of term orders. This work completes the proof of Lemma~\ref{buchend}.

\subsection{Proof of Theorem~\ref{idealqthmintro} from Lemma~\ref{buchend}}
\label{lemmatothmsec}
Once Lemma~\ref{buchend} is established, Theorem~\ref{idealqthmintro} follows readily by applying the rest of Algorithm~\ref{iqalg}.
\begin{proof}[Proof of Theorem~\ref{idealqthmintro} from Lemma~\ref{buchend}]
Having completed Buchberger's Algorithm (Step 1 of Algorithm~\ref{iqalg}) and obtained $\cG^\prime\subset\cE_{k+1}[\nu]$, we must now intersect with $\cE_{k+1}$. Doing so produces
\[\cG^\prime\cap\cE_{k+1}=\left\{\ztau\ksup{k+1}g_\gm\ksup{k+1}\,\vert\,\gm\subset G\right\}\]
as a basis for $\pi_k(N_k)\cap\left(\ztau\ksup{k+1}\right)$ in $\cE_{k+1}$. For the last step of Algorithm~\ref{iqalg}, we divide each element of $\cG^\prime\cap\cE_{k+1}$ by $\ztau\ksup{k+1}$ to obtain
\[\cG_{\text{end}}=\left\{g_\gm\ksup{k+1}\,\vert\,\gm\subset G\right\}\] as our basis for $\pi_k(N_k):\left(\ztau\ksup{k+1}\right)$ in $\cE_{k+1}$. The result is exactly the generating set for $N_{k+1}$ described in Definition~\ref{idealdfn}.
\end{proof}

\section{Example/Illustration}
\label{examplesection}

There are three critical features of our set-up that make it possible to characterize the progress and outcome of Algorithm~\ref{iqalg} as we have done:
\begin{enumerate}
\item The divisor ideal $\left(\ztau\ksup{k+1}\right)$ is principal and monomial, which makes the role of $S(\nutopk,-)$ clear. 
\item The S-polynomial of the non-local generators associated to a pair of subsets can be described in terms of operations on subsets.
\item Divisibility by $\ztau\ksup{k+1}$ is closely related to the determination of a polynomial's leading term, as recorded in Observation~\ref{MOanddiv}.
\end{enumerate}
This section aims to illustrate these features by way of the small graphs in Figure~\ref{smallexample}.

\begin{figure}
\begin{center}
\input{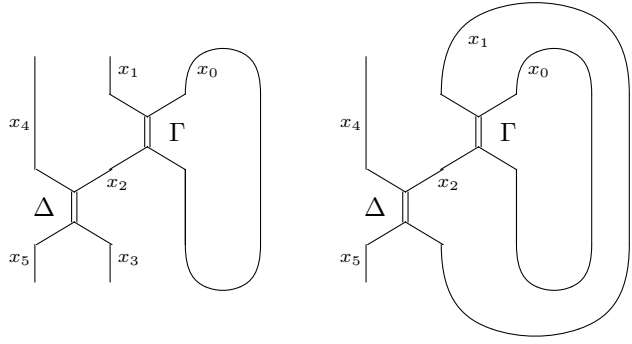}
\caption{Example for Section~\ref{examplesection}: $G_1$ is on the left; $G_2$ is on the right.}
\label{smallexample}
\end{center}
\end{figure}

\subsection{Set-up}
The graphs $G_1$ and $G_2$ in Figure~\ref{smallexample} are labeled as instructed in Definition~\ref{edgeringmo}. For simplicity, we take them to be blackboard framed. Section~\ref{nonbbframings} contains the set-up for a non-blackboard framed example. The edge rings of $G_1$ and $G_2$, respectively, are 
\begin{align*}
&\cE_1=\cR[x_0,x_1,x_2,x_3,x_4,x_5]
&\cE_2=\cR[x_0,x_1,x_2,x_4,x_5].
\end{align*}
It would also be consistent with our notation to say that $x_0=\ztau\ksup{1}$, $x_1=\ztau\ksup{2}$, $x_3=\zbeta\ksup{2}$, $x_4=\ztau\ksup{3}$, and $x_5=\zbeta\ksup{3}$, but we will not need these labels in this section.

The map $\pi_1:\cE_1\to\cE_2$ is the quotient by $(x_1-x_3)$. Let $N_1$ and $N_2$ be the non-local ideals for $G_1$ and $G_2$, respectively. Theorem~\ref{idealqthmintro} claims that \[\pi_1(N_1):(x_1)=N_2\] in $\cE_2$. Algorithm~\ref{iqalg} begins with a basis for $\nu\pi_1(N_1)+(\nu-1)(x_1)$, which Algorithm~\ref{buchalg} will turn into a Gr\"obner basis.

The monomial order on $\cE_2$ is $x_1>x_0>x_2>x_4>x_5$.
All polynomials in this section are written correctly with respect to the monomial order.  Let $\gm$ be the upper (right) thick edge and $\dl$ be the lower (left) thick edge. The non-local generators are as follows.
\begin{align*}
g_\gm\ksup{1}&=\pi_1(x_1-x_2)=x_1-x_2 &\quad g_\gm\ksup{2}&=x_1-x_2\\
g_\dl\ksup{1}&=\pi_1(x_2x_4-x_3x_5)=-x_1x_5+x_2x_4 &\quad g_\dl\ksup{2}&=-x_1x_5+x_2x_4\\
g_{\gm\cup\dl}\ksup{1}&=\pi_1(x_1x_4-x_3x_5)=x_1x_4-x_1x_5 &\quad g_{\gm\cup\dl}\ksup{2}&=x_4-x_5
\end{align*}
Notice that $\gm$ and $\gm\cup\dl$ are out-led while $\dl$ is in-led. These labels are to be interpreted with respect to $\cE_2$ and its monomial order. With respect to $\cE_1$, all three of $\gm$, $\dl$, and $\gm\cup\dl$ would be out-led.

We run Buchberger's Algorithm in $\cE_2[\nu]$ with monomial order \[\nu>x_1>x_0>x_2>x_4>x_5.\] and starting basis 
\[\cG_0=\left\{\nu x_1-x_1,\nu g_\gm\ksup{1},\nu g_\dl\ksup{1},\nu g_{\gm\cup\dl}\ksup{1}\right\}.\]

\subsection{S-polynomials with $\nu x_1-x_1$}
As expected, the S-polynomials with $\nu x_1-x_1$ remove factors of $\nu$, thereby reversing term orders. In two cases, the result is a tilde generator:
\begin{align*}
&S(\nu x_1-x_1, \nu g_\gm\ksup{1})=\nu x_2-x_1=\widetilde{g}_\gm;\\
&S(\nu x_1-x_1, \nu g_\dl\ksup{1})=\nu x_2x_4-x_1x_5=\widetilde{g}_\dl.
\end{align*}
In the third case, the result reduces by $\nu x_1-x_1$ to a precursor of a generator of $N_2$:
\begin{align*}
&S(\nu x_1-x_1, \nu g_{\gm\cup\dl}\ksup{1})=\nu x_1x_5-x_1x_4\xrightarrow{\nu x_1-x_1}-x_1x_4+x_1x_5=-x_1g_{\gm\cup\dl}\ksup{2}.
\end{align*}
We add all three of these outputs to the working basis.

Further S-polynomials between $\nu x_1-x_1$ and the tilde generators produce the remaining precursors to generators of $N_2$, which we also add to the working basis:
\begin{align*}
&S(\nu x_1-x_1, \widetilde{g}_\gm) = x_1^2-x_1x_2 = x_1g_\gm\ksup{2};\\
&S(\nu x_1-x_1, \widetilde{g}_\dl) = x_1^2x_5-x_1x_2x_4 = -x_1g_\dl\ksup{2}.
\end{align*}
As expected, the working basis at this point is 
\[\cG^\prime=\left\{\nu x_1-x_1, \nu g_\gm\ksup{1}, \nu g_\dl\ksup{1}, \nu g_{\gm\cup\dl}\ksup{1}, \widetilde{g}_\gm, \widetilde{g}_\dl, x_1g_\gm\ksup{2}, x_1 g_\dl\ksup{2}, x_1 g_{\gm\cup\dl}\ksup{2}\right\}.\]
It contains all of the precursors to the generators of $N_2$ and contains no other polynomials that will survive the intersection with $\cE_2$. The hope is that all remaining S-polynomials will reduce to zero within $\cG^\prime$.

Notice that the choice of monomial order contributed to the efficiency of the calculations in this section. Since $x_1$ determined the leading terms of $g_\gm\ksup{1}$ and $g_\dl\ksup{1}$, these polynomials were efficiently handled by $S(\nu x_1-x_1,-)$. For example, if the term order of $g_\dl\ksup{1}$ had been reversed, $S(\nu x_1-x_1, \nu g_\dl\ksup{1})$ would have produced $x_1\widetilde{g}_\dl$. Keeping the degree of S-polynomials as low as possible helps to keep the size and composition of the working basis under control.

\subsection{S-polynomials among $\nu g_\gm\ksup{1}$, $\nu g_\dl\ksup{1}$, and $\nu g_{\gm\cup\dl}\ksup{1}$}
As expected, these S-polynomials all reduce to zero within the working basis $\cG^\prime$. There are often several ways to reduce one of these S-polynomials. We follow the methods used in the general arguments of Section~\ref{subsetspolys}.

Recall that $\gm$ and $\gm\cup\dl$ are out-led, while $\dl$ is in-led. The S-polynomials among $\nu g_\gm\ksup{1}$, $\nu g_\dl\ksup{1}$, and $g_{\gm\cup\dl}\ksup{1}$ can be computed from the expressions in Proposition~\ref{ksubsetoverlap} using their relationships to the non-local generators of $N_2$, along with Proposition~\ref{gbprinciplecoeff}. For example:
\begin{align*}
S(\nu g_\gm\ksup{1}, \nu g_\dl\ksup{1})
&=S(\nu g_\gm\ksup{2}, \nu g_\dl\ksup{2})\\
&=\nu S(g_\gm\ksup{2}, g_\dl\ksup{2})\quad\text{by Prop.~\ref{gbprinciplecoeff},}\\
&=\nu \xsub{\dl}{\gm}g_{\gm\cup\dl}\ksup{2}\quad\text{by Prop.~\ref{ksubsetoverlap}, with roles of $\gm$ and $\dl$ reversed,}\\
&=\nu x_2(x_4-x_5)
\end{align*}
Similar computations give
\begin{align*}
S(\nu g_\gm\ksup{1}, \nu g_{\gm\cup\dl}\ksup{1}) &=\nu g_\dl\ksup{2}\\
&=\nu x_1x_5 - \nu x_2x_4\\
S(\nu g_\dl\ksup{1}, \nu g_{\gm\cup\dl}\ksup{1}) &= \nu\left(g_{\gm\cup\dl}^{(2),\out}g_\dl^{(2),\out}-g_{\gm\cup\dl}^{(2),\into}g_\dl^{(2),\into}\right)\\
&= \nu x_1x_5^2 - \nu x_2x_4^2.
\end{align*}
All three of these S-polynomials reduce as described in Lemma~\ref{subsetoverlap}. To determine which case is relevant, notice that $x_1$ is internal to $\gm\cup\dl$ and not to $\gm$ or $\dl$, and that $x_1$ divides exactly one term of $g_\gm\ksup{1}$ and of $g_\dl\ksup{1}$. Therefore:
\begin{align*}
&S(\nu g_\gm\ksup{1}, \nu g_\dl\ksup{1})\xrightarrow{\widetilde{g}_\dl,\widetilde{g}_\gm}0\quad\text{by Case 3 in the proof of Lemma~\ref{subsetoverlap};}\\
&S(\nu g_\gm\ksup{1}, \nu g_{\gm\cup\dl}\ksup{1})\xrightarrow{\nu g_\dl\ksup{1}}0\quad\text{by Case 2(b) in the proof of Lemma~\ref{subsetoverlap}; and}\\
&S(\nu g_\dl\ksup{1}, \nu g_{\gm\cup\dl}\ksup{1}) \xrightarrow{\nu x_1-x_1, \widetilde{g}_\dl, x_1g_{\gm\cup\dl}\ksup{2}}0\quad\text{by Case 2(b).}
\end{align*}
Note also that the full generality of Case 2(b) is not needed in the second computation above because $g_\dl\ksup{2}=g_\dl\ksup{1}$. 

We have now computed all of the S-polynomials among generators in the original basis $\cG_0$. Although reducing these S-polynomials by hand was straightforward enough, the computations might seem rather ad hoc, and would seem more so if we had considered all of the alternative ways to reduce each S-polynomial. Standardizing the reduction procedures is crucial to being able to generalize to arbitrary pairs of subsets in arbitrary graphs. In turn, characterizing the output of these S-polynomials in terms of operations on subsets is crucial to standardizing the reduction procedures. This, in brief, is the content of Section~\ref{subsetspolys}.

\subsection{Remaining S-polynomials}
It remains to check that S-polynomials involving elements of $\cG^\prime\setminus\cG_0$ reduce to zero within $\cG^\prime$. We leave it to the unusually detail-oriented reader to confirm the necessary calculations, but state the results with references to the relevant arguments in Section~\ref{bbrd2}.

It may seem that many of the calculations in this section are redundant. For example, the various S-polynomials involving generators associated to $\gm$ and $\dl$ almost all produce a multiple of $g_{\gm\cup\dl}\ksup{2}$. However, $g_{\gm\cup\dl}\ksup{2}$ itself is not in the working basis, hence not available for reductions. Therefore, the multiple makes a difference: $\nu$ may allow us to reduce by a multiple of $\nu g_{\gm\cup\dl}\ksup{1}$; $x_1$ may allow us to reduce by a multiple of $g_{\gm\cup\dl}\ksup{2}$. Sometimes, as for $S(\nu g_\gm\ksup{1}, \nu g_\dl\ksup{1})$ above, these simple reductions are impossible. We must turn instead to tilde generators or $\nu x_1-x_1$. So, despite the similarity of the remaining S-polynomial calculations, the reduction arguments are delicate.

It is also worth noting that S-polynomials involving tilde generators do not behave like the S-polynomials involving their non-tilde counterparts. The generator $\widetilde{g}_\gm$ is effectively $\nu g_\gm\ksup{1}$ with its term order reversed. Therefore, $S(\widetilde{g}_\gm,-)$ bears little relation to $S(\nu g_\gm\ksup{1},-)$. For example, $S(\nu g_\gm\ksup{1},\nu g_{\gm\cup\dl}\ksup{1})$ reduced via $g_\dl\ksup{1}$ while $S(\widetilde{g}_\gm, \nu g_{\gm\cup\dl}\ksup{1})$ reduces via $x_1g_{\gm\cup\dl}\ksup{2}$ and $x_1g_\gm\ksup{2}$. 

The argument at the beginning of Section~\ref{bbrd2}, referring to Statement (6) of Proposition~\ref{nutopone}, takes care of S-polynomials involving $\nu x_1-x_1$ and elements of $\cG^\prime\setminus\cG_0$. It confirms that
$S(\nu x_1-x_1,x_1g_\Lambda\ksup{2})\xrightarrow{\nu x_1-x_1,x_1g_\Lambda\ksup{2}}0
$
for $\Lambda\in\left\{\gm,\dl,\gm\cup\dl\right\}$. 

The next argument in Section~\ref{bbrd2}, which refers to Proposition~\ref{gbprinciplecoeff} and Lemma~\ref{subsetoverlap}, concerns S-polynomials involving pairs of elements of the form $x_1g_\Lambda\ksup{2}$. It confirms that
\begin{align*}
&S(x_1g_\gm\ksup{2},x_1g_\dl\ksup{2})\xrightarrow{x_1g_{\gm\cup\dl}\ksup{2}}0;\\
&S(x_1g_\gm\ksup{2},x_1g_{\gm\cup\dl}\ksup{2})\xrightarrow{x_1g_{\dl}\ksup{2}}0; \text{ and}\\
&S(x_1g_\dl\ksup{2},x_1g_{\gm\cup\dl}\ksup{2})\xrightarrow{x_1g_\dl\ksup{2},x_1g_{\gm\cup\dl}\ksup{2}}0.
\end{align*}

Lemma~\ref{2nunonu} concerns S-polynomials involving one generator of the form $\nu g_\Lambda\ksup{1}$ and one of the form $x_1g_\Lambda\ksup{2}$. It applies to
\begin{align*}
&S(\nu g_\gm\ksup{1},x_1 g_\gm\ksup{2})=0 \quad\text{(Case 3, but not in its full generality);}\\
&S(\nu g_\gm\ksup{1},x_1 g_\dl\ksup{2})\xrightarrow{\nu g_{\gm\cup\dl}\ksup{1}}0 \quad\text{(Case 3, but not in its full generality); and}\\
&S(\nu g_\gm\ksup{1},x_1 g_{\gm\cup\dl}\ksup{2})=S(\nu g_\gm\ksup{1}, \nu g_{\gm\cup\dl}\ksup{1})\xrightarrow{\nu g_{\dl}\ksup{1}}0 \quad\text{(before the break-down to cases).}\\
\end{align*}
Similarly, it applies to 
\begin{align*}
&S(\nu g_\dl\ksup{1},x_1 g_\gm\ksup{2})\xrightarrow{\nu g_{\gm\cup\dl}\ksup{1}}0 \quad\text{(Case 3, but not in its full generality)}\\
&S(\nu g_\dl\ksup{1},x_1 g_\dl\ksup{2})=0 \quad\text{(Case 3, but not in its full generality); and}\\
&S(\nu g_\dl\ksup{1},x_1 g_{\gm\cup\dl}\ksup{2})=S(\nu g_\dl\ksup{1}, \nu g_{\gm\cup\dl}\ksup{1})\xrightarrow{\nu x_1-x_1,\widetilde{g}_\dl,x_1g_{\gm\cup\dl}\ksup{2}}0 \\&\text{(before the break-down to cases).}
\end{align*}
Finally, it applies to
\begin{align*}
&S(\nu g_{\gm\cup\dl}\ksup{1},x_1 g_\gm\ksup{2})\xrightarrow{x_1 g_{\dl}\ksup{2}}0 \quad\text{(Case 1);}\\
&S(\nu g_{\gm\cup\dl}\ksup{1},x_1 g_\dl\ksup{2})\xrightarrow{x_1g_\dl\ksup{2},x_1g_{\gm\cup\dl}\ksup{2}}0 \quad\text{(Case 1); and}\\
&S(\nu g_{\gm\cup\dl}\ksup{1},x_1 g_{\gm\cup\dl}\ksup{2})=0 \quad\text{(before the break-down to cases).}
\end{align*}

Lemma~\ref{tildeprop} concerns S-polynomials between tilde generators and generators of the form $\nu g_\Lambda\ksup{1}$. Case 1 applies to
\begin{align*}
&S(\nu g_{\gm\cup\dl}\ksup{1},\widetilde{g}_\gm)\xrightarrow{\nu x_1-x_1, x_1g_{\gm\cup\dl}\ksup{2}, x_1g_\gm\ksup{2}}0\quad\text{and}\\
&S(\nu g_{\gm\cup\dl}\ksup{1},\widetilde{g}_\dl)\xrightarrow{\nu x_1-x_1, x_1g_{\gm}\ksup{2}}0.
\end{align*}
Case 2 applies to
\begin{align*}
&S(\nu g_\gm\ksup{1}, \widetilde{g}_\gm)\xrightarrow{\widetilde{g}_\gm,x_1g_\gm\ksup{2}}0;\\
&S(\nu g_\gm\ksup{1}, \widetilde{g}_\dl)\xrightarrow{\widetilde{g}_\gm,x_1g_\dl\ksup{2}}0;\\
&S(\nu g_\dl\ksup{1}, \widetilde{g}_\gm)\xrightarrow{\widetilde{g}_\dl,x_1g_\gm\ksup{2}}0;\text{ and}\\
&S(\nu g_\dl\ksup{1}, \widetilde{g}_\dl)\xrightarrow{\widetilde{g}_\dl,x_1g_\dl\ksup{2}}0.
\end{align*}

Finally, we have
\begin{align*}
&S(\widetilde{g}_\gm,\widetilde{g}_\dl)\xrightarrow{x_1g_{\gm\cup\dl}\ksup{2}}0\quad\text{by Lemma~\ref{doubletilde} and}\\
&S(\widetilde{g}_\dl,x_1g_{\gm\cup\dl}\ksup{2})\xrightarrow{\nu x_1-x_1,x_1g_\dl\ksup{2},x_1g_{\gm\cup\dl}\ksup{2}}0\quad\text{by Lemma~\ref{ztautilde}.}
\end{align*}

The remaining pairs of elements in $\cG^\prime$ involving at least one element of $\cG^\prime\setminus\cG_0$ have no common divisors in their leading monomials, so their S-polynomials reduce to zero by Proposition~\ref{gbprinciplegcd}.

\subsection{Calculating the ideal quotient}
We have checked that all S-polynomials among elements of $\cG^\prime$ reduce to zero within $\cG^\prime$. Therefore, Buchberger's Algorithm has terminated and $\cG^\prime$ is a Gr\"obner basis for $\nu\pi_1(N_1)+(\nu-1)(x_1)$. That is, we have verified Lemma~\ref{buchend} for this example. We now intersect with $\cE_2$ to obtain a basis for $\pi_1(N_1)\cap(x_1)$:
\[\cG^\prime\cap\cE_2=\left\{x_1g_\gm\ksup{2},x_1g_\dl\ksup{2},x_1g_{\gm\cup\dl}\ksup{2}\right\}.\] Divide each of these generators by $x_1$ to obtain a basis for the quotient
\[\pi_1(N_1):(x_1)=\left(g_\gm\ksup{2},g_\dl\ksup{2},g_{\gm\cup\dl}\ksup{2}\right).\] This basis is the defining basis for $N_2$.

\section{Implementing Buchberger's Algorithm: Round 1}
\label{bbrd1}

As we compute S-polynomials, we will record the results in tables showing the propositions used and whether the result of the S-polynomial was added to the working basis. Table~\ref{round1table} records the S-polynomials we compute in this section. 

\begin{table}[htbp]
\begin{center}
\renewcommand\arraystretch{1.5}
\begin{tabular}{|c|c|c|c|}
\hline
$S\!\left(-,-\right)$ & Result & Proposition & Add to $\cG^\prime$? \\\hline\hline
$S\!\left(\nu g_\gm\ksup{k},\nutopk\right)$ & $\ztau\ksup{k+1} g_\gm\ksup{k+1}$ or $\widetilde{g}_\gm$ & Prop.~\ref{nutopone} & yes \\\hline
$S\!\left(\widetilde{g}_\gm,\nutopk\right)$ & $\ztau\ksup{k+1} g_\gm\ksup{k+1}$ & Prop.~\ref{nutopone} & yes \\\hline
$S\!\left(\nu g_\gm\ksup{k},\nu g_\dl\ksup{k}\right)$ & 0 & Lemma~\ref{subsetoverlap} & \text{no} \\\hline
\end{tabular}
\end{center}\medskip
\caption{S-polynomials Round 1: All computations are assumed to be among generators in $\cG^\prime$ and are carried out in $\cE_{k+1}[\nu]$. S-polynomials are listed in the order they are computed in Section~\ref{bbrd1}.}
\label{round1table}
\end{table}

\subsection{S-polynomials with $\nutopk$}
\label{nutoponespolys}
We begin by describing the behavior of $S(\nutopk,-)$ with respect to various types of polynomials in $\cE_{k+1}[\nu]$. 

\begin{prop}
\label{nutopone}
Let $f\in\cE_{k+1}$ and $f=\lt{f}+\overf$.
If $\gcd(\lm{f},\ztau\ksup{k+1})=1$, then 
\begin{align}
\tag{1}\label{nutop2nunodiv}&S(\nutopk,\nu f)\xrightarrow{\nutopk}-\frac{\ztau\ksup{k+1} f}{\lc{f}}\\
\tag{2}\label{nutop1nunodiv}&S(\nutopk, \nu\lt{f} + \overf)=-\frac{\ztau\ksup{k+1} f}{\lc{f}}\\
\tag{3}\label{nutop0nunodiv}&S(\nutopk, f)\xrightarrow{\nutopk}-\frac{\ztau\ksup{k+1} f}{\lc{f}}
\end{align}
\noi If $\gcd(\lm{f},\ztau\ksup{k+1})=\ztau\ksup{k+1}$, then 
\begin{align}
\tag{4}\label{nutop2nu}&S(\nutopk,\nu f)=-\frac{1}{\lc{f}}\left(\nu \overf+\lt{f}\right)\\
\tag{5}\label{nutop1nu}&S(\nutopk, \nu\lt{f}+\overf)=-\frac{f}{\lc{f}}\\
\tag{6}\label{nutop0nu}&S(\nutopk, f)=-\frac{1}{\lc{f}}\left(\nu \overf+\lt{f}\right)
\end{align}
\end{prop}
\begin{proof}
The least common multiple of the leading monomials in the first three cases is $\nu\ztau\ksup{k+1}\lm{f}$. We calculate the first S-polynomial above as follows.
\begin{align*}
S(\nutopk,\nu f)&=\frac{\nu\ztau\ksup{k+1}\lm{f}}{\nu\ztau\ksup{k+1}}\left(-\ztau\ksup{k+1}\right)-\frac{\nu\ztau\ksup{k+1}\lm{f}}{\nu\lt{f}}\left(\nu \overf\right)\\
&=-\frac{\nu\ztau\ksup{k+1} \overf}{\lc{f}}-\ztau\ksup{k+1}\lm{f}\quad\quad\quad\text{LT determined by }\nu\\
\textrm{reduce}\quad&+\frac{\overf}{\lc{f}}\left(\nutopk\right)\\
&=-\frac{\ztau\ksup{k+1} f}{\lc{f}}
\end{align*}
The second and third claims come from similar calculations.

In the latter three cases, the least common multiple of the leading monomials is $\nu\lm{f}$. Given this, we calculate as follows.
\begin{align*}
S(\nutopk,\nu f)=&\frac{\nu\lm{f}}{\nu\ztau\ksup{k+1}}\left(-\ztau\ksup{k+1}\right)-\frac{\nu\lm{f}}{\nu\lt{f}}\left(\nu \overf\right)&\\
=&-\frac{\nu \overf}{\lc{f}}-\lm{f}\quad\quad\quad\quad\quad\text{LT determined by }\nu\\
\end{align*}
The fifth and sixth cases are similar.
\end{proof}


We apply Proposition~\ref{nutopone} to compute $S(\nutopk,\nu g_\gm\ksup{k})$ and see which new generators must be added to the working basis, keeping in mind that leading coefficients are currently assumed to be 1. See the flowchart in Figure~\ref{flowchart}. If $\ztau\ksup{k+1}$ does not divide the leading term of $g_\gm\ksup{k}$, then Statement~(\ref{nutop2nunodiv}) of Proposition~\ref{nutopone} applies, so $S(\nutopk,\nu g_\gm\ksup{k})\xrightarrow{\nutopk}-\ztau\ksup{k+1}g_\gm\ksup{k}$. Since $\ztau\ksup{k+1}$ does not divide both terms of $g_\gm\ksup{k}$,  we have $g_\gm\ksup{k}=g_\gm\ksup{k+1}$. So we may say that it is $\ztau\ksup{k+1}g_\gm\ksup{k+1}$ that should be added to the working basis.

If $\ztau\ksup{k+1}$ does divide the leading term of $g_\gm\ksup{k}$, then Statement~(\ref{nutop2nu}) of Proposition~\ref{nutopone} applies, and  $S(\nutopk,\nu g_\gm\ksup{k}) = -\widetilde{g}_\gm$. Recall from Equation~\ref{tildegendfn} that
\begin{equation*}
\widetilde{g}_\gm=\begin{cases} \nu g_\gm^{(k),\into}-g_\gm^{(k),\out}& \text{ if } \gm \text{ is out-led}\\ 
\nu g_\gm^{(k),\out}-g_\gm^{(k),\into}& \text{ if } \gm \text{ is in-led.}
\end{cases}
\end{equation*} 
If $\ztau\ksup{k+1}$ also divides the trailing term of $g_\gm\ksup{k}$, then $\widetilde{g}_\gm$ reduces via $\nutopk$ to leave $g_\gm\ksup{k}$. In this case, $\ztau\ksup{k+1}$ must have been both an outgoing and an incoming edge to $\gm$ in $G\ksup{k+1}$, so $g_\gm\ksup{k}=\ztau\ksup{k+1} g_\gm\ksup{k+1}$. Therefore, if $\ztau\ksup{k+1}$ divides $g_\gm\ksup{k}$, we end up adding $\ztau\ksup{k+1} g_\gm\ksup{k+1}$ and not $\widetilde{g}_\gm$ to the working basis. These results are recorded in Table~\ref{round1table}.

The only case in which we have added $\widetilde{g}_\gm$ and not $\ztau\ksup{k+1}g_\gm\ksup{k+1}$ to $\cG^\prime$ is when $\ztau\ksup{k+1}$ divides the leading term but not the trailing term of $g_\gm\ksup{k}$. For convenience, we immediately compute S-polynomials $S(\nutopk, \widetilde{g}_\gm)$ in this case: Statement~(\ref{nutop1nunodiv}) of Proposition~\ref{nutopone} implies that $S(\nutopk,\widetilde{g}_\gm)=\ztau\ksup{k+1}g_\gm\ksup{k}$. Since $\ztau\ksup{k+1}$ divided only one term of $g_\gm\ksup{k}$, we also have $g_\gm\ksup{k}=g_\gm\ksup{k+1}$. Therefore, we may record that we are adding $\ztau\ksup{k+1} g_\gm\ksup{k+1}$ to the working basis in this case as well. 

Putting all of this together, we have produced $\ztau\ksup{k+1}g_\gm\ksup{k+1}$ for all $\gm$. The working basis is now
\begin{align*}
\cG^\prime=\,\cG_0&\cup\left\{\widetilde{g}_\gm\,\vert\,\gm\subset G, \ztau\ksup{k+1}\vert\lt{g_\gm\ksup{k}}, \ztau\ksup{k+1}\nmid\trt{g_\gm\ksup{k}}\right\}\nonumber\\&\cup\left\{\ztau\ksup{k+1}g_\gm\ksup{k+1}\,\vert\,\gm\subset G\right\}.
\end{align*}

\subsection{S-polynomials among the $\nu g_\gm\ksup{k}$}
\label{subsetspolys}
Our goal in this section is to describe the results of S-polynomials among generators of the form $\nu g_\gm\ksup{k}$ in terms of generators associated to related subsets. We first establish a general principle that will allow us to tackle products of interior edges (i.e., monomials labeled $x$) separately from products of boundary and closure edges (i.e., monomials labeled $z$).

\begin{prop}
\label{separateinteriorclosure}
Let $f_x,\overline{f}_x,f_z,\overline{f}_z,g_x,\overline{g}_x,g_z,\overline{g}_z\in\mathbb{Q}[\underline{x}^\prime]$ be monomials with the property that any monomial with an $x$ subscript is relatively prime to any monomial with a $z$ subscript. Let $S(f_x+\overline{f}_x,g_x+\overline{g}_x)_1$ and $S(f_x+\overline{f}_x,g_x+\overline{g}_x)_2$ denote the first and second terms of $S(f_x+\overline{f}_x,g_x+\overline{g}_x)$ as written in the definition of S-polynomial in Section~\ref{gbbackground}, not necessarily with respect to the monomial order, and similarly for $S(f_z+\overline{f}_z,g_z+\overline{g}_z)$. Then 
\begin{align*}
S(f_xf_z+\overline{f}_x \overline{f}_z, g_xg_z+\overline{g}_x \overline{g}_z)&= S(f_x+\overline{f}_x,g_x+\overline{g}_x)_1S(f_z+\overline{f}_z,g_z+\overline{g}_z)_1\\
 &\quad- S(f_x+\overline{f}_x,g_x+\overline{g}_x)_2S(f_z+\overline{f}_z,g_z+\overline{g}_z)_2\\
&\xrightarrow{S(f_x+\overline{f}_x,g_x+\overline{g}_x),S(f_z+\overline{f}_z,g_z+\overline{g}_z)}0
\end{align*}
\end{prop}
\begin{proof}
The assumptions about greatest common divisors among the monomials mean that 
$$\lcm{f_xf_z,g_xg_z}=\frac{f_xf_zg_xg_z}{\gcd(f_x,g_x)\gcd(f_z,g_z)}=\lcm{f_x,g_x}\lcm{f_z,g_z}.$$
The S-polynomial calculations proceed as follows.
\begin{align*}
&S(f_xf_z+\overline{f}_x \overline{f}_z, g_xg_z+\overline{g}_x \overline{g}_z)=\\
&\quad\frac{g_xg_z}{\gcd(f_x,g_x)\gcd(f_z,g_z)}\overline{f}_x \overline{f}_z-\frac{f_xf_z}{\gcd(f_x,g_x)\gcd(f_z,g_z)}\overline{g}_x \overline{g}_z\\
&\quad\frac{g_x}{\gcd(f_x,g_x)}\overline{f}_x\frac{g_z}{\gcd(f_z,g_z)}\overline{f}_z-\frac{f_x}{\gcd(f_x,g_x)}\overline{g}_x\frac{f_z}{\gcd(f_z,g_z)}\overline{g}_z
\end{align*}
The term order in this expression is not clear. However, the first term is a product of the first terms of $S(f_x+\overline{f}_x,g_x+\overline{g}_x)$ and $S(f_z+\overline{f}_z,g_z+\overline{g}_z)$ and the second term is a product of their second terms, assuming everything is written as in Definition~\ref{spolydfn}, not necessarily with respect to the monomial order. Regardless of the correct term order, Proposition~\ref{unorderedreduction} shows that expressions of this form reduce to zero by their constituent parts. 
\end{proof}

\begin{prop}
\label{subsetoverlapinterior}
Let $\gm, \dl\subset G$. Assume term orders of the S-polynomial input are as written. Then the following statements hold in any $\cE_i$. 
\begin{align*}\tag{1}\label{subsetinteriorout}
S(\xsub{\gm}{G\sm\gm}&-\xsub{G\sm\gm}{\gm},\xsub{\dl}{G\sm\dl}-\xsub{G\sm\dl}{\dl})\\
&=\, \xsub{G\sm(\gm\cup\dl)}{\gm\cap\dl}\left(x_{\dl\sm(\gm\cap\dl)}^{\out}x_{\gm\sm(\gm\cap\dl)}^{\into}-x_{\gm\sm(\gm\cap\dl)}^{\out}x_{\dl\sm(\gm\cap\dl)}^{\into}\right)\\
\tag{2}\label{subsetinteriorin}
S(\xsub{G\sm\gm}{\gm}&-\xsub{\gm}{G\sm\gm},\xsub{G\sm\dl}{\dl}-\xsub{\dl}{G\sm\dl})\\
&=\, \xsub{\gm\cap\dl}{G\sm(\gm\cup\dl)}\left(x_{\dl\sm(\gm\cap\dl)}^{\into}x_{\gm\sm(\gm\cap\dl)}^{\out}-x_{\gm\sm(\gm\cap\dl)}^{\into}x_{\dl\sm(\gm\cap\dl)}^{\out}\right)\\
\tag{3}\label{subsetinteriormix}
S(\xsub{G\sm\gm}{\gm}&-\xsub{\gm}{G\sm\gm},\xsub{\dl}{G\sm\dl}-\xsub{G\sm\dl}{\dl})\\
&=\,\xsub{\gm\sm(\gm\cap\dl)}{\dl\sm(\gm\cap\dl)}\left(x_{\gm\cup\dl}^\out x_{\gm\cap\dl}^\out
-x_{\gm\cup\dl}^\into x_{\gm\cap\dl}^\into\right)
\end{align*}

\noindent The term orders of the results in the first two statements are undetermined in general.
\end{prop}
\begin{proof}
This proof is mainly a long calculation. It holds in any $\cE_i$ because it involves only interior edges and the projection of edge rings $\pi_i$ is the identity when restricted to the subring of $\cE_i$ generated by such edges. The outcome of the calculation in all cases relies on the fact that least common multiples and greatest common divisors of monomials behave in the same way as union and intersection of subsets. Statement~(\ref{subsetinteriorin}) follows from Statement~(\ref{subsetinteriorout}) by taking complements, so we exhibit the calculation only in the first and last cases.  \smallskip

\noindent\emph{Case 1:}
The greatest common divisor of the leading monomials is $\xsub{\gd{\cap}}{G\sm(\gd{\cup})}$ so the least common multiple is 
$$
\frac{\xsub{\gm}{G\sm\gm}\xsub{\dl}{G\sm\dl}}
{\xsub{\gd{\cap}}{G\sm(\gd{\cup})}}.
$$
The least common multiple divided by each leading term is 
\begin{align*}
&\frac{\xsub{\gm}{G\sm\gm}\xsub{\dl}{G\sm\dl}}{\xsub{\gd{\cap}}{G\sm(\gd{\cup})}\xsub{\gm}{G\sm\gm}}
=\xsub{\dl\sm(\gd{\cap})}{G\sm(\gd{\cup})}\xsub{\dl}{\gm\sm(\gd{\cap})}\\
&\frac{\xsub{\gm}{G\sm\gm}\xsub{\dl}{G\sm\dl}}{\xsub{\gd{\cap}}{G\sm(\gd{\cup})}\xsub{\dl}{G\sm\dl}}
=\xsub{\gm\sm(\gd{\cap})}{G\sm(\gd{\cup})}\xsub{\gm}{\dl\sm(\gd{\cap})}
\end{align*}
We may now compute the S-polynomial. Expanding, then regrouping produces the form claimed in the proposition. Term order is unknown throughout.
\begin{align*}
&S(\xsub{\gm}{G\sm\gm}-\xsub{G\sm\gm}{\gm},\xsub{\dl}{G\sm\dl}-\xsub{G\sm\dl}{\dl})\\
&=\xsub{\dl\sm(\gd{\cap})}{G\sm(\gd{\cup})}\xsub{\dl}{\gm\sm(\gd{\cap})}\xsub{G\sm\gm}{\gm}
+\xsub{\gm\sm(\gd{\cap})}{G\sm(\gd{\cup})}\xsub{\gm}{\dl\sm(\gd{\cap})}\xsub{G\sm\dl}{\dl}\\
&=\xsub{\dl\setminus(\gm\cap\dl)}{G\setminus(\gm\cup\dl)}
\xsub{\dl\setminus(\gm\cap\dl)}{\gm\setminus(\gm\cap\dl)}
\xsub{\gm\cap\dl}{\gm\setminus(\gm\cap\dl)}\\
&\quad\cdot
\xsub{\dl\setminus(\gm\cap\dl)}{\gm\setminus(\gm\cap\dl)}
\xsub{G\setminus(\gm\cup\dl)}{\gm\setminus(\gm\cap\dl)}
\xsub{\dl\setminus(\gm\cap\dl)}{\gm\cap\dl}
\xsub{G\setminus(\gm\cup\dl)}{\gm\cap\dl}\\
&+
\xsub{\gm\setminus(\gm\cap\dl)}{G\setminus(\gm\cup\dl)}
\xsub{\gm\setminus(\gm\cap\dl)}{\dl\setminus(\gm\cap\dl)}
\xsub{\gm\cap\dl}{\dl\setminus(\gm\cap\dl)}\\
&\quad\cdot
\xsub{\gm\setminus(\gm\cap\dl)}{\dl\setminus(\gm\cap\dl)}
\xsub{G\setminus(\gm\cup\dl)}{\dl\setminus(\gm\cap\dl)}
\xsub{\gm\setminus(\gm\cap\dl)}{\gm\cap\dl}
\xsub{G\setminus(\gm\cup\dl)}{\gm\cap\dl}\\\smallskip
&=\xsub{G\sm(\gm\cup\dl)}{\gm\cap\dl}(\xsub{\dl\sm(\gm\cap\dl)}{G\sm\dl}\xsub{\gm\cap\dl}{\gm\sm(\gm\cap\dl)}\xsub{G\sm\gm}{\gm\sm(\gm\cap\dl)}\xsub{\dl\sm(\gm\cap\dl)}{\gm\cap\dl}\\
&-\xsub{\gm\sm(\gm\cap\dl)}{G\sm\gm}\xsub{\gm\cap\dl}{\dl\sm(\gm\cap\dl)}\xsub{G\sm\dl}{\dl\sm(\gm\cap\dl)}\xsub{\gm\sm(\gm\cap\dl)}{\gm\cap\dl})\\
&=\xsub{G\sm(\gm\cup\dl)}{\gm\cap\dl}\left(x_{\dl\sm(\gm\cap\dl)}^{\out}x_{\gm\sm(\gm\cap\dl)}^{\into}-x_{\gm\sm(\gm\cap\dl)}^{\out}x_{\dl\sm(\gm\cap\dl)}^{\into}\right)
\end{align*}


\noindent\emph{Case 3:}
The greatest common divisor of leading monomials in this case is the product of the edges that go from $\dl\sm(\gm\cap\dl)$ to $\gm\sm(\gm\cap\dl)$. The S-polynomial removes those edges, which are internal to $\gm\cup\dl$, while combining the incoming edges of $\gm$ with those of $\dl$ and the outgoing edges of $\gm$ with those of $\dl$. Specifically, the greatest common divisor of the leading monomials is
$
\xsub{\dl\sm(\gm\cap\dl)}{\gm\sm(\gm\cap\dl)},$
so the least common multiple is
\begin{align*}
&\frac{\xsub{G\sm\gm}{\gm}\xsub{\dl}{G\sm\dl}}{\xsub{\dl\sm(\gm\cap\dl)}{\gm\sm(\gm\cap\dl)}}\\
&=\xsub{\dl\sm(\gm\cap\dl)}{\gm\sm(\gm\cap\dl)}
\xsub{G\sm(\gm\cup\dl)}{\gm\sm(\gm\cap\dl)}\xsub{G\sm\gm}{\gm\cap\dl}
\xsub{\dl\sm(\gm\cap\dl)}{G\sm(\gm\cup\dl)}\xsub{\gm\cap\dl}{G\sm\dl}
\end{align*}
and the S-polynomial calculation is as follows.
\begin{align}
\nonumber S(\xsub{G\sm\gm}{\gm}-\xsub{\gm}{G\sm\gm},&\xsub{\dl}{G\sm\dl}-\xsub{G\sm\dl}{\dl})=
\xsub{\dl\sm(\gm\cap\dl)}{G\sm(\gm\cup\dl)}\xsub{\gm\cap\dl}{G\sm\dl}
\cdot\xsub{\gm}{G\sm\gm}\\
\label{altfactorzn}&\phantom{\xsub{\dl}{G\sm\dl}-\xsub{G\sm\dl}{\dl})}-\xsub{G\sm(\gm\cup\dl)}{\gm\sm(\gm\cap\dl)}\xsub{G\sm\gm}{\gm\cap\dl}
\cdot\xsub{G\sm\dl}{\dl}\\
\nonumber =&\,
\xsub{\dl\sm(\gm\cap\dl)}{G\sm(\gm\cup\dl)}\xsub{\gm\cap\dl}{G\sm(\gm\cup\dl)}\xsub{\gm\cap\dl}{\gm\sm(\gm\cap\dl)}\\
\nonumber \cdot\,&\xsub{\gm\sm(\gm\cap\dl)}{G\sm(\gm\cup\dl)}\xsub{\gm\sm(\gm\cap\dl)}{\dl\sm(\gm\cap\dl)}\xsub{\gm\cap\dl}{G\sm(\gm\cup\dl)}\xsub{\gm\cap\dl}{\dl\sm(\gm\cap\dl)}\\
\nonumber -&\xsub{G\sm(\gm\cup\dl)}{\gm\sm(\gm\cap\dl)}\xsub{G\sm(\gm\cup\dl)}{\gm\cap\dl}\xsub{\dl\sm(\gm\cap\dl)}{\gm\cap\dl}\\
\nonumber \cdot\,&\xsub{G\sm(\gm\cup\dl)}{\dl\sm(\gm\cap\dl)} \xsub{G\sm(\gm\cup\dl)}{\gm\cap\dl} \xsub{\gm\sm(\gm\cap\dl)}{\dl\sm(\gm\cap\dl)} \xsub{\gm\sm(\gm\cap\dl)}{\gm\cap\dl}\\
\nonumber =&\,\xsub{\gm\sm(\gm\cap\dl)}{\dl\sm(\gm\cap\dl)}
\nonumber (\xsub{\gm\cup\dl}{G\sm(\gm\cup\dl)}\xsub{\gm\cap\dl}{G\sm(\gm\cap\dl)}\\
\nonumber &\,-\xsub{G\sm(\gm\cup\dl)}{\gm\cup\dl}\xsub{G\sm(\gm\cap\dl)}{\gm\cap\dl})
\end{align}
\end{proof}

So far, we have established that S-polynomials of the interior edge portions of the $g_\gm\ksup{k}$ can always be written in terms of the interior edge portions of generators of the same form associated to unions, intersections, and complements of the original subsets. The next task is to consider the boundary and closure edge portions of the $g_\gm\ksup{k}$. In consideration of S-polynomials that will need to be computed later in the algorithm, we do the necessary computations for the boundary and closure edge portions of $g\ksup{k+1}_\gm\in\cE_{k+1}$, which have the form $\zsub{\gm}{G\sm\gm}\ksup{k+1}\zsub{\gm}{G\sm\gm}\ksup{k+1}-\zsub{G\sm\gm}{\gm}\ksup{k+1}\zsub{\beta}{\gm}\ksup{k+1}$. In each of these terms, the first factor is divisible by $\ztau\ksup{i}$ only if $i\leq k+1$ and the second only if $i>k+1$, so the two factors are relatively prime. In light of Proposition~\ref{separateinteriorclosure}, we may separate the S-polynomial computations for products of boundary edges like $\zsub{\gm}{\tau}\ksup{k+1}$ from those for products of closure edges like $\zsub{\gm}{G\sm\gm}\ksup{k+1}$. For the closure edges, the computations will be identical to those we did for interior edges, but we state the results below for completeness. For the boundary edges, the computations are slightly different, but still straightforward. They are outlined in Proposition~\ref{subsetoverlapboundary}.


\begin{prop}
\label{subsetoverlapclosure}
Let $\gm, \dl\subset G$.  Assume term orders of the S-polynomial input are as written. The following equalities hold in $\cE_{k+1}$. 
\begin{align*}
\tag{1}\label{subsetclosureout}
S(\zsub{\gm}{G\sm\gm}\ksup{k+1}&-\zsub{G\sm\gm}{\gm}\ksup{k+1}, \,
\zsub{\dl}{G\sm\dl}\ksup{k+1}-\zsub{G\sm\dl}{\dl}\ksup{k+1})
=\zsub{G\sm(\gm\cup\dl)}{\gm\cap\dl}\ksup{k+1}\\
&\cdot\left(\zsub{\dl\sm(\gm\cap\dl)}{G\sm\dl}\ksup{k+1}\zsub{\dl\sm(\gm\cap\dl)}{\gm\cap\dl}\ksup{k+1}
\zsub{\gm\cap\dl}{\gm\sm(\gm\cap\dl)}\ksup{k+1}\zsub{G\sm\gm}{\gm\sm(\gm\cap\dl)}\ksup{k+1}
\right. \\
&-
\left.\zsub{\gm\sm(\gm\cap\dl)}{\gm\cap\dl}\ksup{k+1}\zsub{\gm\sm(\gm\cap\dl)}{G\sm\gm}\ksup{k+1}\zsub{\gm\cap\dl}{\dl\sm(\gm\cap\dl)}\ksup{k+1}\zsub{G\sm\dl}{\dl\sm(\gm\cap\dl)}\ksup{k+1}\right)\text{;}\\
\tag{2}\label{subsetclosurein}
S(\zsub{G\sm\gm}{\gm}\ksup{k+1}&-\zsub{\gm}{G\sm\gm}\ksup{k+1},
\,\zsub{G\sm\dl}{\dl}\ksup{k+1}-\zsub{\dl}{G\sm\dl}\ksup{k+1})
=\zsub{\gm\cap\dl}{G\sm(\gm\cup\dl)}\ksup{k+1}\\
&\cdot\left(\zsub{\gm\cap\dl}{\dl\sm(\gm\cap\dl)}\ksup{k+1}\zsub{G\sm\dl}{\dl\sm(\gm\cap\dl)}\ksup{k+1}
\zsub{\gm\sm(\gm\cap\dl)}{\gm\cap\dl}\ksup{k+1}\zsub{\gm\sm(\gm\cap\dl)}{G\sm\gm}\ksup{k+1}\right.\\
&-
\left.\zsub{\gm\cap\dl}{\gm\sm(\gm\cap\dl)}\ksup{k+1}\zsub{G\sm\gm}{\gm\sm(\gm\cap\dl)}\ksup{k+1}
\zsub{\dl\sm(\gm\cap\dl)}{G\sm\dl}\ksup{k+1}\zsub{\dl\sm(\gm\cap\dl)}{\gm\cap\dl}\ksup{k+1}
\right)
\text{; and}\\
\tag{3}\label{subsetclosuremix}
S(\zsub{G\sm\gm}{\gm}\ksup{k+1}&-\zsub{\gm}{G\sm\gm}\ksup{k+1}, \,
\zsub{\dl}{G\sm\dl}\ksup{k+1}-\zsub{G\sm\dl}{\dl}\ksup{k+1})
=\zsub{\gm\sm(\gm\cap\dl)}{\dl\sm(\gm\cap\dl)}\ksup{k+1}\\
&\cdot\left(\zsub{\gm\cup\dl}{G\sm(\gm\cup\dl)}\ksup{k+1}\zsub{\gm\cap\dl}{G\sm(\gm\cap\dl)}\ksup{k+1}
-\zsub{G\sm(\gm\cup\dl)}{\gm\cup\dl}\ksup{k+1}\zsub{G\sm(\gm\cap\dl)}{\gm\cap\dl}\ksup{k+1}\right).
\end{align*}

\noindent In all cases, the term orders of the results are undetermined in general.
\end{prop}


\begin{proof}
These are all straightforward computations that are analogous to those in the proof of Proposition~\ref{subsetoverlapinterior}.
\end{proof}

\begin{prop}
\label{subsetoverlapboundary}
Let $\gm, \dl\subset G$.  Assume term orders of the S-polynomial input are as written. The following equalities hold in $\cE_{k+1}$.
\begin{align*}
\tag{1}\label{subsetboundaryout}
S(\zsub{\gm}{\tau}\ksup{k+1}&-\zsub{\beta}{\gm}\ksup{k+1}, \,
\zsub{\dl}{\tau}\ksup{k+1}-\zsub{\beta}{\dl}\ksup{k+1})\\
&=-\zsub{\beta}{\gm\cap\dl}\ksup{k+1}
\left(\zsub{\dl\sm(\gm\cap\dl)}{\tau}\ksup{k+1} \zsub{\beta}{\gm\sm(\gm\cap\dl)}\ksup{k+1}
-
\zsub{\gm\sm(\gm\cap\dl)}{\tau}\ksup{k+1} \zsub{\beta}{\dl\sm(\gm\cap\dl)}\ksup{k+1}\right)\text{;}\\
\tag{2}\label{subsetboundaryin}S(\zsub{\beta}{\gm}\ksup{k+1}&-\zsub{\gm}{\tau}\ksup{k+1},\,\zsub{\beta}{\dl}\ksup{k+1}-\zsub{\dl}{\tau}\ksup{k+1})\\
&=-\zsub{\gm\cap\dl}{\tau}\ksup{k+1}\left(\zsub{\beta}{\dl\sm(\gm\cap\dl)}\ksup{k+1}\zsub{\gm\sm(\gm\cap\dl)}{\tau}\ksup{k+1} 
-
\zsub{\beta}{\gm\sm(\gm\cap\dl)}\ksup{k+1}\zsub{\dl\sm(\gm\cap\dl)}{\tau}\ksup{k+1} 
\right)\text{; and}\\
\tag{3}\label{subsetboundarymix}S(\zsub{\beta}{\gm}\ksup{k+1}&-\zsub{\gm}{\tau}\ksup{k+1},\,\zsub{\dl}{\tau}\ksup{k+1}-\zsub{\beta}{\dl}\ksup{k+1})
=\zsub{\gm\cup\dl}{\tau}\ksup{k+1}\zsub{\gm\cap\dl}{\tau}\ksup{k+1}-\zsub{\beta}{\gm\cup\dl}\ksup{k+1}\zsub{\beta}{\gm\cap\dl}\ksup{k+1}.
\end{align*}
The term orders of the results in the first two statements are undetermined in general.
\end{prop}

\begin{proof} The argument is a straightforward calculation in each case. Keep in mind throughout that all monomials involved in this computation are products of $\ztau\ksup{i}$ and $\zbeta\ksup{i}$ for $i>k+1$. \smallskip\\
\noindent\emph{Case 1:}
\begin{align*}
S(\zsub{\gm}{\tau}\ksup{k+1}-\zsub{\beta}{\gm}\ksup{k+1},\zsub{\dl}{\tau}\ksup{k+1}-\zsub{\beta}{\dl}\ksup{k+1})=&\frac{\zsub{\gm}{\tau}\ksup{k+1} \zsub{\dl}{\tau}\ksup{k+1}}{\zsub{\gm\cap\dl}{\tau}\ksup{k+1}}\frac{1}{\zsub{\gm}{\tau}\ksup{k+1}}\left(-\zsub{\beta}{\gm}\ksup{k+1}\right)\\
&\,-\frac{\zsub{\gm}{\tau}\ksup{k+1} \zsub{\dl}{\tau}\ksup{k+1}}{\zsub{\gm\cap\dl}{\tau}\ksup{k+1}}\frac{1}{\zsub{\dl}{\tau}\ksup{k+1}}\left(-\zsub{\beta}{\dl}\ksup{k+1}\right)\\
=&-\zsub{\dl\sm(\gm\cap\dl)}{\tau}\ksup{k+1} \zsub{\beta}{\gm}\ksup{k+1}+\zsub{\gm\sm(\gm\cap\dl)}{\tau}\ksup{k+1} \zsub{\beta}{\dl}\ksup{k+1}\\
=&-\zsub{\beta}{\gm\cap\dl}\ksup{k+1}\left(\zsub{\dl\sm(\gm\cap\dl)}{\tau}\ksup{k+1} \zsub{\beta}{\gm\sm(\gm\cap\dl)}\ksup{k+1}\right.\\
&\left.\,\phantom{\zsub{\beta}{\gm\cap\dl}\ksup{k+1}\left(\right.}-\zsub{\gm\sm(\gm\cap\dl)}{\tau}\ksup{k+1} \zsub{\beta}{\dl\sm(\gm\cap\dl)}\ksup{k+1}\right)
\end{align*}
\smallskip

\noindent\emph{Case 2:} The calculation here is almost identical to that of Case 1.
\smallskip

\noindent\emph{Case 3:} Since $\zsub{\beta}{\gm}\ksup{k+1}$ is a product of $\zbeta\ksup{i}$ and $\zsub{\dl}{\tau}\ksup{k+1}$ is a product of $\ztau\ksup{i}$, they cannot have any common divisors. The result as stated just follows from rewriting $$\zsub{\gm}{\tau}\ksup{k+1}\zsub{\dl}{\tau}\ksup{k+1}=\zsub{\gm\cup\dl}{\tau}\ksup{k+1}\zsub{\gm\cap\dl}{\tau}\ksup{k+1}
\quad\text{and}\quad
\zsub{\beta}{\gm}\ksup{k+1}\zsub{\beta}{\dl}\ksup{k+1}=\zsub{\beta}{\gm\cup\dl}\ksup{k+1}\zsub{\beta}{\gm\cap\dl}\ksup{k+1}.$$
\end{proof}

Combining the calculations in Propositions~\ref{subsetoverlapinterior},~\ref{subsetoverlapclosure}, and~\ref{subsetoverlapboundary} with the general principle in Proposition~\ref{separateinteriorclosure}, we obtain $S(g_\gm\ksup{k+1},g_\dl\ksup{k+1})$ for various combinations of in-led and out-led subsets.

\begin{prop}
\label{ksubsetoverlap}
Let $\gm, \dl\subset G$. The following statements hold in $\cE_{k+1}$.
\begin{align*}
\tag{1}\label{ksubsetoverlapout}
\text{If $\gm$ and $\dl$ are both out-led}& \text{, then}\\
S(g_\gm\ksup{k+1},g_\dl\ksup{k+1})=&\,\xsub{G\sm(\gm\cup\dl)}{\gm\cap\dl}\zsub{G\sm(\gm\cup\dl)}{\gm\cap\dl}\ksup{k+1}\zsub{\beta}{\gm\cap\dl}\ksup{k+1}\\
&\cdot\left(g_{\dl\sm(\gm\cap\dl)}^{(k+1), \out}g_{\gm\sm(\gm\cap\dl)}^{(k+1), \into}-g_{\gm\sm(\gm\cap\dl)}^{(k+1), \out}g_{\dl\sm(\gm\cap\dl)}^{(k+1), \into}\right)\\
&\xrightarrow{g_{\gm\sm(\gm\cap\dl)}\ksup{k+1},g_{\dl\sm(\gm\cap\dl)}\ksup{k+1}}0
\end{align*}
\begin{align*}
\tag{2}\label{ksubsetoverlapin}
\text{If $\gm$ and $\dl$ are both in-led, then}\\
S(g\ksup{k+1}_\gm,g\ksup{k+1}_\dl)=&\,\xsub{\gm\cap\dl}{G\sm(\gm\cup\dl)}\zsub{\gm\cap\dl}{G\sm(\gm\cup\dl)}\ksup{k+1}\zsub{\gm\cap\dl}{\tau}\ksup{k+1}\\
&\cdot\left(g_{\dl\sm(\gm\cap\dl)}^{(k+1),\into}g_{\gm\sm(\gm\cap\dl)}^{(k+1),\out}-g_{\gm\sm(\gm\cap\dl)}^{(k+1),\into}g_{\dl\sm(\gm\cap\dl)}^{(k+1),\out}\right)\\
&\xrightarrow{g_{\gm\sm(\gm\cap\dl)}\ksup{k+1},g_{\dl\sm(\gm\cap\dl)}\ksup{k+1}}0
\end{align*}
\begin{align*}
\tag{3}\label{ksubsetoverlapmix}
\text{If $\gm$ is in-led and $\dl$ is out-led, }& \text{then}\\
S(g_\gm\ksup{k+1},g_\dl\ksup{k+1})=&\,\xsub{\gm\sm(\gm\cap\dl)}{\dl\sm(\gm\cap\dl)}\zsub{\gm\sm(\gm\cap\dl)}{\dl\sm(\gm\cap\dl)}\ksup{k+1}\\
&\cdot\left(g_{\gm\cup\dl}^{(k+1),\out} g_{\gm\cap\dl}^{(k+1),\out}-g_{\gm\cup\dl}^{(k+1),\into} g_{\gm\cap\dl}^{(k+1),\into}\right)\\
&\xrightarrow{g_{\gm\cup\dl}\ksup{k+1},g_{\gm\cap\dl}\ksup{k+1}}0.
\end{align*}
\end{prop}
\begin{proof}
All three cases follow directly from applying the appropriate cases of Propositions~\ref{subsetoverlapinterior},~\ref{subsetoverlapclosure}, and~\ref{subsetoverlapboundary}. The reduction statements follow from Proposition~\ref{unorderedreduction}.
\end{proof}

Finally, we compute the S-polynomials among generators of the form $\nu g_\gm\ksup{k}$ from the original basis $\cG_0$ and show that they can all be reduced by generators in the working basis $\cG^\prime$. This will be the first argument in which the choice of monomial order comes into play. Observation~\ref{MOanddiv} will be used repeatedly.

\begin{lemma}
\label{subsetoverlap}
Let $\gm, \dl\subset G$. Then 
$$S(\nu g_\gm\ksup{k}, \nu g_\dl\ksup{k})\xrightarrow{\cG^\prime}0$$ in $\cE_{k+1}[\nu]$.
\end{lemma}

\begin{proof} For any in-led / out-led combination of $\gm$ and $\dl$, the first step to compute $S(\nu g_\gm\ksup{k}, \nu g_\dl\ksup{k})$ is to rewrite $g_\gm\ksup{k}=\anyzeta_\gm g_\gm\ksup{k+1}$ and $g_\dl=\anyzeta_\dl g_\dl\ksup{k+1}$ so that we may apply the results of Proposition~\ref{ksubsetoverlap}. The extra factors of $\anyzeta_\gm$ and $\anyzeta_\dl $ are either $\ztau\ksup{k+1}$ or 1, depending on whether $\ztau\ksup{k+1}$ is internal to $\gm$ and/or $\dl$. 

Consider  $S(\nu g_\gm\ksup{k}, \nu g_\dl\ksup{k})=\nu S(\anyzeta_\gm g_\gm\ksup{k+1}, \anyzeta_\dl g_\dl\ksup{k+1})$, where we have used Proposition~\ref{gbprinciplecoeff} to move the factor of $\nu$. The possible values of $\anyzeta_\gm$ and $\anyzeta_\dl$, along with the rules in Proposition~\ref{gbprinciplecoeff} give us the following cases.
\begin{enumerate}
\item $\anyzeta_\gm=\anyzeta_\dl=\ztau\ksup{k+1}$; that is, $\ztau\ksup{k+1}$ is internal to both $\gm$ and $\dl$. Then Statement (1) of Proposition~\ref{gbprinciplecoeff} implies $$\nu S(\anyzeta_\gm g_\gm\ksup{k+1}, \anyzeta_\dl g_\dl\ksup{k+1})=\nu\ztau\ksup{k+1} S(g_\gm\ksup{k+1}, g_\dl\ksup{k+1}).$$
\item $\anyzeta_\gm=\ztau\ksup{k+1}$, $\anyzeta_\dl=1$; that is, $\ztau\ksup{k+1}$ is internal to exactly one of $\gm$ and $\dl$, which we take to be $\gm$ without loss of generality. In this case, $\ztau\ksup{k+1}$ divides at most one term of $g_\dl\ksup{k}$.
\begin{enumerate}
\item $\ztau\ksup{k+1}$ divides neither term of $g_\dl\ksup{k}$. Then Statement (2) of Proposition~\ref{gbprinciplecoeff} implies $$\nu S(\anyzeta_\gm g_\gm\ksup{k+1}, \anyzeta_\dl g_\dl\ksup{k+1})=\nu \ztau\ksup{k+1}S(g_\gm\ksup{k+1}, g_\dl\ksup{k+1}).$$
\item $\ztau\ksup{k+1}$ divides exactly one term of $g_\dl\ksup{k}$. Observation~\ref{MOanddiv} implies that it divides the leading term, so Statement (3) of Proposition~\ref{gbprinciplecoeff} implies $$\nu S(\anyzeta_\gm g_\gm\ksup{k+1}, \anyzeta_\dl g_\dl\ksup{k+1})=\nu S(g_\gm\ksup{k+1}, g_\dl\ksup{k+1}).$$
\end{enumerate}
\item $\anyzeta_\gm=\anyzeta_\dl=1$; that is, $\ztau\ksup{k+1}$ is internal to neither $\gm$ nor $\dl$. Then $$\nu S(\anyzeta_\gm g_\gm\ksup{k+1}, \anyzeta_\dl g_\dl\ksup{k+1})=\nu S(g_\gm\ksup{k+1}, g_\dl\ksup{k+1}).$$
\end{enumerate}

Explicit expressions for $S(\nu g_\gm\ksup{k}, \nu g_\dl\ksup{k})$ in $\cE_{k+1}[\nu]$ are then multiples (by $\nu$ and possibly by $\ztau\ksup{k+1}$) of the expressions for $S(g_\gm\ksup{k+1}, g_\dl\ksup{k+1})$ in Proposition~\ref{ksubsetoverlap}. In Cases 1 and 2(a), where there is a factor of $\ztau\ksup{k+1}$ in front of $S(g_\gm\ksup{k+1}, g_\dl\ksup{k+1})$, we may reduce by some combination of $\ztau\ksup{k+1}g_\Lambda\ksup{k+1}$ for $\Lambda\in\{\gm\sm(\gm\cap\dl), \dl\sm(\gm\cap\dl), \gm\cap\dl, \gm\cup\dl\}$, exactly paralleling the reductions in Proposition~\ref{ksubsetoverlap}. Generators of the form $\ztau\ksup{k+1} g_\Lambda\ksup{k+1}$ are in the working basis for any $\Lambda$.

Cases 2(b) and 3 are more delicate. Since we do not have the factor of $\ztau\ksup{k+1}$ in front of $\nu S(g_\gm\ksup{k+1}, g_\dl\ksup{k+1})$, we cannot necessarily reduce by generators of the form $\ztau\ksup{k+1}g_\Lambda\ksup{k+1}$. We may always reduce by generators of the form $\nu g_\Lambda\ksup{k+1}$ as in Proposition~\ref{ksubsetoverlap}, but these are in the working basis $\cG^\prime$ only if $g_\Lambda\ksup{k}=g_\Lambda\ksup{k+1}$ (i.e.~if $\anyzeta_\Lambda=1$), which is not always true. We suppose now that $g_\Lambda\ksup{k}\neq g_\Lambda\ksup{k+1}$ for at least one of $\Lambda\in\{\gm\sm(\gm\cap\dl), \dl\sm(\gm\cap\dl), \gm\cap\dl, \gm\cup\dl\}$.
\smallskip

\noindent\emph{Case 2(b):} $\anyzeta_\gm=\ztau\ksup{k+1}$, $\anyzeta_\dl=1$, $\ztau\ksup{k+1}\,\vert\,\lt{g_\dl\ksup{k}}$, and $g_\Lambda\ksup{k}\neq g_\Lambda\ksup{k+1}$ for at least one of $\Lambda\in\{\gm\sm(\gm\cap\dl), \dl\sm(\gm\cap\dl), \gm\cap\dl, \gm\cup\dl\}$. \smallskip 

We have assumed that $\ztau\ksup{k+1}$ is internal to $\gm$ but not $\dl$, which means that it cannot be internal to $\gm\cap\dl$ or $\dl\sm(\gm\cap\dl)$.  We have also assumed that $\ztau\ksup{k+1}$ divides one term of $g_\dl\ksup{k}$, which means that $\ztau\ksup{k+1}$ must go either into or out of $\dl$. Therefore, $\ztau\ksup{k+1}$ cannot be internal to $\gm\sm(\gm\cap\dl)$ either. So $g_\Lambda\ksup{k}=g_\Lambda\ksup{k+1}$ for $\Lambda\in\left\{\gm\cap\dl,\dl\sm(\gm\cap\dl), \gm\sm(\gm\cap\dl)\right\}$. 

The only scenario compatible with our assumptions, then, is that $\ztau\ksup{k+1}$ is internal to $\gm\cup\dl$ and so $g_{\gm\cup\dl}\ksup{k}\neq g_{\gm\cup\dl}\ksup{k+1}$. We then must find an alternative method of reducing $\nu S(g_\gm\ksup{k+1}, g_\dl\ksup{k+1})$ when one subset is in-led and the other out-led (Case 3 of Proposition~\ref{ksubsetoverlap}). 

Suppose first that $\gm$ is in-led and $\dl$ is out-led. Then $\ztau\ksup{k+1}$ must go out of $\gm\cap\dl$ and into $\gm\sm(\gm\cap\dl)$. Proposition~\ref{ksubsetoverlap}(3) gives the following expression for the S-polynomial we are trying to reduce.
\begin{align*}
\nu S(g_\gm\ksup{k+1},g_\dl\ksup{k+1})&=\,\nu \xsub{\gm\sm(\gm\cap\dl)}{\dl\sm(\gm\cap\dl)}\zsub{\gm\sm(\gm\cap\dl)}{\dl\sm(\gm\cap\dl)}\ksup{k+1}\\
&\cdot\left(g_{\gm\cup\dl}^{(k+1),\out} g_{\gm\cap\dl}^{(k+1),\out}-g_{\gm\cup\dl}^{(k+1),\into} g_{\gm\cap\dl}^{(k+1),\into}\right)
\end{align*}
\noi The term order shown is correct; it is determined by the fact that $\ztau\ksup{k+1}$ divides only one of the terms. Since $\ztau\ksup{k+1}$ is internal to $\gm\cup\dl$, it divides neither term of $g_{\gm\cup\dl}\ksup{k+1}$. Since it goes out of but not into $\gm\cap\dl$, it divides $g_{\gm\cap\dl}^{(k+1),\out}$ but not $g_{\gm\cap\dl}^{(k+1),\into}$. Since both $\nu$ and $\ztau\ksup{k+1}$ divide the leading term of this expression, we reduce first by $\nutopk$.
\begin{align*}
\nu S(g_\gm\ksup{k+1},g_\dl\ksup{k+1})\xrightarrow{\nutopk}\, &\xsub{\gm\sm(\gm\cap\dl)}{\dl\sm(\gm\cap\dl)}\zsub{\gm\sm(\gm\cap\dl)}{\dl\sm(\gm\cap\dl)}\ksup{k+1}\\
\cdot\left(\right. & \nu g_{\gm\cup\dl}^{(k+1),\into} g_{\gm\cap\dl}^{(k+1),\into} - g_{\gm\cup\dl}^{(k+1),\out} g_{\gm\cap\dl}^{(k+1),\out}\left.\right)
\end{align*}
The leading term in the result is determined by $\nu$. Now since $\ztau\ksup{k+1}$ goes out of but not into $\gm\cap\dl$, we have available in the working basis the $\gm\cap\dl$ tilde generator, which must have the term order
$$\widetilde{g}_{\gm\cap\dl}=\nu g_{\gm\cap\dl}^{(k+1),\into} - g_{\gm\cap\dl}^{(k+1),\out}.$$
This term order is compatible with the term order of our reduced expression for $\nu S(g_\gm\ksup{k+1},g_\dl\ksup{k+1})$, so we reduce further as follows.
\begin{align*}
\nu S(g_\gm\ksup{k+1},g_\dl\ksup{k+1})&\xrightarrow{\nutopk,\,\widetilde{g}_{\gm\cap\dl}}\\
&\quad\xsub{\gm\sm(\gm\cap\dl)}{\dl\sm(\gm\cap\dl)}\zsub{\gm\sm(\gm\cap\dl)}{\dl\sm(\gm\cap\dl)}\ksup{k+1}g_{\gm\cap\dl}^{(k+1),\out}g_{\gm\cup\dl}\ksup{k+1}.
\end{align*}
Since $g_{\gm\cup\dl}\ksup{k+1}$ is the only factor in this expression with more than one term, its term order determines the term order of the expression. Since $\ztau\ksup{k+1}$ divides $g_{\gm\cap\dl}^{(k+1),\out}$, we may reduce to zero using $\ztau\ksup{k+1}g_{\gm\cup\dl}\ksup{k+1}$, which is in the working basis.

The other possibility in Case 2(b) was that $\gm$ was out-led and $\dl$ in-led. This means that $\ztau\ksup{k+1}$ goes out of $\gm\sm(\gm\cap\dl)$ and into $\gm\cap\dl$. Our expression for $\nu S(g_\gm\ksup{k+1},g_\dl\ksup{k+1})$ comes from Part (3) of Proposition~\ref{ksubsetoverlap} again, but with the roles of $\gm$ and $\dl$ reversed. The argument for reducing by $\nutopk$, then $\widetilde{g}_{\gm\cap\dl}$, then $\ztau\ksup{k+1}g_{\gm\cup\dl}\ksup{k+1}$ is very similar to the argument just given, so we omit the details here. 

\smallskip

\noindent\emph{Case 3:} $\anyzeta_\gm=\anyzeta_\dl=1$ and $g_\Lambda\ksup{k}\neq g_\Lambda\ksup{k+1}$ for at least one of $\Lambda\in\{\gm\sm(\gm\cap\dl), \dl\sm(\gm\cap\dl), \gm\cap\dl, \gm\cup\dl\}$. \smallskip 

Our assumptions about $\anyzeta_\gm$ and $\anyzeta_\dl$ mean that $\ztau\ksup{k+1}$ is not internal to $\gm$ or $\dl$, hence not to $\gm\sm(\gm\cap\dl)$, $\dl\sm(\gm\cap\dl)$, or $\gm\cap\dl$. Therefore, we assume that it is internal to $\gm\cup\dl$, so that $g_{\gm\cup\dl}\ksup{k}\neq g_{\gm\cup\dl}\ksup{k+1}$. Then $\ztau\ksup{k+1}$ must go between $\gm\sm(\gm\cap\dl)$ and $\dl\sm(\gm\cap\dl)$ in one direction or the other. We will assume that it goes out of $\dl\sm(\gm\cap\dl)$ and into $\gm\sm(\gm\cap\dl)$. The other case is analogous, with the roles of $\gm$ and $\dl$ reversed throughout the argument.

We know, then, that $\dl$ is out-led, $\gm$ is in-led, and $\ztau\ksup{k+1}$ divides exactly one term of $g_\gm\ksup{k+1}$ and exactly one term of $g_\dl\ksup{k+1}$. Therefore, we have available in the working basis
$$\widetilde{g}_\gm = \nu g_\gm^{(k+1),\out} - g_\gm^{(k+1),\into} \quad\text{and}\quad\widetilde{g}_\dl = \nu g_\dl^{(k+1),\into} - g_\dl^{(k+1),\out}.$$ 
We will use these to reduce $\nu S(g_\gm\ksup{k+1},g_\dl\ksup{k+1})$ to zero. Begin with a refactored version of Statement (3) in Proposition~\ref{ksubsetoverlap} (e.g.~refer to the first line of Case \ref{subsetinteriormix} in Proposition~\ref{subsetoverlapinterior}). 
\begin{align*}
\nu S(g_\gm\ksup{k+1},\, &g_\dl\ksup{k+1})=\\
%
&\nu \xsub{\dl\sm(\gm\cap\dl)}{G\sm(\gm\cup\dl)}\xsub{\gm\cap\dl}{G\sm\dl}\zsub{\dl\sm(\gm\cap\dl)}{G\sm(\gm\cup\dl)}\zsub{\gm\cap\dl}{G\sm\dl}\zsub{\dl}{\tau}\ksup{k+1}
  g_\gm^{(k+1),\out}\\
\,- &\nu\xsub{G\sm(\gm\cup\dl)}{\gm\sm(\gm\cap\dl)}\xsub{G\sm\gm}{\gm\cap\dl}\zsub{G\sm(\gm\cup\dl)}{\gm\sm(\gm\cap\dl)}\zsub{G\sm\gm}{\gm\cap\dl}\zsub{\beta}{\gm}\ksup{k+1}
 g_{\dl}^{(k+1),\into}
\end{align*}
The term order of this expression is unknown since $\nu$ divides both terms and $\ztau\ksup{k+1}$ divides neither. It is reducible by $\widetilde{g}_\gm$ if the term order shown is correct and by $\widetilde{g}_\dl$ if not. The argument is similar either way, so we suppose now that the term order shown is correct and omit the other case.  Reducing by $\widetilde{g}_\gm$ produces the following, in which the term order is determined by $\nu$ and shown correctly.
\begin{align*}
\nu S(g_\gm\ksup{k+1},\, &g_\dl\ksup{k+1})\xrightarrow{\widetilde{g}_\gm}\\
&\nu\xsub{G\sm(\gm\cup\dl)}{\gm\sm(\gm\cap\dl)}\xsub{G\sm\gm}{\gm\cap\dl}\zsub{G\sm(\gm\cup\dl)}{\gm\sm(\gm\cap\dl)}\zsub{G\sm\gm}{\gm\cap\dl}\zsub{\beta}{\gm}\ksup{k+1}
 g_{\dl}^{(k+1),\into}\\
%
\,- & \xsub{\dl\sm(\gm\cap\dl)}{G\sm(\gm\cup\dl)}\xsub{\gm\cap\dl}{G\sm\dl}\zsub{\dl\sm(\gm\cap\dl)}{G\sm(\gm\cup\dl)}\zsub{\gm\cap\dl}{G\sm\dl}\zsub{\dl}{\tau}\ksup{k+1}
  g_\gm^{(k+1),\into}
\end{align*}
This expression can be reduced by $\widetilde{g}_\dl$, leaving the following, in which the term order is unknown.
\begin{align*}
\nu S(g_\gm\ksup{k+1},\, &g_\dl\ksup{k+1})\xrightarrow{\widetilde{g}_\gm,\widetilde{g}_\dl}\\
& \xsub{\dl\sm(\gm\cap\dl)}{G\sm(\gm\cup\dl)}\xsub{\gm\cap\dl}{G\sm\dl}\zsub{\dl\sm(\gm\cap\dl)}{G\sm(\gm\cup\dl)}\zsub{\gm\cap\dl}{G\sm\dl}\zsub{\dl}{\tau}\ksup{k+1}
  g_\gm^{(k+1),\into}\\
\,- &\xsub{G\sm(\gm\cup\dl)}{\gm\sm(\gm\cap\dl)}\xsub{G\sm\gm}{\gm\cap\dl}\zsub{G\sm(\gm\cup\dl)}{\gm\sm(\gm\cap\dl)}\zsub{G\sm\gm}{\gm\cap\dl}\zsub{\beta}{\gm}\ksup{k+1}
 g_{\dl}^{(k+1),\out}
\end{align*}
This expression is actually zero, which we may see by refactoring the term written first above.
\begin{align*}
&\xsub{\dl\sm(\gm\cap\dl)}{G\sm(\gm\cup\dl)}\xsub{\gm\cap\dl}{G\sm\dl}\zsub{\dl\sm(\gm\cap\dl)}{G\sm(\gm\cup\dl)}\zsub{\gm\cap\dl}{G\sm\dl}\zsub{\dl}{\tau}\ksup{k+1}  g_\gm^{(k+1),\into}\\
=&\xsub{\dl\sm(\gm\cap\dl)}{G\sm(\gm\cup\dl)}\xsub{\gm\cap\dl}{G\sm\dl}\xsub{G\sm\gm}{\gm}\\
\cdot &\zsub{\dl\sm(\gm\cap\dl)}{G\sm(\gm\cup\dl)}\zsub{\gm\cap\dl}{G\sm\dl}\zsub{G\sm\gm}{\gm}\zsub{\dl}{\tau}\ksup{k+1}\zsub{\beta}{\gm}\ksup{k+1}\\
=&\xsub{\dl\sm(\gm\cap\dl)}{G\sm(\gm\cup\dl)}\xsub{\gm\cap\dl}{G\sm\dl}\xsub{G\sm(\gm\cup\dl)}{\gm}\xsub{\dl\sm(\gm\cap\dl)}{\gm\sm(\gm\cap\dl)}\xsub{\dl\sm(\gm\cap\dl)}{\gm\cap\dl}\\
\cdot &\zsub{\dl\sm(\gm\cap\dl)}{G\sm(\gm\cup\dl)}\zsub{\gm\cap\dl}{G\sm\dl}\zsub{G\sm(\gm\cup\dl)}{\gm}\zsub{\dl\sm(\gm\cap\dl)}{\gm\sm(\gm\cap\dl)}\zsub{\dl\sm(\gm\cap\dl)}{\gm\cap\dl}\\
\cdot &\zsub{\dl}{\tau}\ksup{k+1}\zsub{\beta}{\gm}\ksup{k+1}\\
=&\xsub{\dl}{G\sm\dl)}\xsub{G\sm(\gm\cup\dl)}{\gm}\xsub{\dl\sm(\gm\cap\dl)}{\gm\cap\dl}\\
\cdot &\zsub{\dl}{G\sm\dl}\zsub{G\sm(\gm\cup\dl)}{\gm}\zsub{\dl\sm(\gm\cap\dl)}{\gm\cap\dl}
\cdot \zsub{\dl}{\tau}\ksup{k+1}\zsub{\beta}{\gm}\ksup{k+1}\\
=&g_\dl^{(k+1),\out}\xsub{G\sm(\gm\cup\dl)}{\gm}\xsub{\dl\sm(\gm\cap\dl)}{\gm\cap\dl}
\cdot \zsub{G\sm(\gm\cup\dl)}{\gm}\zsub{\dl\sm(\gm\cap\dl)}{\gm\cap\dl}
\cdot \zsub{\beta}{\gm}\ksup{k+1}
\end{align*}
One more refactoring shows that this last expression is the same as the second term in our reduced expression for $\nu S(g_\gm\ksup{k+1},\, g_\dl\ksup{k+1})$ above. So we have shown that 
$$\nu S(g_\gm\ksup{k+1},\, g_\dl\ksup{k+1})\xrightarrow{\widetilde{g}_\gm,\widetilde{g}_\dl}0.$$
\end{proof}

We have now computed S-polynomials among generators of the form $\nu g_\gm\ksup{k}$ for all combinations of in-led and out-led subsets, and seen that they all reduce to zero by elements of the working basis $\cG^\prime$. Therefore, we close this section with the same working basis as at the end of the previous section.

\section{Implementing Buchberger's Algorithm: Round 2}
\label{bbrd2}
Buchberger's Algorithm now calls for a new round of S-polynomials: those involving elements of $\cG^\prime\setminus\cG_0$. As we will see, these all reduce to zero within $\cG^\prime$. So, at the close of this section, we will have confirmed that the S-polynomial of any pair of generators in $\cG^\prime$ reduces to zero within $\cG^\prime$. This will complete the proof of Lemma~\ref{buchend}. Table~\ref{round2table} records the computations undertaken in this section.

\begin{table}[htdp]
\begin{center}
\renewcommand\arraystretch{1.5}
\begin{tabular}{|c|c|c|c|}
\hline
$S\!\left(-,-\right)$ & Result & Prop. & Add to $\cG^\prime$? \\\hline\hline
$S\!\left(\nutopk,\ztau\ksup{k+1} g_\gm\ksup{k+1}\right)$ & 0 & Stmt.~(\ref{nutop0nu}) of Prop.~\ref{nutopone} & no \\\hline
$S\!\left(\ztau\ksup{k+1} g_\gm\ksup{k+1}, \ztau\ksup{k+1} g_\dl\ksup{k+1}\right)$ & 0 & Prop.~\ref{gbprinciplecoeff} \& Prop.~\ref{ksubsetoverlap} & no \\\hline
$S\!\left(\ztau\ksup{k+1} g_\gm\ksup{k+1}, \widetilde{g}_\dl\right)$ & 0 & Lemma~\ref{ztautilde} & no \\\hline
$S\!\left( \widetilde{g}_\gm, \widetilde{g}_\dl\right)$ & 0 & Lemma~\ref{doubletilde} & no \\\hline
$S\!\left(\nu g_\gm\ksup{k},\ztau\ksup{k+1} g_\dl\ksup{k+1}\right)$ & 0 & Lemma~\ref{2nunonu} & no \\\hline
$S\!\left(\nu g_\gm\ksup{k}, \widetilde{g}_\dl\right)$ & 0 & Lemma~\ref{tildeprop} & no \\\hline
\end{tabular}
\end{center}
\medskip
\caption{S-polynomials Round 2: All computations are assumed to be among generators in $\cG^\prime$ and are carried out in $\cE_{k+1}[\nu]$. S-polynomials are listed in the order they are computed in Section~\ref{bbrd2}.}
\label{round2table}
\end{table}

First, we may quickly take care of $S(\nutopk,\ztau\ksup{k+1}g_\gm\ksup{k+1})$ using Statement~(\ref{nutop0nu}) of Proposition~\ref{nutopone}. Assuming that $\lc{g_\gm\ksup{k+1}}=1,$
\begin{align*}
S(\nutopk,\ztau\ksup{k+1}g_\gm\ksup{k+1})&=\,-\nu\ztau\ksup{k+1}\trt{g_\gm\ksup{k+1}}-\ztau\ksup{k+1}\lt{g_\gm\ksup{k+1}}\\
\text{reduce}\quad&+\,\,\trt{g_\gm\ksup{k+1}}\left(\nutopk\right)\\
&=-\ztau\ksup{k+1}g_\gm\ksup{k+1}
\end{align*}
Therefore, $S(\nutopk,\ztau\ksup{k+1}g_\gm\ksup{k+1})\xrightarrow{\nutopk, \ztau\ksup{k+1}g_\gm\ksup{k+1}}0.$

Second, we apply Proposition~\ref{gbprinciplecoeff} and Proposition~\ref{ksubsetoverlap} to see that 
\[S(\ztau\ksup{k+1}g_\gm\ksup{k+1},\ztau\ksup{k+1}g_\dl\ksup{k+1})=\ztau\ksup{k+1}S(g_\gm\ksup{k+1},g_\dl\ksup{k+1}),\]
which reduces to zero either by the pair $\ztau\ksup{k+1}g_{\gm\sm(\gm\cap\dl)}\ksup{k+1}$ and $\ztau\ksup{k+1}g_{\dl\sm(\gm\cap\dl)}\ksup{k+1}$,
or by the pair
$\ztau\ksup{k+1}g_{\gm\cup\dl}\ksup{k+1}$ and $\ztau\ksup{k+1}g_{\gm\cap\dl}\ksup{k+1}$.
In either case, these are elements of the working basis $\cG^\prime$.

A similar method works for $S(\ztau\ksup{k+1} g_\gm\ksup{k+1},\widetilde{g}_\dl)$, but in this case the term order of the S-polynomial's result will be determined by the presence of $\nu$ on one term but not the other. This means that we cannot use Proposition~\ref{unorderedreduction} to reduce these expressions to zero (as we did in Proposition~\ref{ksubsetoverlap}).

 \begin{lemma}
 \label{ztautilde}
 Let $\gm, \dl\subset G$. Suppose that $\ztau\ksup{k+1}$ divides the leading term but not the trailing term of $g_\dl\ksup{k}$. Then 
 \begin{align*}
 S(\ztau\ksup{k+1} g_\gm\ksup{k+1},\widetilde{g}_\dl)&\xrightarrow{\nutopk,\ztau\ksup{k+1}g_{\gm\sm(\gm\cap\dl)}\ksup{k+1},\ztau\ksup{k+1}g_{\dl\sm(\gm\cap\dl)}\ksup{k+1}}0\\
\quad\text{or}\quad
&\xrightarrow{\nutopk, \ztau\ksup{k+1}g_{\gm\cup\dl}\ksup{k+1},\ztau\ksup{k+1}g_{\gm\cap\dl}\ksup{k+1}}0.
\end{align*}
\end{lemma}
\begin{proof}
Applying Statement (2) of Proposition~\ref{gbprinciplecoeff}, we have 
\begin{align*}
S(\ztau\ksup{k+1} g_\gm\ksup{k+1},\widetilde{g}_\dl)&=\ztau\ksup{k+1} S(g_\gm\ksup{k+1},\widetilde{g}_\dl)\\
&=\ztau\ksup{k+1}S(g_\gm\ksup{k+1},\nu\trt{g_\dl\ksup{k+1}}-\lt{g_\dl\ksup{k+1}})
\end{align*}
Propositions~\ref{subsetoverlapinterior},~\ref{subsetoverlapclosure}, and~\ref{subsetoverlapboundary} can then be used to compute the S-polynomial explicitly in terms of $g_\Lambda\ksup{k+1}$ for $\Lambda\in\left\{\gm\sm(\gm\cap\dl), \dl\sm(\gm\cap\dl), \gm\cup\dl, \gm\cap\dl \right\}$, just as in Proposition~\ref{ksubsetoverlap}. The leading term of the result will be divisible by both $\nu$ and $\ztau\ksup{k+1}$, so it will be reducible by $\nutopk$ (and term order will be determined by $\nu$). Let $d=\gcd\left(\lm{g_\gm\ksup{k+1}},\trm{g_\dl\ksup{k+1}}\right)$. Then assuming that all coefficients are $\pm 1$,
\begin{align*}
\ztau\ksup{k+1}S(g_\gm\ksup{k+1}&, \widetilde{g}_\dl)=\\
&\ztau\ksup{k+1}\!\left(\frac{\nu\trt{g_\dl\ksup{k+1}}}{d}\trt{g_\gm\ksup{k+1}}-
\frac{\lt{g_\gm\ksup{k+1}}}{d}\lt{g_\dl\ksup{k+1}}
\right)\\
\xrightarrow{\nutopk}&\ztau\ksup{k+1}\!\left(\frac{\lt{g_\gm\ksup{k+1}}}{d}\lt{g_\dl\ksup{k+1}}
-\frac{\trt{g_\dl\ksup{k+1}}}{d}\trt{g_\gm\ksup{k+1}}\right). 
\end{align*}
The expression inside the parentheses may be expanded in terms of $g_\Lambda\ksup{k+1}$ for $\Lambda\in\{\gm\sm(\gm\cap\dl),\dl\sm(\gm\cap\dl),\gm\cap\dl,\gm\cup\dl\}$ just as in Proposition~\ref{ksubsetoverlap}. The term order of the resulting expression will be unknown, but since we have $\ztau\ksup{k+1}$ in front of it, we are effectively in the situation of Cases 1 and 2(a) of Lemma~\ref{subsetoverlap}. Proposition~\ref{unorderedreduction} allows us to reduce by some combination of $\ztau\ksup{k+1} g_\Lambda\ksup{k+1}$. These generators are in the working basis $\cG^\prime$ for any $\Lambda$.
\end{proof}

Next, we consider S-polynomials between tilde generators. Note that we only consider tilde generators that occur in $\cG^\prime$, which justifies the assumptions about divisibility by $\ztau\ksup{k+1}$ in the statement below.

\begin{lemma}
\label{doubletilde}
Let $\gm, \dl\subset G$. Assume that $\ztau\ksup{k+1}$ divides $\lt{g_\gm\ksup{k}}$ and $\lt{g_\dl\ksup{k}}$ and $\ztau\ksup{k+1}$ does not divide $\trt{g_\gm\ksup{k}}$ or $\trt{g_\dl\ksup{k}}$. Then the reduction
$$S(\widetilde{g}_\gm,\widetilde{g}_\dl)\xrightarrow{\left\{\ztau\ksup{k+1}g_\Lambda\ksup{k+1}\right\}}0$$
holds in $\cE_{k+1}[\nu]$ for some combination of $\Lambda\in\{\gm\sm(\gm\cap\dl),\dl\sm(\gm\cap\dl),\gm\cap\dl,\gm\cup\dl\}$. All such generators are in the working basis $\cG^\prime$.
\end{lemma}

\begin{proof}
Since we have assumed that $\ztau\ksup{k+1}$ divides the leading terms of $g_\gm\ksup{k}$ and $g_\dl\ksup{k}$ but does not divide the trailing terms, we know that $\ztau\ksup{k+1}$ is incident but not internal to both $\gm$ and $\dl$. Therefore $g_\gm\ksup{k}=g_\gm\ksup{k+1}$ and $g_\dl\ksup{k}=g_\dl\ksup{k+1}$. Let $d=\gcd\left(\trm{g_\gm\ksup{k+1}},\trm{g_\dl\ksup{k+1}}\right).$ Assuming that all coefficients are $\pm 1$, 
\begin{align}
\nonumber S(\widetilde{g}_\gm,\widetilde{g}_\dl)&=S(\nu \trt{g_\gm\ksup{k+1}} - \lt{g_\gm\ksup{k+1}},\nu \trt{g_\gm\ksup{k+1}} - \lt{g_\gm\ksup{k+1}})\\
\label{expandmedoubletilde}&=\frac{\trt{g_\dl\ksup{k+1}}}{d}\lt{g_\gm\ksup{k+1}} - \frac{\trt{g_\gm\ksup{k+1}}}{d}\lt{g_\dl\ksup{k+1}}.
\end{align}
The term order here is unknown, since $\nu$ divides neither term and $\ztau\ksup{k+1}$ divides both (because it divides the leading terms of both $g_\gm\ksup{k+1}$ and $g_\dl\ksup{k+1}$). We may expand this expression in terms of $g_\Lambda\ksup{k+1}$ for some combination of $\Lambda\in\{\gm\sm(\gm\cap\dl),\dl\sm(\gm\cap\dl),\gm\cap\dl,\gm\cup\dl\}$ just as in Proposition~\ref{ksubsetoverlap} or directly, using the computations in Propositions~\ref{subsetoverlapinterior},~\ref{subsetoverlapclosure}, and~\ref{subsetoverlapboundary}. Effectively, we are computing the S-polynomial of $g_\gm\ksup{k+1}$ and $g_\dl\ksup{k+1}$ both with their term orders reversed.

If the result of expanding line~(\ref{expandmedoubletilde}) looks like Case (1) of Proposition~\ref{ksubsetoverlap}, then both trailing terms must be products of outgoing edges, which means $\ztau\ksup{k+1}$ is incoming to both $\gm$ and $\dl$. Since it is internal to neither $\gm$ nor $\dl$, $\ztau\ksup{k+1}$ must go from $G\sm(\gm\cup\dl)$ to $\gm\cap\dl$. Then $\ztau\ksup{k+1}$ divides the factor of $\zsub{G\sm(\gm\cup\dl)}{\gm\cap\dl}\ksup{k+1}$ in front of the expanded expression. Therefore, regardless of term order, we may invoke Proposition~\ref{unorderedreduction} to reduce to zero by $\ztau\ksup{k+1}g_{\gm\sm(\gm\cap\dl)}\ksup{k+1}$ and $\ztau\ksup{k+1}g_{\dl\sm(\gm\cap\dl)}\ksup{k+1}$, which are in the working basis $\cG^\prime$.

If the result of expanding line (\ref{expandmedoubletilde}) looks like Case (2) of Proposition~\ref{ksubsetoverlap}, then both trailing terms are products of incoming edges, which means $\ztau\ksup{k+1}$ goes from $\gm\cap\dl$ to $G\sm(\gm\cup\dl)$. Therefore, it divides the factor of $\zsub{\gm\cap\dl}{G\sm(\gm\cup\dl)}\ksup{k+1}$ in front of the expanded expression. So we may again reduce by $\ztau\ksup{k+1}g_{\gm\sm(\gm\cap\dl)}\ksup{k+1}$ and $\ztau\ksup{k+1}g_{\dl\sm(\gm\cap\dl)}\ksup{k+1}$ regardless of term order.

If the result of expanding line (\ref{expandmedoubletilde}) looks like Case (3) of Proposition~\ref{ksubsetoverlap}, then the trailing term of $g_\gm\ksup{k+1}$ is a product of incoming edges and the trailing term of $g_\dl\ksup{k+1}$ is a product of outgoing edges. Therefore, $\ztau\ksup{k+1}$ goes from $\gm\sm(\gm\cap\dl)$ to $\dl\sm(\gm\cap\dl)$, which means that it divides the factor of $\zsub{\gm\sm(\gm\cap\dl)}{\dl\sm(\gm\cap\dl)}\ksup{k+1}$ in front of the expanded expression. Regardless of term order, we may reduce to zero by a combination of $\ztau\ksup{k+1}g_{\gm\cap\dl}\ksup{k+1}$ and $\ztau\ksup{k+1}g_{\gm\cup\dl}\ksup{k+1}$, both of which are in the working basis $\cG^\prime$.

The only other possibility is that the result of expanding line~(\ref{expandmedoubletilde}) looks like Case (3) of Proposition~\ref{ksubsetoverlap} with the roles of $\gm$ and $\dl$ reversed. In that case, reverse the roles of $\gm$ and $\dl$ in the previous paragraph's argument to see that we may again reduce to zero by $\ztau\ksup{k+1}g_{\gm\cap\dl}\ksup{k+1}$ and $\ztau\ksup{k+1}g_{\gm\cup\dl}\ksup{k+1}$.
\end{proof}

The remaining two propositions carry out the computations necessary to see that $S(\nu g_\gm\ksup{k},\ztau\ksup{k+1} g_\dl\ksup{k+1})\xrightarrow{\cG^\prime}0$ and $S(\nu g_\gm\ksup{k},\widetilde{g}_\dl)\xrightarrow{\cG^\prime}0$. In both cases, the computations themselves follow directly from our earlier results, but some additional work is required to see that reductions are always possible by generators in $\cG^\prime$. This will complete Table~\ref{round2table} and finish the proof of Lemma~\ref{buchend}.

\begin{lemma}
\label{2nunonu}
Let $\gm,\dl\subset G$. Then 
$S(\nu g_\gm\ksup{k},\ztau\ksup{k+1} g_\dl\ksup{k+1})\xrightarrow{\cG^\prime}0$
in $\cE_{k+1}[\nu]$.
\end{lemma}
\begin{proof}
This S-polynomial is closely related to the one considered in Lemma~\ref{subsetoverlap}. By Statement (2) of Proposition~\ref{gbprinciplecoeff}, we have 
$$S(\nu g_\gm\ksup{k},\ztau\ksup{k+1} g_\dl\ksup{k+1})=\nu S(g_\gm\ksup{k},\ztau\ksup{k+1} g_\dl\ksup{k+1}).$$

If $\ztau\ksup{k+1}$ is internal to $\dl$, then $$\nu S(g_\gm\ksup{k},\ztau\ksup{k+1} g_\dl\ksup{k+1})=\nu S(g_\gm\ksup{k}, g_\dl\ksup{k})=S(\nu g_\gm\ksup{k}, \nu g_\dl\ksup{k}).$$ We have already showed in Lemma~\ref{subsetoverlap} that this S-polynomial reduces to zero within $\cG^\prime$.

If $\ztau\ksup{k+1}$ is not internal to $\dl$, then we have the following breakdown of cases.
\begin{enumerate}
\item $\ztau\ksup{k+1}$ is internal to $\gm$. Then 
\begin{align*}
\nu S(g_\gm\ksup{k},\ztau\ksup{k+1} g_\dl\ksup{k+1})&=\nu S(\ztau\ksup{k+1}g_\gm\ksup{k+1},\ztau\ksup{k+1} g_\dl\ksup{k+1})\\
&=\nu\ztau\ksup{k+1} S(g_\gm\ksup{k+1},g_\dl\ksup{k+1}),
\end{align*}
by Statement~(1) of Proposition~\ref{gbprinciplecoeff}. After expanding via Proposition~\ref{ksubsetoverlap}, we may reduce this expression to zero by some combination of $\ztau\ksup{k+1}g_\Lambda\ksup{k+1}$ for $\Lambda\in\{\gm\sm(\gm\cap\dl),\dl\sm(\gm\cap\dl),\gm\cap\dl,\gm\cup\dl\}$.
\item $\ztau\ksup{k+1}$ divides neither term of $g_\gm\ksup{k+1}$. Then $g_\gm\ksup{k}=g_\gm\ksup{k+1}$ and Statement (2) of Proposition~\ref{gbprinciplecoeff} implies that 
$$\nu S(g_\gm\ksup{k},\ztau\ksup{k+1} g_\dl\ksup{k+1}) = \nu\ztau\ksup{k+1} S(g_\gm\ksup{k+1},g_\dl\ksup{k+1}).$$
After expanding via Proposition~\ref{ksubsetoverlap}, we may reduce this expression to zero by some combination of $\ztau\ksup{k+1}g_\Lambda\ksup{k+1}$ for $\Lambda\in\{\gm\sm(\gm\cap\dl),\dl\sm(\gm\cap\dl),\gm\cap\dl,\gm\cup\dl\}$.
\item $\ztau\ksup{k+1}$ divides exactly one term, hence the leading term, of $g_\gm\ksup{k+1}$. Then $g_\gm\ksup{k}=g_\gm\ksup{k+1}$ and Statement~(3) of Proposition~\ref{gbprinciplecoeff} implies that
\begin{align*}
\nu S(g_\gm\ksup{k},\ztau\ksup{k+1} g_\dl\ksup{k+1})&=\nu S(g_\gm\ksup{k+1},\ztau\ksup{k+1} g_\dl\ksup{k+1})\\
&=\nu S(g_\gm\ksup{k+1},g_\dl\ksup{k+1}).
\end{align*}
Also, our assumptions to this point mean that $\ztau\ksup{k+1}$ is internal to neither $\gm$ nor $\dl$. Therefore, we are in the same situation as Case (3) of Lemma~\ref{subsetoverlap}, in which we have already established that we may reduce to zero by generators already contained in $\cG^\prime$.
\end{enumerate}
\end{proof}

\begin{lemma}
\label{tildeprop}
Let $\gm, \dl\subset G$ and assume that $\widetilde{g}_\dl\in\cG^\prime$.  Then $S(\nu g_\gm\ksup{k},\widetilde{g}_\dl)\xrightarrow{\cG^\prime}0$ in $\cE_{k+1}[\nu]$.
\end{lemma}

\begin{proof}
For the most part, the argument here is parallel to the proof of Lemma~\ref{subsetoverlap}. However,  the leading terms in the results of these S-polynomials are determined by the presence of $\nu$ on exactly one of the terms, so it is not clear that the same reductions are always possible. Applying Statement (3) of Proposition~\ref{gbprinciplecoeff}, we have 
$S(\nu g_\gm\ksup{k}, \widetilde{g}_\dl)=S(g_\gm\ksup{k}, \widetilde{g}_\dl).$
Note that our assumption that $\widetilde{g}_\dl\in\cG^\prime$ implies that $\ztau\ksup{k+1}$ is not internal to $\dl$ and that $g_\dl\ksup{k}=g_\dl\ksup{k+1}$. We consider the following cases, which parallel those in Lemma~\ref{subsetoverlap}.

\begin{enumerate}
\item $\ztau\ksup{k+1}$ is internal to $\gm$
\item $\ztau\ksup{k+1}$ divides exactly one term, hence the leading term of $g_\gm\ksup{k+1}$.
\item $\ztau\ksup{k+1}$ divides neither term of $g_\gm\ksup{k+1}$.
\end{enumerate}

\noi\emph{Case 1:} If $\ztau\ksup{k+1}$ is internal to $\gm$, then $g_\gm\ksup{k}=\ztau\ksup{k+1}g_\gm\ksup{k+1}$. We know that $\ztau\ksup{k+1}$ does not divide the trailing term of $g_\dl\ksup{k}$, so it does not divide the leading term of $\widetilde{g}_\dl$. Applying Statement (2) of Proposition~\ref{gbprinciplecoeff}, we have
\begin{align*}
S(g_\gm\ksup{k}, \widetilde{g}_\dl)&=S(\ztau\ksup{k+1}g_\gm\ksup{k+1}, \widetilde{g}_\dl)\\
&=\ztau\ksup{k+1}S(g_\gm\ksup{k+1}, \widetilde{g}_\dl)\\
&=\ztau\ksup{k+1}S(\lt{g_\gm\ksup{k+1}}-\trt{g_\gm\ksup{k+1}}, \nu \trt{g_\dl\ksup{k+1}}-\lt{g_\dl\ksup{k+1}})
\end{align*}
Let $d=\gcd\left(\lm{g_\gm\ksup{k+1}}, \trm{g_\dl\ksup{k+1}}\right)$. Note that $\ztau\ksup{k+1}$ does not divide $d$. Then we expand and reduce the S-polynomial above as follows, assuming that coefficients are all $\pm 1$.
\begin{align*}
&=\ztau\ksup{k+1}\left(\nu\frac{\trt{g_\dl\ksup{k+1}}}{d}\trt{g_\gm\ksup{k+1}} - \frac{\lt{g_\gm\ksup{k+1}}}{d}\lt{g_\dl\ksup{k+1}}\right)\\
&\xrightarrow{\nutopk} \ztau\ksup{k+1}\left(\frac{\trt{g_\dl\ksup{k+1}}}{d}\trt{g_\gm\ksup{k+1}} - \frac{\lt{g_\gm\ksup{k+1}}}{d}\lt{g_\dl\ksup{k+1}}\right)
\end{align*}
The term order in the last expression is unknown. The expression inside the parentheses may be expanded in terms of $g_\Lambda\ksup{k+1}$ using Proposition~\ref{ksubsetoverlap}, for $\Lambda\in\{\gm\sm(\gm\cap\dl),\dl\sm(\gm\cap\dl),\gm\cap\dl,\gm\cup\dl\}$. Since the factor of $\ztau\ksup{k+1}$ is available out front, we may then reduce by $\ztau\ksup{k+1}g_\Lambda\ksup{k+1}$ for the appropriate combination of $\Lambda$. All such generators are contained in $\cG^\prime$.

\medskip

In Cases 2 and 3, $\ztau\ksup{k+1}$ is not internal to $\gm$, and we continue to assume that it is not internal to $\dl$. Then $g_\gm\ksup{k}=g_\gm\ksup{k+1}$, so we expand the S-polynomial as in Case 1, but do not obtain a factor of $\ztau\ksup{k+1}$ in front.
\begin{align}
\nonumber S(g_\gm\ksup{k}, \widetilde{g}_\dl)&=S(g_\gm\ksup{k+1}, \widetilde{g}_\dl)\\
\label{reduceme}&=\nu\frac{\trt{g_\dl\ksup{k+1}}}{d}\trt{g_\gm\ksup{k+1}} - \frac{\lt{g_\gm\ksup{k+1}}}{d}\lt{g_\dl\ksup{k+1}}
\end{align}

\noi\emph{Case 2:}
Since we have assumed that $\ztau\ksup{k+1}$ divides $\lt{g_\gm\ksup{k}}$ but not $\trt{g_\gm\ksup{k}}$, we have $\widetilde{g}_\gm\in\cG^\prime$.  We may use it to reduce line~(\ref{reduceme}) above.
\begin{align*}
S(g_\gm\ksup{k}, \widetilde{g}_\dl)\xrightarrow{\widetilde{g}_\gm}&
\frac{\trt{g_\dl\ksup{k+1}}}{d}\lt{g_\gm\ksup{k+1}}-\frac{\lt{g_\gm\ksup{k+1}}}{d}\lt{g_\dl\ksup{k+1}}\\
=&\frac{\lt{g_\gm\ksup{k+1}}}{d}g_\dl\ksup{k+1}
\end{align*}
Since $\ztau\ksup{k+1}$ is not internal to $\dl$, it does not divide $\trt{g_\dl\ksup{k+1}}$, which means it also does not divide $d$. Our assumption that $\ztau\ksup{k+1}$ divides $\lt{g_\gm\ksup{k+1}}$ then implies that we may reduce further by $\ztau\ksup{k+1}g_\dl\ksup{k+1}$. So we have showed that $S(\nu g_\gm\ksup{k},\widetilde{g}_\dl)$ reduces to zero via $\widetilde{g}_\gm$ and $\ztau\ksup{k+1}g_\dl\ksup{k+1}$, both of which are in $\cG^\prime$.

\noi\emph{Case 3:} Here we have assumed that $\ztau\ksup{k+1}$ is not incident to $\gm$, and we continue to assume that it divides only the leading term of $g_\dl\ksup{k}=g_\dl\ksup{k+1}$. We would like to reduce the expression in line~(\ref{reduceme}). Since $\ztau\ksup{k+1}$ is not incident to $\gm$, we do not have $\widetilde{g}_\gm$ available in $\cG^\prime$ in this case. Instead, we must expand line~(\ref{reduceme}) using Propositions~\ref{subsetoverlapinterior},~\ref{subsetoverlapclosure}, and~\ref{subsetoverlapboundary}, keeping careful track of which term retains the factor of $\nu$. Since $\ztau\ksup{k+1}$ is not incident to $\gm$, we know that $\gm$ is out-led. The computation will depend on whether $\dl$ is in-led or out-led. 

Suppose first that $\dl$ is in-led. Then $\ztau\ksup{k+1}$ goes from $G\sm(\gm\cup\dl)$ to $\dl\sm(\gm\cap\dl)$. We use Statement~(\ref{subsetinteriorout}) in each of Propositions~\ref{subsetoverlapinterior},~\ref{subsetoverlapclosure}, and~\ref{subsetoverlapboundary} to expand. We use $\widetilde{g}_{\dl\sm(\gm\cap\dl)}$ to reduce the resulting expression. It is available in $\cG^\prime$ because $\ztau\ksup{k+1}$ is incident but not internal to $\dl\sm(\gm\cap\dl)$.
\begin{align*}
S(g_\gm\ksup{k}, \widetilde{g}_\dl)
=&\,\xsub{G\sm(\gm\cap\dl)}{\gm\cap\dl}\zsub{G\sm(\gm\cap\dl)}{\gm\cap\dl}\ksup{k+1}\zsub{\beta}{\gm\cap\dl}\ksup{k+1}\\
&\cdot\,\left(\nu g_{\dl\sm(\gm\cap\dl)}^{(k+1),\out}g_{\gm\sm(\gm\cap\dl)}^{(k+1),\into} - g_{\gm\sm(\gm\cap\dl)}^{(k+1),\out}g_{\dl\sm(\gm\cap\dl)}^{(k+1),\into}\right)\\
\text{reduce}\quad-\,&\widetilde{g}_{\dl\sm(\gm\cap\dl)}\cdot \xsub{G\sm(\gm\cap\dl)}{\gm\cap\dl}\zsub{G\sm(\gm\cap\dl)}{\gm\cap\dl}\ksup{k+1}\zsub{\beta}{\gm\cap\dl}\ksup{k+1}g_{\gm\sm(\gm\cap\dl)}^{(k+1),\into} \\
=&\,\xsub{G\sm(\gm\cap\dl)}{\gm\cap\dl}\zsub{G\sm(\gm\cap\dl)}{\gm\cap\dl}\ksup{k+1}\zsub{\beta}{\gm\cap\dl}\ksup{k+1}g_{\dl\sm(\gm\cap\dl)}^{(k+1),\into}g_{\gm\sm(\gm\cap\dl)}\ksup{k+1}.
\end{align*}
Since $\ztau\ksup{k+1}$ divides$g_{\dl\sm(\gm\cap\dl)}^{(k+1),\into}$, we may reduce the above expression to zero by $\ztau\ksup{k+1}g_{\gm\sm(\gm\cap\dl)}\ksup{k+1}$, which is in $\cG^\prime$.

Suppose instead that $\dl$ is out-led. Then $\ztau\ksup{k+1}$ goes from $\dl\sm(\gm\cap\dl)$ to $G\sm(\gm\cup\dl)$. We use Statement~(\ref{subsetinteriormix}) in each of Propositions~\ref{subsetoverlapinterior},~\ref{subsetoverlapclosure}, and~\ref{subsetoverlapboundary}, with the roles of $\gm$ and $\dl$ reversed. We will use $\widetilde{g}_{\gm\cup\dl}$ to reduce the resulting expression. It is available in $\cG^\prime$ because $\ztau\ksup{k+1}$ is incident but not internal to $\gm\cup\dl$.
\begin{align*}
S(g_\gm\ksup{k}, \widetilde{g}_\dl)=&\,\xsub{\dl\sm(\gm\cap\dl)}{\gm\sm(\gm\cap\dl)}\zsub{\dl\sm(\gm\cap\dl)}{\gm\sm(\gm\cap\dl)}\ksup{k+1}\\
&\cdot\,\left(\nu g_{\gm\cup\dl}^{(k+1),\into}g_{\gm\cap\dl}^{(k+1),\into}-g_{\gm\cup\dl}^{(k+1),\out}g_{\gm\cap\dl}^{(k+1),\out}\right)\\
\text{reduce}\quad-\,&\widetilde{g}_{\gm\cup\dl}\cdot\xsub{\dl\sm(\gm\cap\dl)}{\gm\sm(\gm\cap\dl)}\zsub{\dl\sm(\gm\cap\dl)}{\gm\sm(\gm\cap\dl)}\ksup{k+1}g_{\gm\cap\dl}^{(k+1),\into}\\
=&\,\xsub{\dl\sm(\gm\cap\dl)}{\gm\sm(\gm\cap\dl)}\zsub{\dl\sm(\gm\cap\dl)}{\gm\sm(\gm\cap\dl)}\ksup{k+1}g_{\gm\cup\dl}^{(k+1),\out}g_{\gm\cap\dl}\ksup{k+1}
\end{align*}
Since $\ztau\ksup{k+1}$ divides $g_{\gm\cup\dl}^{(k+1),\out}$, we may reduce the above expression to zero by $\ztau\ksup{k+1}g_{\gm\cap\dl}\ksup{k+1}$, which is in $\cG^\prime$.
\end{proof}

\section{Extending to non-blackboard framings}
\label{nonbbframings}

We adopted the blackboard framing assumption in Section~\ref{startingbasissec} when setting up the starting basis for Buchberger's Algorithm. We now pick up from there to describe the modifications necessary to generalize to arbitrary framings. Refer to Section~\ref{twistedexamplesection} to see how these modifications change the set-up of the small example in Section~\ref{examplesection}.

\subsection{Set-up}
\label{twistedsetup}
In non-blackboard framed graphs, the weight of $\gm$ as a subset in $G\ksup{i}$ depends on $i$. As we close strands of the braid, non-blackboard-framed boundary edges may become internal to $\gm$, and therefore contribute to its weight. Let $\weight_i(\gm)$ denote the weight of $\gm$ as a subset in $G\ksup{i}$. That is, $\weight_i(\gm)$ is the sum of the framings on edges that are internal to $\gm$ in $G\ksup{i}$. Recall that the generator of $N_i$ associated to a set $\gm$ was defined to be
$t^{\weight_i(\gm)}x_{\gm,\out}-y_{\gm,\into}$.

When we pass to the edge ring, the monomial $y_{\gm,\into}$ picks up a coefficient as each $y_i$ is replaced by $t^{-\ell_i}x_i$, where $\ell_i$ denotes the framing on the edge labeled by $x_i$ and $y_i$. Let $\weight_{i,\into}(\gm)$ denote the sum of the framings on edges incoming to $\gm$ in $G\ksup{i}$ from its complement or from the bottom boundary. There is a dependence on $i$ because edges that were incoming to $\gm$ from the bottom boundary may become internal to $\gm$ as we close braid strands. It will be convenient to clear denominators in our standard generators for $N_i$ before beginning Buchberger's Algorithm, so we multiply the usual generators by $t^{\weight_{i,\into}}$. Using the more detailed notation of Section~\ref{mosec}, we let
\begin{equation}
h_\gm\ksup{k}=t^{\weight_k(\gm)+\weight_{k,\into}(\gm)}x_{\gm,G\sm\gm}z_{\gm,G\sm\gm}\ksup{k}z_{\gm,\tau}\ksup{k}-x_{G\sm\gm,\gm}z_{G\sm\gm,\gm}\ksup{k}z_{\beta,\gm}\ksup{k}
\end{equation}
be the new standard form of a generator of $N_k$ and
\begin{equation}
\label{ggmkplusonetw}
g_\gm\ksup{k+1}=t^{\weight_{k+1}(\gm)+\weight_{k+1,\into}(\gm)}x_{\gm,G\sm\gm}z_{\gm,G\sm\gm}\ksup{k+1}z_{\gm,\tau}\ksup{k+1}-x_{G\sm\gm,\gm}z_{G\sm\gm,\gm}\ksup{k+1}z_{\beta,\gm}\ksup{k+1}
\end{equation}
be the new standard form of a generator of $N_{k+1}$. Compare to Equations~\ref{ggmkdfn} and~\ref{ggmkplusonedfn}.

Applying $\pi_k$ to $h_\gm\ksup{k}$ may also cause it to pick up a new factor of $t$. Recall that $\pi_i$ is the quotient map $\cE_i\to\cE_{i+1}$ with kernel $Z_{i+1}=\left(\ztau\ksup{i+1}-t^{a_{i+1}}\zbeta\ksup{i+1}\right)$, where $a_{i+1}$ is the framing on the top edge of the $(i+1)^{\text{st}}$ strand of $G$. Recall also that $\pi_i$ retains the label $\ztau\ksup{i+1}$. So applying $\pi_k$ produces a factor of $t^{-a_{k+1}}$ whenever $\zbeta\ksup{k+1}$ appears.
Let
\[\weight_{k,\tau}(\gm)=\begin{cases}
a_{k+1} & \text{if }\zbeta\ksup{k+1} \text{ is incoming to $\gm$ in }G\ksup{k}\\
0 & \text{otherwise.}
\end{cases}\]
Then we have
\begin{align*}
\pi_k\!\left(h_\gm\ksup{k}\right)&=t^{\weight_{k}(\gm)+\weight_{k,\into}(\gm)}x_{\gm,G\sm\gm}z_{\gm,G\sm\gm}\ksup{k+1}z_{\gm,\tau}\ksup{k+1}\anyzeta_\gm\\&-x_{G\sm\gm,\gm}z_{G\sm\gm,\gm}\ksup{k+1}z_{\beta,\gm}\ksup{k+1}t^{-\weight_{k,\tau}(\gm)}\anyzeta_\gm,
\end{align*}
where $\anyzeta_\gm$ is defined as in Equation~\ref{anyzetadfn} to be $\ztau\ksup{k+1}$ if $\ztau\ksup{k+1}$ is internal to $\gm$ in $G\ksup{k+1}$ and $1$ otherwise. The collection of $\pi_k\!\left(h_\gm\ksup{k}\right)$, taken over all subsets $\gm$ in $G$, is a basis for $\pi_k(N_k)$ just as it was in the blackboard framed case. We will clear denominators to arrive at a more convenient starting basis for Buchberger's Algorithm. Define
\begin{equation}
\label{ggmktw}
g_\gm\ksup{k}=t^{\weight_{k,\tau}(\gm)}\pi_k\!\left(h_\gm\ksup{k}\right).
\end{equation}
The collection of $g_\gm\ksup{k}$ is also a basis for $\pi_k(N_k)$. 

The advantage of clearing denominators is that $g_\gm\ksup{k}$ and $g_\gm\ksup{k+1}$ now have the same relationship as in the blackboard framed case:
$g_\gm\ksup{k}=\anyzeta_\gm g_\gm\ksup{k+1}$,
just as in Equation~\ref{ggmkvkplusone}. The equality follows from the fact that \[\weight_k(\gm)+\weight_{k,\into}(\gm)+\weight_{k,\tau}(\gm)=\weight_{k+1}(\gm)+\weight_{k+1,\into}(\gm)\] for any $\gm$.
With the new notation of this section, we may again say that 
\begin{equation*}
\cG_0=\left\{\nu g_\gm\ksup{k}\,\vert\,\gm\subset G\right\}\cup\left\{\nu\ztau\ksup{k+1}-\ztau\ksup{k+1}\right\},
\end{equation*}
is our starting basis for Buchberger's Algorithm. Compare to Equation~\ref{startingbasis}.

\subsection{Modifications for Section~\ref{examplesection} example}
\label{twistedexamplesection}
Refer to Figure~\ref{smalltwistedexample}. The various weights associated to $\gm$, $\dl$, and $\gm\cup\dl$ are shown in Table~\ref{weighttable}. The generators associated to $\gm$, $\dl$, and $\gm\cup\dl$ are as follows. Observe that the generators of $N_1$ and $N_2$ are related exactly as they were in the blackboard framing case after we have modified their coefficients as described in the previous section.

\begin{table}
\begin{center}
\renewcommand{\tabcolsep}{11pt}
\begin{tabular}{lll}
$\weight_1(\gm)=\ell_0$ &
$\weight_1(\dl)=0$ &
$\weight_1(\gm\cup\dl)=\ell_0+\ell_2$ \\
$\weight_{1,\into}(\gm)=\ell_2$ &
$\weight_{1,\into}(\dl)=\ell_3+\ell_5$ &
$\weight_{1,\into}(\gm\cup\dl)=\ell_3+\ell_5$ \\
$\weight_{1,\tau}(\gm)=0$ &
$\weight_{1,\tau}(\dl)=\ell_1$ &
$\weight_{1,\tau}(\gm\cup\dl)=\ell_1$ \\
$\weight_2(\gm)=\ell_0$ &
$\weight_2(\dl)=0$ &
$\weight_2(\gm\cup\dl)=\ell_0+\ell_1+\ell_2+\ell_3$ \\
$\weight_{2,\into}(\gm)=\ell_2$ &
$\weight_{2,\into}(\dl)=\ell_1+\ell_3+\ell_5$ &
$\weight_{2,\into}(\gm\cup\dl)=\ell_5$ \\
\end{tabular}
\end{center}
\caption{Weights for the graphs in Figure~\ref{smalltwistedexample}.}
\label{weighttable}
\end{table}

\begin{align*}
h_\gm\ksup{1}&=t^{\weight_{1,\into}(\gm)}\left(t^{\weight_1(\gm)}x_1-t^{-\ell_2}x_2\right)\\
&=t^{\ell_0+\ell_2}x_1-x_2\\
g_\gm\ksup{1}&=t^{\weight_{1,\tau}(\gm)}\pi_1\!\left(h_\gm\ksup{1}\right)\\
&=t^{\ell_0+\ell_2}x_1-x_2\\
g_\gm\ksup{2}&=t^{\weight_{2,\into}(\gm)}\left(t^{\weight_2(\gm)}x_1-t^{-\ell_2}x_2\right)\\
&=t^{\ell_0+\ell_2}x_1-x_2
\end{align*}

\begin{align*}
h_\dl\ksup{1}&=t^{\weight_{1,\into}(\dl)}\left(t^{\weight_1(\dl)}x_2x_4-t^{-\ell_3-\ell_5}x_3x_5\right)\\
&=t^{\ell_3+\ell_5}x_2x_4-x_3x_5\\
g_\dl\ksup{1}&=t^{\weight_{1,\tau}(\dl)}\pi_1\!\left(h_\dl\ksup{1}\right)\\
&=t^{\ell_1}\left(t^{\ell_3+\ell_5}x_2x_4-t^{-\ell_1}x_1x_5\right)\\
&=t^{\ell_1+\ell_3+\ell_5}x_2x_4-x_1x_5\\
g_\dl\ksup{2}&=t^{\weight_{2,\into}(\dl)}\left(t^{\weight_2(\dl)}x_2x_4-t^{-\ell_1-\ell_3-\ell_5}x_1x_5\right)\\
&=t^{\ell_1+\ell_3+\ell_5}x_2x_4-x_1x_5
\end{align*}

\begin{align*}
h_{\gm\cup\dl}\ksup{1}&=t^{\weight_{1,\into}(\gm\cup\dl)}\left(t^{\weight_1(\gm\cup\dl)}x_1x_4-t^{-\ell_3-\ell_5}x_3x_5\right)\\
&=t^{\ell_0+\ell_2+\ell_3+\ell_5}x_1x_4-x_3x_5\\
g_{\gm\cup\dl}\ksup{1}&=t^{\weight_{1,\tau}(\gm\cup\dl)}\pi_1\!\left(h_{\gm\cup\dl}\ksup{1}\right)\\
&=t^{\ell_1}\left(t^{\ell_0+\ell_2+\ell_3+\ell_5}x_1x_4-t^{-\ell_1}x_1x_5\right)\\
&=t^{\ell_0+\ell_1+\ell_2+\ell_3+\ell_5}x_1x_4-x_1x_5\\
g_{\gm\cup\dl}\ksup{2}&=t^{\weight_{2,\into}(\gm\cup\dl)}\left(t^{\weight_2(\gm\cup\dl)}x_4-t^{-\ell_5}x_5\right)\\
&=t^{\ell_0+\ell_1+\ell_2+\ell_3+\ell_5}x_4-x_5
\end{align*}

\begin{figure}
\begin{center}
$$\scalebox{1.3}{\input{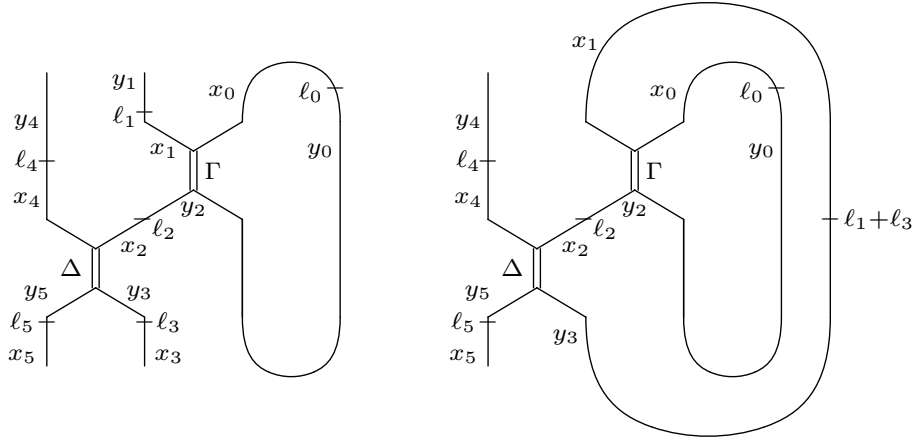}}$$
\caption{Modification of Figure~\ref{smallexample} for arbitrary framings.}
\label{smalltwistedexample}
\end{center}
\end{figure}

\subsection{Buchberger's Algorithm}
We use the same monomial order on $\cE_k$ from Definition~\ref{edgeringmo}. Coefficients play no role in the definition or application of a monomial order, so our analysis of the leading terms of standard generators of $N(G\ksup{k})$ and $Z_k$ in Sections~\ref{mosec} and~\ref{startingbasissec} goes through unchanged for arbitrary framings.

The propositions in which S-polynomials are computed also change very little, since analysis of greatest common divisors and least common multiples of leading monomials is not affected by the presence of coefficients. We merely have to check that the exponents on the factors of $t$ work out correctly.  They do, because both $\weight_i+\weight_{i,\into}$ and $\weight_{k,\tau}$ are additive under disjoint union. Define a total weight $\Weight$ by
\[\Weight(\gm)=\weight_k(\gm)+\weight_{k,\into}(\gm)+\weight_{k,\tau}(\gm)=\weight_{k+1}(\gm)+\weight_{k+1,\into}(\gm).\] Note that $\Weight$ is also additive under disjoint union.

Analyses involving the monomial order are also unaffected by the presence of coefficients, so reduction arguments generally proceed in the same way as before. In particular, Observation~\ref{MOanddiv} carries through unchanged.

One technical note is in order before we check through Sections~\ref{bbrd1} and~\ref{bbrd2}: We will clear denominators before adding any new generator to our working basis. This does not affect the progress of the algorithm. If $f, g, h\in\itk[\underline{x}]$ and $a\in\itk$, then $S(af,g)=S(f,g)$ and $g\xrightarrow{f}h$ if any only if $g\xrightarrow{af}h$. That is, multiplying a generator in the working basis by an element of the ground field does not affect its properties with respect to S-polynomials or the division algorithm, which means that it does not affect its role in Buchberger's Algorithm.

\subsection*{Reprise of Section~\ref{nutoponespolys}}

Proposition~\ref{nutopone} is already stated in sufficient generality for arbitrary framings. After clearing denominators, the application of Proposition~\ref{nutopone} goes through as explained in the blackboard case. In particular, it remains true that if $\ztau\ksup{k+1}$ does not divide both terms of $g_\gm\ksup{k}$, then $g_\gm\ksup{k}=g_\gm\ksup{k+1}$. So, if $\ztau\ksup{k+1}$ fails to divide the leading term of $g_\gm\ksup{k}$, then we apply Statement~(\ref{nutop2nunodiv}) of Proposition~\ref{nutopone} and add $\ztau\ksup{k+1}g_\gm\ksup{k+1}$ to the working basis $\cG^\prime$.

Let $g_\gm^{(k),\out}$, $g_\gm^{(k),\into}$, $g_\gm^{(k+1),\out}$, and $g_\gm^{(k+1),\into}$ be monomials defined exactly as in the blackboard case (i.e., all with coefficient 1). Define the tilde generators to be
\begin{align*}
\widetilde{g}_\gm&=\nu g_\gm^{(k),\into}-t^{\Weight(\gm)}g_\gm^{(k), \out} \text{ if } \gm \text{ is out-led}\\
\widetilde{g}_\gm&= \nu t^{\Weight(\gm)}g_\gm^{(k), \out}-g_\gm^{(k), \into}
\text{ if } \gm \text{ is in-led.}
\end{align*}
If $\ztau\ksup{k+1}$ divides the leading term of $g_\gm\ksup{k}$, then we apply Statement~(\ref{nutop2nu}) of Proposition~\ref{nutopone} to see that the S-polynomial between $\nutopk$ and $\nu g_\gm\ksup{k}$ still produces $\widetilde{g}_\gm$, possibly after clearing denominators. If $\ztau\ksup{k+1}$ also divides the trailing term of $g_\gm\ksup{k}$, then we have the same reduction to $\ztau\ksup{k+1}g_\gm\ksup{k+1}$ as in the blackboard case.

Overall, then, S-polynomials between $\nutopk$ and generators of the form $\nu g_\gm\ksup{k}$ yield the working basis $\cG^\prime$ defined at the end of Section~\ref{nutoponespolys}.

\subsection*{Reprise of Section~\ref{subsetspolys}} Propositions~\ref{separateinteriorclosure},~\ref{subsetoverlapinterior},~\ref{subsetoverlapclosure}, and~\ref{subsetoverlapboundary} are not specific to the blackboard case. We may still use these as building blocks to describe the outcome of S-polynomials among generators of the form $\nu g_\gm\ksup{k}$ for arbitrary framings -- that is, to prove the following analogue of Proposition~\ref{ksubsetoverlap}. 

\begin{prop}[Analogue of Proposition~\ref{ksubsetoverlap}]
\label{twistedksubsetoverlap}
Let $\gm, \dl\subset G$. After clearing denominators, the following statements hold in $\cE_{k+1}$, up to multiplication by non-zero elements of the ground field.
\begin{align*}
\tag{1}\label{twistedksubsetoverlapout}
\text{If $\gm$ and $\dl$ are both out-led, then}\\
S(g_\gm\ksup{k+1},g_\dl\ksup{k+1})=&\,\xsub{G\sm(\gm\cup\dl)}{\gm\cap\dl}\zsub{G\sm(\gm\cup\dl)}{\gm\cap\dl}\ksup{k+1}\zsub{\beta}{\gm\cap\dl}\ksup{k+1}\\
&\cdot\left(t^{\Weight(\dl\sm(\gm\cap\dl))}g_{\dl\sm(\gm\cap\dl)}^{(k+1), \out}g_{\gm\sm(\gm\cap\dl)}^{(k+1), \into}\right.\\
&\left.-\,t^{\Weight(\gm\sm(\gm\cap\dl))}g_{\gm\sm(\gm\cap\dl)}^{(k+1), \out}g_{\dl\sm(\gm\cap\dl)}^{(k+1), \into}\right)\\
&\xrightarrow{g_{\gm\sm(\gm\cap\dl)}\ksup{k+1},g_{\dl\sm(\gm\cap\dl)}\ksup{k+1}}0\\
\tag{2}\label{twistedksubsetoverlapin}
\text{If $\gm$ and $\dl$ are both in-led then}\\
S(g\ksup{k+1}_\gm,g\ksup{k+1}_\dl)&=\,\xsub{\gm\cap\dl}{G\sm(\gm\cup\dl)}\zsub{\gm\cap\dl}{G\sm(\gm\cup\dl)}\ksup{k+1}\zsub{\gm\cap\dl}{\tau}\ksup{k+1}\\
&\cdot\left(t^{\Weight(\gm\sm(\gm\cap\dl))}g_{\dl\sm(\gm\cap\dl)}^{(k+1),\into}g_{\gm\sm(\gm\cap\dl)}^{(k+1),\out}\right.\\
&\left.-\,t^{\Weight(\dl\sm(\gm\cap\dl))}g_{\gm\sm(\gm\cap\dl)}^{(k+1),\into}g_{\dl\sm(\gm\cap\dl)}^{(k+1),\out}\right)\\
&\xrightarrow{g_{\gm\sm(\gm\cap\dl)}\ksup{k+1},g_{\dl\sm(\gm\cap\dl)}\ksup{k+1}}0\\
\tag{3}\label{twistedksubsetoverlapmix}
\text{If $\gm$ is in-led and $\dl$ is out-led, then}\\
S(g_\gm\ksup{k+1},g_\dl\ksup{k+1})&=\,\xsub{\gm\sm(\gm\cap\dl)}{\dl\sm(\gm\cap\dl)}\zsub{\gm\sm(\gm\cap\dl)}{\dl\sm(\gm\cap\dl)}\ksup{k+1}\\
&\cdot\left(t^{\Weight(\gm\cup\dl)+\Weight(\gm\cap\dl)}g_{\gm\cup\dl}^{(k+1),\out} g_{\gm\cap\dl}^{(k+1),\out}\right.\\
&\left.-\,g_{\gm\cup\dl}^{(k+1),\into} g_{\gm\cap\dl}^{(k+1),\into}\right)\\
&\xrightarrow{g_{\gm\cup\dl}^{(k+1)},g_{\gm\cap\dl}\ksup{k+1}}0.
\end{align*}
\end{prop}
\begin{proof}
Apply Propositions~\ref{separateinteriorclosure},~\ref{subsetoverlapinterior},~\ref{subsetoverlapclosure}, and~\ref{subsetoverlapboundary} to the definitions of $g_\gm\ksup{k+1}$ and $g_\dl\ksup{k+1}$ given in Equation~\ref{ggmkplusonetw}, then multiply by appropriate powers of $t$ to clear any negative exponents. The reduction statements follow from Proposition~\ref{unorderedreduction}.
\end{proof}

Lemma~\ref{subsetoverlap} remains true as stated. We follow through the details of the proof just to be careful.

\begin{proof}[Proof of Lemma~\ref{subsetoverlap} for arbitrary framings]
We established at the beginning of Section~\ref{nonbbframings} that \[g_\gm\ksup{k}=\anyzeta_\gm g_\gm\ksup{k+1}\] holds for arbitrary framings. We have also seen that Proposition~\ref{unorderedreduction} and Observation~\ref{MOanddiv} carry through unchanged, so the division into cases in the proof of Lemma~\ref{subsetoverlap} remains valid. Proposition~\ref{twistedksubsetoverlap} (the analogue of Proposition~\ref{ksubsetoverlap}) immediately allows us to reduce the S-polynomials in Cases 1 and 2(a) to zero using generators in the working basis.

The analysis of Case 2(b) proceeds as before, establishing that $\ztau\ksup{k+1}$ goes between $\gm\cap\dl$ and $\gm\sm(\gm\cap\dl)$. If $\ztau\ksup{k+1}$ is outgoing from $\gm\cap\dl$, then Proposition~\ref{twistedksubsetoverlap} says that
\begin{align*}
\nu S(g_\gm\ksup{k+1},g_\dl\ksup{k+1})&=\,\nu \xsub{\gm\sm(\gm\cap\dl)}{\dl\sm(\gm\cap\dl)}\zsub{\gm\sm(\gm\cap\dl)}{\dl\sm(\gm\cap\dl)}\ksup{k+1}\\
&\cdot\left(t^{\Weight(\gm\cup\dl)+\Weight(\gm\cap\dl)}g_{\gm\cup\dl}^{(k+1),\out} g_{\gm\cap\dl}^{(k+1),\out}-g_{\gm\cup\dl}^{(k+1),\into} g_{\gm\cap\dl}^{(k+1),\into}\right)
\end{align*}
with the term order shown. As in the original proof, we may reduce by $\nu\ztau\ksup{k+1}-\ztau\ksup{k+1}$ to obtain
\begin{align*}
\nu S(g_\gm\ksup{k+1},g_\dl\ksup{k+1})\xrightarrow{\nutopk}\, &\xsub{\gm\sm(\gm\cap\dl)}{\dl\sm(\gm\cap\dl)}\zsub{\gm\sm(\gm\cap\dl)}{\dl\sm(\gm\cap\dl)}\ksup{k+1}\\
&\cdot\left(\nu g_{\gm\cup\dl}^{(k+1),\into} g_{\gm\cap\dl}^{(k+1),\into}\right.\\
&\left. -\, t^{\Weight(\gm\cup\dl)+\Weight(\gm\cap\dl)}g_{\gm\cup\dl}^{(k+1),\out} g_{\gm\cap\dl}^{(k+1),\out}\right)
\end{align*}
and then by $$\widetilde{g}_{\gm\cap\dl}=\nu g_{\gm\cap\dl}^{(k+1),\into} - t^{\Weight(\gm\cap\dl)}g_{\gm\cap\dl}^{(k+1),\out}$$
to produce
\begin{align*}
&\quad\xsub{\gm\sm(\gm\cap\dl)}{\dl\sm(\gm\cap\dl)}\zsub{\gm\sm(\gm\cap\dl)}{\dl\sm(\gm\cap\dl)}\ksup{k+1}t^{\Weight(\gm\cap\dl)}g_{\gm\cap\dl}^{(k+1),\out}g_{\gm\cup\dl}\ksup{k+1}
\end{align*}
and finally by $\ztau\ksup{k+1}g_{\gm\cup\dl}\ksup{k+1}$ to get zero. There is a similar argument for the reduction if $\ztau\ksup{k+1}$ is incoming to $\gm\cap\dl$.

In the analysis of Case 3, the initial argument remains the same, so we may assume without loss of generality that $\ztau\ksup{k+1}$ goes from $\dl\sm(\gm\cap\dl)$ to $\gm\sm(\gm\cap\dl)$. We may likewise assume that the working basis contains tilde generators for $\gm$ and $\dl$ with the forms
\[\widetilde{g}_\gm = \nu t^{\Weight(\gm)} g_\gm^{(k+1),\out} - g_\gm^{(k+1),\into} \quad\text{and}\quad\widetilde{g}_\dl = \nu g_\dl^{(k+1),\into} - t^{\Weight(\dl)}g_\dl^{(k+1),\out}.\]
The appropriate re-factoring of Statement (3) in Proposition~\ref{twistedksubsetoverlap} is then 
\begin{align*}
\nu S(g_\gm\ksup{k+1},\, &g_\dl\ksup{k+1})=\\
%
&\nu t^{\Weight(\dl)}\xsub{\dl\sm(\gm\cap\dl)}{G\sm(\gm\cup\dl)}\xsub{\gm\cap\dl}{G\sm\dl}\\
\,\cdot&\zsub{\dl\sm(\gm\cap\dl)}{G\sm(\gm\cup\dl)}\zsub{\gm\cap\dl}{G\sm\dl}\zsub{\dl}{\tau}\ksup{k+1}
  t^{\Weight(\gm)}g_\gm^{(k+1),\out}\\
\,- &\nu\xsub{G\sm(\gm\cup\dl)}{\gm\sm(\gm\cap\dl)}\xsub{G\sm\gm}{\gm\cap\dl}\zsub{G\sm(\gm\cup\dl)}{\gm\sm(\gm\cap\dl)}\zsub{G\sm\gm}{\gm\cap\dl}\zsub{\beta}{\gm}\ksup{k+1}
 g_{\dl}^{(k+1),\into}
\end{align*}
with unknown term order. As in the argument for blackboard framings, this expression is reducible by either $\widetilde{g}_\gm$ or $\widetilde{g}_\dl$ depending on its term order, and then by whichever of the tilde generators was not already used. For the sake of illustration, we reduce by $\widetilde{g}_\gm$ and then $\widetilde{g}_\dl$ as follows.
\begin{align*}
\nu S(&g_\gm\ksup{k+1}, g_\dl\ksup{k+1})\xrightarrow{\widetilde{g}_\gm}\\
&\nu\xsub{G\sm(\gm\cup\dl)}{\gm\sm(\gm\cap\dl)}\xsub{G\sm\gm}{\gm\cap\dl}\zsub{G\sm(\gm\cup\dl)}{\gm\sm(\gm\cap\dl)}\zsub{G\sm\gm}{\gm\cap\dl}\zsub{\beta}{\gm}\ksup{k+1}
 g_{\dl}^{(k+1),\into}\\
%
\,- & t^{\Weight(\dl)}\xsub{\dl\sm(\gm\cap\dl)}{G\sm(\gm\cup\dl)}\xsub{\gm\cap\dl}{G\sm\dl}\zsub{\dl\sm(\gm\cap\dl)}{G\sm(\gm\cup\dl)}\zsub{\gm\cap\dl}{G\sm\dl}\zsub{\dl}{\tau}\ksup{k+1}
  g_\gm^{(k+1),\into}\\
\xrightarrow{\widetilde{g}_\dl}\\
& t^{\Weight(\dl)}\xsub{\dl\sm(\gm\cap\dl)}{G\sm(\gm\cup\dl)}\xsub{\gm\cap\dl}{G\sm\dl}\zsub{\dl\sm(\gm\cap\dl)}{G\sm(\gm\cup\dl)}\zsub{\gm\cap\dl}{G\sm\dl}\zsub{\dl}{\tau}\ksup{k+1}
  g_\gm^{(k+1),\into}\\
\,- &t^{\Weight(\dl)}\xsub{G\sm(\gm\cup\dl)}{\gm\sm(\gm\cap\dl)}\xsub{G\sm\gm}{\gm\cap\dl}\zsub{G\sm(\gm\cup\dl)}{\gm\sm(\gm\cap\dl)}\zsub{G\sm\gm}{\gm\cap\dl}\zsub{\beta}{\gm}\ksup{k+1}
 g_{\dl}^{(k+1),\out}.
\end{align*}
The final expression can be seen to be zero by re-factoring exactly as in the original proof.
\end{proof}

This completes our reprise of Section~\ref{bbrd1}. We are left with the same working basis $\cG^\prime$ as in the blackboard case, given the changes to notation described at the beginning of Section~\ref{nonbbframings}.

\subsection*{Reprise of Section~\ref{bbrd2}}
As in the blackboard case, we must confirm that S-polynomials involving the generators added to the working basis in Round 1 may be reduced to zero within the working basis. The arguments generalize easily to arbitrary framings, but we verify them line-by-line for the sake of completeness.

Applying Statement~(\ref{nutop0nu}) of Proposition~\ref{nutopone} to $S(\nu\ztau\ksup{k+1}-\ztau\ksup{k+1},\ztau\ksup{k+1}g_\gm\ksup{k+1})$ produces a multiple of the blackboard result by $1/\lc{g_\gm\ksup{k+1}}$. That result can be reduced to zero in the same way as before just by multiplying each step of the computation by the same factor of $1/\lc{g_\gm\ksup{k+1}}$. So, as before, \[S(\nutopk,\ztau\ksup{k+1}g_\gm\ksup{k+1})\xrightarrow{\nutopk, \ztau\ksup{k+1}g_\gm\ksup{k+1}}0.\]

Propositions~\ref{gbprinciplecoeff} and~\ref{twistedksubsetoverlap} may be applied to
$S(\ztau\ksup{k+1}g_\gm\ksup{k+1},\ztau\ksup{k+1}g_\dl\ksup{k+1})$ to see that it reduces by some combination of $\ztau\ksup{k+1}g_\Lambda\ksup{k+1}$ for $\Lambda\in\{\gm\sm(\gm\cap\dl)$, $\dl\sm(\gm\cap\dl)$, $\gm\cap\dl$, $\gm\cup\dl\}$, all of which are in $\cG^\prime$.

Lemma~\ref{ztautilde} holds as stated for arbitrary framings. The proof is the same, except that the computations must be multiplied by $\displaystyle{\frac{1}{\lc{g_\gm\ksup{k+1}}\tc{g_\dl\ksup{k+1}}}},$ where $\mathrm{TC}$ denotes the coefficient on the trailing term. These factors do not affect any analysis about leading terms or the availability of the generators needed to reduce the S-polynomial to zero, so the argument for reduction to zero is unchanged.

Lemma~\ref{doubletilde} also holds as stated for arbitrary framings. The expression that must be reduced in the proof for arbitrary framings is the product of the expression in Line~\ref{expandmedoubletilde} of the blackboard case with $\displaystyle{\frac{1}{\tc{g_\gm\ksup{k+1}}\tc{g_\dl\ksup{k+1}}}}$. The method of reduction then depends on how this expression expands in terms of $g_\Lambda\ksup{k+1}$ for $\Lambda\in\left\{\gm\sm(\gm\cap\dl),\dl\sm(\gm\cap\dl),\gm\cap\dl,\gm\cup\dl\right\}$. Just as the possible cases paralleled Proposition~\ref{ksubsetoverlap} in the blackboard setting, they parallel Proposition~\ref{twistedksubsetoverlap} in the non-blackboard setting. In each case, the arguments about term order and reduction to zero by generators in $\cG^\prime$ are identical in the blackboard and non-blackboard settings.

The proof of Lemma~\ref{2nunonu} relies entirely on results that we have already generalized to arbitrary framings, so we may conclude that the proposition holds as stated for arbitrary framings.

Finally, Lemma~\ref{tildeprop} holds as stated. Since its proof involves some explicit calculations, we check it line by line. 

\begin{proof}[Proof of Lemma~\ref{tildeprop} for arbitrary framings]
The breakdown into cases is unaffected by the presence of coefficients. In Cases 1 and 2, the expanded S-polynomials for arbitrary framings are the product of the expressions in the blackboard case by a factor of $\displaystyle{\frac{1}{\lc{g_\gm\ksup{k+1}}\tc{g_\dl\ksup{k+1}}}}$. Multiplying through by that factor while reducing makes the reduction by $\nu\ztau\ksup{k+1}-\ztau\ksup{k+1}$ and $\ztau\ksup{k+1}g_\Lambda\ksup{k+1}$ or $\widetilde{g}_\gm$ and $\ztau\ksup{k+1}g_\dl\ksup{k+1}$ possible, just as before.

Case 3 breaks into subcases as in the blackboard case. When $\ztau\ksup{k+1}$ goes from $G\sm(\gm\cup\dl)$ to $\dl\sm(\gm\cap\dl)$, the expanded S-polynomial expression with coefficients is 
\begin{align*}
S(g_\gm\ksup{k}, \widetilde{g}_\dl)
=&\,\xsub{G\sm(\gm\cap\dl)}{\gm\cap\dl}\zsub{G\sm(\gm\cap\dl)}{\gm\cap\dl}\ksup{k+1}\zsub{\beta}{\gm\cap\dl}\ksup{k+1}\\
&\cdot\,\left(\nu t^{\Weight(\dl\sm(\gm\cap\dl)}g_{\dl\sm(\gm\cap\dl)}^{(k+1),\out}g_{\gm\sm(\gm\cap\dl)}^{(k+1),\into} - t^{\Weight(\gm\sm(\gm\cap\dl))}g_{\gm\sm(\gm\cap\dl)}^{(k+1),\out}g_{\dl\sm(\gm\cap\dl)}^{(k+1),\into}\right).
\end{align*}
This expression reduces by $\widetilde{g}_{\dl\sm(\gm\cap\dl)}$ and then $\ztau\ksup{k+1}g_{\gm\sm(\gm\cap\dl)}\ksup{k+1}$ as before. When $\ztau\ksup{k+1}$ goes from $\dl\sm(\gm\cap\dl)$ to $G\sm(\gm\cup\dl)$, the expanded S-polynomial is
\begin{align*}
S(g_\gm\ksup{k}, \widetilde{g}_\dl)=&\,\xsub{\dl\sm(\gm\cap\dl)}{\gm\sm(\gm\cap\dl)}\zsub{\dl\sm(\gm\cap\dl)}{\gm\sm(\gm\cap\dl)}\ksup{k+1}\\
&\cdot\,\left(\nu g_{\gm\cup\dl}^{(k+1),\into}g_{\gm\cap\dl}^{(k+1),\into}-t^{\Weight(\gm\cup\dl)}g_{\gm\cup\dl}^{(k+1),\out}t^{\Weight(\gm\cap\dl)}g_{\gm\cap\dl}^{(k+1),\out}\right),
\end{align*}
which reduces by $\widetilde{g}_{\gm\cup\dl}$ and then $\ztau\ksup{k+1}g_{\gm\cap\dl}\ksup{k+1}$, as before.
\end{proof}

This completes the analysis of S-polynomials and reductions parallel to Section~\ref{bbrd2}, along with the proof of Lemma~\ref{buchend} for arbitrary framings. Since all of the necessary reductions to zero can be accomplished without expanding the working basis, Buchberger's Algorithm terminates with the working basis $\cG^\prime$ described above. The proof of Theorem~\ref{idealqthmintro} from Lemma~\ref{buchend} is identical to that in Section~\ref{lemmatothmsec}. Intersecting with $\cE_{k+1}$, then dividing each of the remaining generators by $\ztau\ksup{k+1}$ we produce the set of $g_\gm\ksup{k+1}$ as a basis for the ideal quotient $\pi_k\!\left(N_k\right):\left(\ztau\ksup{k+1}\right)$. Therefore, we have established Theorem~\ref{idealqthmintro} in the full generality that it was stated.

\bibliographystyle{amsplain}
\bibliography{/Users/allisongilmore/Dropbox/Math/writing/universalmathbib}
\end{document}